\documentclass[11pt]{amsart}

\usepackage{latexsym,amsfonts,amssymb,exscale,enumerate,comment}
\usepackage{amsmath,amsthm,amscd}
\usepackage[all,knot]{xy}

\usepackage{fullpage, braket}

\usepackage{caption}
\usepackage{subcaption}

\usepackage{url}
\usepackage[bookmarks=true,%
    colorlinks=true,%
    linkcolor=blue,%
    citecolor=blue,%
    filecolor=blue,%
    menucolor=blue,%
    urlcolor=blue,%
    breaklinks=true]{hyperref}

\usepackage[normalem]{ulem}

\input xy
\usepackage[all]{xy}
\SelectTips{cm}{}

\usepackage{tikz}
\usepackage{tikz-cd}
\usetikzlibrary{shapes}
\usetikzlibrary{decorations.markings}
\usetikzlibrary{decorations.pathreplacing}
\tikzstyle directed=[postaction={decorate,decoration={markings,
    mark=at position #1 with {\arrow{>}}}}]
\tikzstyle rdirected=[postaction={decorate,decoration={markings,
    mark=at position #1 with {\arrow{<}}}}]

\newcommand{\hackcenter}[1]{
 \xy (0,0)*{#1}; \endxy}

\tikzset{->-/.style={decoration={
  markings,
  mark=at position #1 with {\arrow{>}}},postaction={decorate}}}

\usepackage{graphicx}
\usepackage{color}
\newcommand{\brk}[1]{{\left\langle{#1}\right\rangle}}

\newcommand{\md}{\operatorname{\mathsf{d}}}
\newcommand{\mt}{\operatorname{\mathsf{t}}}

\newcommand{\id}{{\rm id}}

\newcommand{\cat}{\mathcal{C}}

\theoremstyle{plain}
\newtheorem{theorem}{Theorem}

\newtheorem*{theo}{Theorem}
\newtheorem{corollary}[theorem]{Corollary}
\newtheorem{proposition}[theorem]{Proposition}
\newtheorem{lemma}[theorem]{Lemma}

\newtheorem{notation}[theorem]{Notation}

\theoremstyle{definition}

\newtheorem{definition}[theorem]{Definition}

\theoremstyle{definition}
\newtheorem{remark}[theorem]{Remark}

\newcommand{\maps}{\colon}

\newcommand{\refequal}[1]{\xy {\ar@{=}^{#1}
(-1,0)*{};(1,0)*{}};
\endxy}

\newcommand{\Hom}{{\rm Hom}}

\renewcommand{\to}{\rightarrow}

\newcommand{\Br}{{\rm Br}}

\def\Id{\mathrm{Id}}

\def\Br{{\mathrm{Br}}}

\numberwithin{equation}{section}

\newcommand{\wb}{\overline}

\newcommand{\slt}{{\mathfrak{sl}(2)}}
\newcommand{\Uq}{{U_q\slt}}
\newcommand{\UqMed}{{\wb U_q\slt}}

\newcommand{\UsltH}{{U_{q}^{H}\slt}}
\newcommand{\Ubar}{{\wb U_q^{H}\slt}}

\let\tilde=\widetilde

\let\epsilon=\varepsilon

\usepackage{bbm}
\def\C{{\mathbb{C}}}
\def\N{{\mathbbm N}}

\def\Z{{\mathbbm Z}}

\def\1{\mathbbm{1}}%
\usepackage{bbm}

\newcommand\nc{\newcommand}
\nc\rnc{\renewcommand}
\nc\Kar{\operatorname{Kar}}
\nc\End{\operatorname{End}}

\newcommand{\scs}{\scriptstyle}

\nc\Sym{\operatorname{Sym}}

\allowdisplaybreaks

\newcommand{\coev}{\stackrel{\longrightarrow}{\operatorname{coev}}}
\newcommand{\ev}{\stackrel{\longrightarrow}{\operatorname{ev}}}
\newcommand{\tev}{\stackrel{\longleftarrow}{\operatorname{ev}}}
\newcommand{\tcoev}{\stackrel{\longleftarrow}{\operatorname{coev}}}

\newcommand{\qr}{{q}}

\newcommand{\qn}[1]{{\left\{#1\right\}}}
\newcommand{\qN}[1]{{\left[#1\right]}}

\newcommand{\qdim}{\operatorname{qdim}}

\newcommand{\et}{{\quad\text{and}\quad}}

\title{Fusion Trees and Homological Representations}
\author[S. Kim]{Sung Kim}
\address{Department of Mathematics\\
 University of Southern Californa \\
  Los Angeles, California 90089, USA}
  \email{skim2261@usc.edu}

\begin{document}

\begin{abstract}
We establish an identification between the spaces of $\alpha$-fusion trees in non-semisimple topological quantum computation (NSS TQC) and a family of homological representations of the braid group known as the Lawrence representations specialized at roots of unity. Leveraging this connection, we provide a new proof of Ito's colored Alexander invariant formula using graphical calculus. Inspired by Anghel's topological model, we derive a formula involving the Hermitian pairing of fusion trees. This formula verifies that non-semisimple quantum knot invariants can be explicitly encoded via the language of fusion trees in the NSS TQC mathematical architecture.
\end{abstract}

\maketitle

\section{Introduction}
Topological quantum field theory (TQFT) is a rich subject that appears in various prominent physical systems, such as the fractional quantum Hall effect. Modular categories (MCs) encode the algebraic description of TQFTs, and there is a well-known correspondence between modular fusion categories and $(2+1)$-dimensional TQFTs \cite{BK}. Freedman, Kitaev, Larsen, and Wang demonstrated that these categories provide the algebraic theory to explore topological order with strong potential applications to fault-tolerant quantum computation \cite{FKLW}. 

\subsection{Unitary TQFTs and Topological Phases of Matter} One of the most influential physical theories, the Chern-Simons-Witten theory, appears in mathematical physics as an equivalent notion of the Witten-Reshetikhin-Turaev TQFT coming from the quantum group $\mathfrak{sl}(2)$ at a root of unity. When unitary, these theories govern the mathematics of $(2+1)$-dimensional topological phases with rich properties that can host quasiparticle excitations called \textit{anyons}. These quasiparticles possess intriguing properties when they swap positions with one another; the spacetime graph of anyons carve out shapes that look like knots and links. The algebraic theory of anyons are encoded in a mathematical framework known as MCs. In particular, an MC encodes data used to compute quantum link invariants that correspond to the braiding statistics of an anyon that only depends on the topology of the knot or link.

The exotic properties of topological phases of matters and their anyons are the foundation for fault-tolerant approaches to quantum computation via topological quantum computation (TQC). In such framework, quantum information are encoded in the topologies of the system, quantum states are encoded as \emph{fusion trees} in Hom-spaces (see Figure~\ref{fig:fusion-trees}), and quantum computation via unitary gates are performed by braiding anyons. More information on this architecture can be found in \cite{FKLW, SimonBook, TQC}.

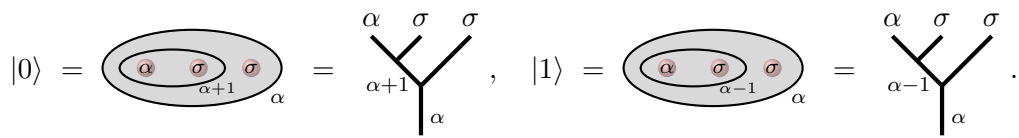
\begin{figure}[htp!]  
    \centering
$
 \ket{0}~=~\hackcenter{\begin{tikzpicture}[scale=0.35]
          \draw[thick, fill=gray!30] (1.75,0,0) ellipse (3.4 and 1.45);
        \draw[thick,   fill=gray!30] (1,0,0) ellipse (2 and .7);
    \shade[ball color = red!40, opacity = 0.5] (0,0,0) circle (.35);
    \shade[ball color = red!40, opacity = 0.5] (2,0,0) circle (.35);
    \shade[ball color = red!40, opacity = 0.5] (4,0,0) circle (.35);
    \node at (0,0) {$\scriptstyle \alpha$};
    \node at (2,0) {$\scriptstyle \sigma$};
    \node at (4,0) {$\scriptstyle \sigma$};
            \node at (2.7,-.8) {$\scriptscriptstyle \alpha+1$};
            \node at (5,-1.15) {$\scriptstyle \alpha$};
\end{tikzpicture}}~=~
\hackcenter{\begin{tikzpicture}[ scale=1.1]
  \draw[ultra thick, black] (0,0) to (.6,-.6) to (.6,-1.2);
  \draw[ultra thick, black] (0.6,0) to (.3,-.3);
  \draw[ultra thick, black] (1.2,0) to (.6,-.6); 
   \node at (0,0.2) {$\alpha$};
   \node at (.6,0.2) {$\sigma$}; 
   \node at (1.2,0.2) {$\sigma$};  
    \node at (.8,-1) {$\scriptstyle \alpha$};
    \node at (.2,-.6) {$\scriptstyle \alpha + 1$}; 
\end{tikzpicture} },
\quad 
\ket{1}~=~\hackcenter{\begin{tikzpicture}[scale=0.35]
          \draw[thick, fill=gray!30] (1.75,0,0) ellipse (3.4 and 1.45);
        \draw[thick,   fill=gray!30] (1,0,0) ellipse (2 and .7);
    \shade[ball color = red!40, opacity = 0.5] (0,0,0) circle (.35);
    \shade[ball color = red!40, opacity = 0.5] (2,0,0) circle (.35);
    \shade[ball color = red!40, opacity = 0.5] (4,0,0) circle (.35);
    \node at (0,0) {$\scriptstyle \alpha$};
    \node at (2,0) {$\scriptstyle \sigma$};
    \node at (4,0) {$\scriptstyle \sigma$};
            \node at (2.7,-.8) {$\scriptscriptstyle \alpha-1$};
            \node at (5,-1.15) {$\scriptstyle \alpha$};
\end{tikzpicture}}~=~
\hackcenter{\begin{tikzpicture}[ scale=1.1]
  \draw[ultra thick, black] (0,0) to (.6,-.6) to (.6,-1.2);
  \draw[ultra thick, black] (0.6,0) to (.3,-.3);
  \draw[ultra thick, black] (1.2,0) to (.6,-.6); 
   \node at (0,0.2) {$\alpha$};
   \node at (.6,0.2) {$\sigma$}; 
   \node at (1.2,0.2) {$\sigma$};  
    \node at (.8,-1) {$\scriptstyle \alpha$};
    \node at (.2,-.6) {$\scriptstyle \alpha - 1$}; 
\end{tikzpicture} }.
$
  \caption{An example of quantum states encoded as fusion trees (see~\cite{neglecton, NSSTQC}).}  \label{fig:fusion-trees}  
\end{figure}

\subsection{The Non-semisimple Paradigm} The fundamental examples of semisimple MCs come from the representation theory of quantum groups with a parameter set at a root of unity. These representation categories are generally non-semisimple: we perform a semisimplification procedure to eliminate an infinite number of non-isomorphic representations with zero quantum dimension. Then, our resulting category has a semisimple MC structure with a finite collection of simple objects. Modern machinery made this procedure necessary for a long time because the standard techniques for building a TQFT from an MC, e.g. the Witten-Reshetikhin-Turaev construction, do not permit inputs of representations of vanishing quantum dimensions. 

How does one build a robust MC theory that retains the non-semisimple information? The works of Blanchet-Costantino-Geer-Patureau (BCGP) and their many collaborators extended our understanding of modular categories and their TQFTs in the semisimple setting to the non-semisimple setting~\cite{BCGP2, CGP14}. There are many advantages to working on the non-semisimple case over their semisimple counterpart: in the context of quantum computation, the permission to keep the infinite number of representations with vanishing dimensions gives greater flexibility in creating state spaces. Furthermore, non-semisimple MCs lead to non-semisimple TQFTs. Two key features of non-semisimple TQFTs that differs from semisimple TQFTs is its ability to distinguish lens spaces \cite{BCGP2} and gives rise to mapping class group representations where the Dehn twist has infinite order \cite{mappingclassrep}. Due to this, the non-semisimple theory suggests far more powerful topological invariants which is a crucial feature for topological orders.

\subsection{Non-semisimple TQFTs and TQC}
Given the clear advantages that the non-semisimple paradigm offers, it is natural to investigate \emph{non-semisimple topological quantum computation} (NSS TQC). To explore these problems, we require Hermitian structures on our category. These non-semisimple Hermitian structures were first developed in~\cite{GLPMS} in the context of the unrolled quantum $\slt$ and were further generalized in~\cite{NSS-Hermitian}. More importantly, it was shown in~\cite{NSS-LV, NSS-Hermitian} that these structures give rise to non-semisimple analogues of the Levin-Wen string-net models.

A fundamental question in TQC asks which anyonic frameworks, i.e. MCs, admit universal quantum computation via braiding alone. A primary strategy to approach this question is to first ask: when does our braid group representation, acting on fusion trees, have a dense image? In~\cite{NSS-Burau}, the authors demonstrated at the fourth root of unity that the braid group acting on fusion trees that involve $V_\alpha$ modules\footnote{These $V_\alpha$ representations recently acquired the name \emph{neglectons} in~\cite{NSSTQC}.}, which have vanishing quantum dimensions, have a dense image. We mention two primary takeaways of their article:
\begin{enumerate}
    \item Non-semisimple categories may provide far more powerful frameworks for TQC compared to their semisimple counterparts via incorporating modules that traditionally get killed in semisimplification.
    \item Their key strategy to prove density came from the perspective of the (reduced) Burau representation, which belongs to a larger family of homological braid group representations called the \emph{Lawrence representations}.
\end{enumerate}
Takeaway (1) was successfully explored in~\cite{neglecton, NSSTQC} where the author and collaborators showed that the non-semisimple Ising model is universal for quantum computation via braiding alone. This demonstrates that non-semisimple categories are effective models for quantum computation; recall that the standard theory of Ising anyons, which is believed to characterize excitations in the $\nu=5/2$ fractional quantum Hall state, is not universal for quantum computation via braiding alone.

In this article, we explore takeaway (2) to acquire a new perspective on a family of fusion trees, called \emph{$\alpha$-fusion trees} (see Figure~\ref{fig:trees}), through the lens of homological representations. Unlike~\cite{neglecton, NSSTQC}, we investigate braid group representations acting on a series of $V_\alpha$ rather than the affine braid group with one $V_\alpha$. We emphasize the important role of homological representations and hope to inspire new homological tools to dissect representations of the braid group that appear in NSS TQC.

\begin{figure}
  \hfill\subcaptionbox{An $\alpha$-fusion tree \label{fig:trees(a)}}[12em]{\centering \hackcenter{\begin{tikzpicture}[scale=1]
\begin{scope}[decoration={markings, mark=at position 0.5 with {\arrow{>}}}]
  \draw[ultra thick, black, postaction=decorate] (0,0) to (0.6,-0.6);
  \draw[ultra thick, black, postaction=decorate] (0.6,-0.6) to (1.2,-1.2);
  \draw[ultra thick, dotted] (1.2,-1.2) to (1.8,-1.8);
  \draw[ultra thick, black, postaction=decorate] (1.8,-1.8) to (2.4,-2.4);
  \draw[ultra thick, black, postaction=decorate] (1.2,0) to (.6,-.6);
  \draw[ultra thick, black, postaction=decorate] (2.4,0) to (1.2,-1.2);
  \draw[ultra thick, black, postaction=decorate] (3.6,0) to (1.8,-1.8);
  \draw[ultra thick, black, postaction=decorate] (4.8,0) to (2.4,-2.4);
  \draw[ultra thick, black, postaction=decorate] (2.4,-2.4) to (2.4,-3.3);
\end{scope}
   \node at (0,0.2) {$\alpha$};
   \node at (1.2,0.2) {$\alpha$};
   \node at (2.4,0.2) {$\alpha$};
   \node at (3,0.2) {$\cdots$};
   \node at (3.6,0.2) {$\alpha$};
   \node at (4.8,0.2) {$\alpha$};
   \node at (.7,-1.1) {$\gamma_1$};
   \node at (1.75,-2.3) {$\gamma_{n-2}$};
   \node at (2.4,-3.4) {$\gamma_{n-1}$};
\end{tikzpicture} } }
  \hfill\subcaptionbox{A special $\alpha$-fusion tree\label{fig:trees(b)}}{\centering \hackcenter{\begin{tikzpicture}[baseline=0, thick, scale=0.5, shift={(0,0)}]
\begin{scope}[shift={(1,-5.5)}]
\begin{scope}[decoration={markings, mark=at position 0.6 with {\arrow{>}}}]
    \draw[ultra thick, black, postaction=decorate] (0,4.5) to (0.8125,3.75);
    \draw[ultra thick, black, postaction=decorate] (1,4.5) to (0.8125,3.75);
    \draw[ultra thick, black, postaction=decorate] (0.8125,3.75) to (1.625,3);
    \draw[ultra thick, black, postaction=decorate] (2,4.5) to (1.625,3);
    \draw[ultra thick, black, dotted] (1.625, 3) to (2.4375, 2.25);
    \draw[ultra thick, black, postaction=decorate] (3, 4.5) to (2.4375, 2.25);
    \draw[ultra thick, black, postaction=decorate] (2.4375,2.25) to (3.25,1.5);
    \draw[ultra thick, black, postaction=decorate] (4, 4.5) to (3.25,1.5);
    \draw[ultra thick, black, postaction=decorate] (5, 4.5) to (4.0625,0.75);
    \draw[ultra thick, black, postaction=decorate] (3.25,1.5) to (4.0625,0.75);
    \draw[ultra thick, black, postaction=decorate] (6, 4.5) to (4.875,0);
    \draw[ultra thick, black, postaction=decorate] (4.0625,0.75) to (4.875,0);
    \draw[ultra thick, black, postaction=decorate] (7, 4.5) to (5.6875,-0.75);
    \draw[ultra thick, black, dotted] (4.875,0) to (5.6875,-0.75);
    \draw[ultra thick, black, postaction=decorate] (8, 4.5) to (6.5,-1.5);
    \draw[ultra thick, black, postaction=decorate] (5.6875,-0.75) to (6.5,-1.5);
    \draw[ultra thick, black, postaction=decorate] (6.5,-1.5) to (6.5,-2.5);
\end{scope}
\end{scope}
\draw (1,-0.8) node {$\scs \alpha$};
\draw (2,-0.8) node {$\scs \alpha$};
\draw (3,-0.8) node {$\scs \alpha$};
\draw (3.5,-1.25) node {$\scs \cdots$};
\draw (4,-0.8) node {$\scs \alpha$};
\draw (5,-0.8) node {$\scs \alpha$};
\draw (6,-0.8) node {$\scs -\alpha$};
\draw (7,-0.8) node {$\scs -\alpha$};
\draw (7.5,-1.25) node {$\scs \cdots$};
\draw (8,-0.8) node {$\scs -\alpha$};
\draw (9,-0.8) node {$\scs -\alpha$};
\draw (2.2,-2.5) node {$\scs \gamma_1$};
\draw (3.3,-4) node {$\scs \gamma_{n-2}$};
\draw (4.25,-4.75) node {$\scs \gamma_{n-1}$};
\draw (5.2,-5.6) node {$\scs \gamma_{n-2}$};
\draw (6.5,-7) node {$\scs \gamma_{1}$};
\draw (7.5,-8.2) node {$\scs \alpha$};
\end{tikzpicture} }}
  \hfill\null
  \caption{Fusion trees involved in this study.}\label{fig:trees}
\end{figure}
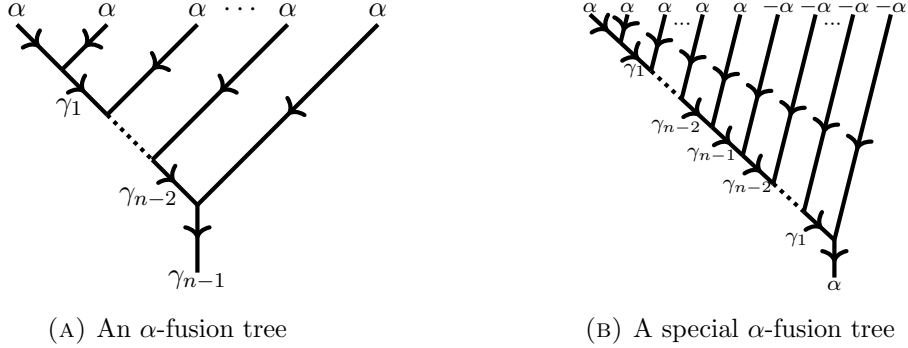

\subsection{Homological Representations}
Before investigating representations of the braid group, we first examine the braid group itself. There are two common flavors of the braid group: an algebraic flavor 
\begin{equation*}
\Br_n^{\sigma}:=\left\langle \sigma_1, \ldots, \sigma_{n-1} \mid \;
 \begin{array}{ll}
  \sigma_i \sigma_{i+1} \sigma_i =  \sigma_{i+1} \sigma_i \sigma_{i+1}  ,  \; &i=1,\ldots, n-2,   \\
\sigma_i\sigma_j = \sigma_j\sigma_i, &|i-j|>2.
\end{array}
  \right\rangle 
\end{equation*}
defined as a group presentation, using the Artin generators satisfying the braid relations, and a geometric flavor 
\begin{equation*}
\Br_n^{\Sigma}:= \mathrm{MCG}(\Sigma_n)
\end{equation*}
defined as the mapping class group of the $n$-punctured disc $\Sigma_n$.

Representations of the braid group constructed from MCs arising from representations of a quantum group, which we call \textit{quantum representations}, ubiquitously use the algebraic flavor of the braid group $\mathrm{Br}_n^{\sigma}$. Using these quantum representations, one can construct quantum link invariants, which we emphasize are \textit{topological} invariants, from a purely algebraic construction. Therefore, these quantum invariants leaves the user with an interesting mystery on what topological information they capture. Despite this weakness of topological interpretability, the combinatorial nature of $\mathrm{Br}_n^\sigma$ promotes a major aspect of algorithmic computability.

Braid group representations built from homology theories, which we call \textit{homological representations}, takes advantage of the geometric landscape provided by the mapping class group definition of the braid group $\Br_n^{\Sigma}$. These homological representations provide a wealth of geometric intuition and topological insights. However, relative to quantum representations of the braid group, algorithmic computations are not always readily available. These two types of braid group representations unveil a clear tension between features of algorithmic computability and topological interpretability. Moreover, one should expect there to be a way to translate the information provided by a quantum (resp. homological) representation $\Br_n^{\sigma} \rightarrow \mathrm{End}(V)$ to a homological (resp. quantum) representation $\Br_n^{\Sigma} \rightarrow \mathrm{End}(V')$ to gain more topological insights (resp. computational advantages) with the potential expense of sacrificing some versatility of the other feature.

\subsection{Fusion Trees and Homological Representations}
The history between homological and quantum representations is well-recorded in~\cite{Homological-Martel}. As elaborated above, fusion trees are a type of quantum representations that encode quantum states and thus vital for the mathematics of TQC. Consequently, one should expect an explicit translation between fusion trees and homological representations. We prove the following statement.
\begin{theo}[Theorem~\ref{thm:fusiontree-homological}]
The braid group representation on $\mathcal{H}_{n,m,\alpha}^r$ (the $\Hom$-space spanned by $\alpha$-fusion trees; see Section~\ref{sec:fusion-trees}) is projectively isomorphic to the Lawrence representation specialized at a root of unity $\mathcal{L}_{n,m,\alpha}^{r}$.
\end{theo}

The original quantum interpretation of the (truncated) Lawrence representations comes from certain highest weight subspaces of tensor powers of the Verma modules associated to quantum $\slt$. Our new quantum reinterpretation of these representations as fusion trees equips us with the powerful technique of \emph{graphical calculus}. Leveraging this perspective, we provide a new proof of Ito's homological formula for the colored Alexander invariants using graphical calculus.

\begin{theo}[Theorem~\ref{thm:Ito}, Ito's formula~\cite{Ito16}]
Let $K$ be a knot represented as a closure of an $n$-braid $\beta_n$. Then, the colored Alexander invariant formula is given by
\begin{equation}
N_r^{\alpha}(K)= q^{\frac{(\alpha+r-1)^2}{2}writhe(\beta_n)}\sum\limits_{m=0}^{(n-1)(r-1)}\md(n\alpha+(n-1)(r-1)-2m) \chi_{\mathcal{L}_{n,m,\alpha}^r}(\beta_n)
\end{equation}
where $\chi_{\mathcal{L}_{n,m,\alpha}^r}$ is the character of the Lawrence representation specialized at a root of unity $\mathcal{L}_{n,m,\alpha}^r$ (Definition~\ref{def:l2r}) and $\md(-)$ is the modified dimension (Equation~\eqref{eq:md}). 

In particular,
\begin{equation}
\Delta_K(q^{2\alpha-2})=\frac{q^{(\alpha+1) writhe(\beta_n)}}{\md(\alpha)}\sum\limits_{m=0}^{n-1}\md(n\alpha+n-1-2m) \chi_{\mathcal{L}_{n,m,\alpha}^2}(\beta_n) 
\end{equation}
where $\Delta_K(x)\in \mathbb{Z}[x,x^{-1}]$ is the Alexander polynomial of $K$. 
\end{theo}

From a quantum information processing point of view, we want a formula that involves three primary ingredients: (1) fusion trees to encode quantum states characterizing information, (2) braids to process quantum information, and (3) a Hermitian form to measure and observe our processed information. We now mention that the (truncated) Lawrence representations have a limitation: it is not guaranteed that the topological intersection pairing will remain non-degenerate when specialized at a root of unity~\cite[Section 6]{Anghel23}. To reconcile this, Anghel introduced a new but isomorphic representation called the \emph{special Lawrence representation} where the specialization preserves the non-degeneracy of the intersection form. Using this version, Anghel derived a colored Alexander invariant formula given as the topological intersection pairing.

Inspired by Anghel's topological model, we define \emph{special $\alpha$-fusion trees} (see Figure~\ref{fig:trees}). Utilizing the non-degenerate Hermitian pairing of a Hermitian ribbon category, we derive an analogous version of Anghel's formula in the language of fusion trees useful for quantum computation.

\begin{theo}[Theorem~\ref{thm:special-pairing}]
Let $K$ be a knot represented as a closure of an $n$-braid $\beta_n$. Define
\begin{equation}
\mathcal{Y}:=\sum\limits_{\substack{\gamma_1 \in 2\alpha + H_r \\ \gamma_2 \in \gamma_1 + \alpha + H_r \\ \vdots \\ \gamma_{n-1} \in \gamma_{n-2} + \alpha + H_r \\ \mathbf{p}:=(\alpha,\gamma_1, \dots, \gamma_{n-1})}} \sqrt{\frac{\md(\gamma_{n-1})}{\langle Y^s_{\mathbf{p}}, Y^s_{\mathbf{p}} \rangle}}
\hackcenter{\begin{tikzpicture}[baseline=0, thick, scale=0.5, shift={(0,0)}]
\begin{scope}[shift={(1,-5.5)}]
\begin{scope}[decoration={markings, mark=at position 0.6 with {\arrow{>}}}]
    \draw[ultra thick, black, postaction=decorate] (0,4.5) to (0.8125,3.75);
    \draw[ultra thick, black, postaction=decorate] (1,4.5) to (0.8125,3.75);
    \draw[ultra thick, black, postaction=decorate] (0.8125,3.75) to (1.625,3);
    \draw[ultra thick, black, postaction=decorate] (2,4.5) to (1.625,3);
    \draw[ultra thick, black, dotted] (1.625, 3) to (2.4375, 2.25);
    \draw[ultra thick, black, postaction=decorate] (3, 4.5) to (2.4375, 2.25);
    \draw[ultra thick, black, postaction=decorate] (2.4375,2.25) to (3.25,1.5);
    \draw[ultra thick, black, postaction=decorate] (4, 4.5) to (3.25,1.5);
    \draw[ultra thick, black, postaction=decorate] (5, 4.5) to (4.0625,0.75);
    \draw[ultra thick, black, postaction=decorate] (3.25,1.5) to (4.0625,0.75);
    \draw[ultra thick, black, postaction=decorate] (6, 4.5) to (4.875,0);
    \draw[ultra thick, black, postaction=decorate] (4.0625,0.75) to (4.875,0);
    \draw[ultra thick, black, postaction=decorate] (7, 4.5) to (5.6875,-0.75);
    \draw[ultra thick, black, dotted] (4.875,0) to (5.6875,-0.75);
    \draw[ultra thick, black, postaction=decorate] (8, 4.5) to (6.5,-1.5);
    \draw[ultra thick, black, postaction=decorate] (5.6875,-0.75) to (6.5,-1.5);
    \draw[ultra thick, black, postaction=decorate] (6.5,-1.5) to (6.5,-2.5);
\end{scope}
\end{scope}
\draw (1,-0.8) node {$\scs \alpha$};
\draw (2,-0.8) node {$\scs \alpha$};
\draw (3,-0.8) node {$\scs \alpha$};
\draw (3.5,-1.25) node {$\scs \cdots$};
\draw (4,-0.8) node {$\scs \alpha$};
\draw (5,-0.8) node {$\scs \alpha$};
\draw (6,-0.8) node {$\scs -\alpha$};
\draw (7,-0.8) node {$\scs -\alpha$};
\draw (7.5,-1.25) node {$\scs \cdots$};
\draw (8,-0.8) node {$\scs -\alpha$};
\draw (9,-0.8) node {$\scs -\alpha$};
\draw (2.2,-2.5) node {$\scs \gamma_1$};
\draw (3.3,-4) node {$\scs \gamma_{n-2}$};
\draw (4.25,-4.75) node {$\scs \gamma_{n-1}$};
\draw (5.2,-5.6) node {$\scs \gamma_{n-2}$};
\draw (6.5,-7) node {$\scs \gamma_{1}$};
\draw (7.5,-8.2) node {$\scs \alpha$};
\end{tikzpicture} } 
\end{equation}
where $Y_{\mathbf{p}}^s$ denotes a special $\alpha$-fusion tree defined in Definition~\ref{def:special-fusion-tree}. Then, we have the following formula:
\begin{equation}
N_r^\alpha(K) = \langle \mathcal{Y}, (\beta_n \cup \mathbb{I}_{n-1})\mathcal{Y} \rangle
\end{equation}
where $\mathbb{I}_{n-1}$ is the identity braid word on $(n-1)$-strands.
\end{theo}

\subsection{Organization} In Section 2, we provide a brief survey of the Lawrence representations of the braid group. In Section 3, we review the unrolled quantum $\slt$ at an even root of unity and its category of representations. In Section 4, we study fusion trees coming from the generic part of the category and prove the key result that identifies the space of fusion trees with the Lawrence representations. Leveraging this identification, we prove Ito's formula in the language of graphical calculus. In Section 5, we incorporate Hermitian structures and extend the story of Section 4 that is parallel to Anghel's special Lawrence representations. In particular, we provide a non-semisimple quantum knot invariant formula suitable for the NSS TQC architecture.

\subsection{Acknowledgments} The author would like to thank Cristina Anghel, Christian Blanchet, Filippo Iulianelli, Aaron Lauda, Jules Martel, Martin Palmer, and Joshua Sussan for helpful discussions and advice. S.K. is supported by the NSF Graduate Research Fellowship DGE-1842487 and partially supported by the NSF grant DMS-2200419 and the Simons Foundation Collaboration grant on New Structures in Low-dimensional Topology.

\section{Homological Representations}
In this section, we review the Lawrence representations along with their truncations. 

The original Lawrence representations play a role in recovering invariants at generic $q$, such as the Jones polynomial~\cite{Big,Law}, while their truncations provide explicit basis for when we specialize at a root of unity. While these homological representations record information about braid groups acting on tensor products of finite-dimensional highest weight modules, Martel's extended Lawrence representations correspond to braid group representations on tensor powers of the quantum $\slt$ Verma modules~\cite{Homological-Martel}.

\subsection{Lawrence representations} First, we outline the construction of a homological representation $L_{n,m}$ called the \textit{Lawrence representation}. We refer the reader to~\cite{Lawrence} for the original introduction and~\cite{AnghelPalmer} for the general construction. The Lawrence representations recover the (reduced) Burau representations and the Lawrence-Krammer-Bigelow representations at $m=1$ and $m=2$ respectively.

First, denote the closed disc as $\mathbb{D}^2 =  \{ z \in \mathbb{C} : |z|\leq 1 \}$. Let $n,m \in \mathbb{Z}_{>0}$ and consider the $n$-punctured disc 
\begin{equation*}
\Sigma_n = \mathbb{D}^2\setminus \{ p_1, \dots, p_n\}
\end{equation*}
where $p_1, \dots, p_n$ are $n$ distinct points in the interior of $\mathbb{D}^2$. Without loss of generality, one may place each puncture on the real line so that $-1 < p_1 < \dots < p_n < 1$.

The \textit{ordered configuration space} of $m$ points on the $n$-punctured disc is defined as $\Sigma_n^{\times m} \setminus \Delta$ where $\Delta = \{ (x_1, \dots, x_m) \in \Sigma_n^{\times m} : x_i=x_j \text{ for some } i \neq j \}$. Then, the \textit{unordered configuration space} of $m$ points on the $n$-punctured disc is defined as 
\begin{equation*}C_{n,m} = (\Sigma_n^{\times m} \setminus \Delta)/\mathfrak{S}_m
\end{equation*}where $\mathfrak{S}_m$ is the symmetric group acting on the permutation of the indices.

We fix a basepoint $d=(d_1, \dots, d_m) \in C_{n,m}$ such that $d_1, \dots, d_m \in \partial\Sigma_n$. The Hurewicz map $\rho: \pi_1(C_{n,m}) \rightarrow H_1(C_{n,m})$ tells us that, for $m\geq 2$, the first homology group of the configuration space is $H_1(C_{n,m}) \cong \langle \rho(\sigma_1), \dots, \rho(\sigma_{n})\rangle \oplus \langle \rho(\delta)\rangle \cong \mathbb{Z}^n \oplus \mathbb{Z}$. Here, $\sigma_i \in \pi_1(C_{n,m})$ is represented by the loop in the configuration space with $(m-1)$ fixed components and the first one going on a loop in $\Sigma_n$ around puncture $p_i$ and $\delta \in \pi_1(C_{n,m})$ is represented by a loop in the configuration space given by $(m-2)$ constant points and the first two components making a circle, which swaps the two initial points (see Figure~\ref{fig:fundamental-group}).

\begin{figure}[htp!]  
    \centering
\begin{tikzpicture}[baseline=0, thick, scale=0.75, shift={(0,0)}]
    \draw (0,0) circle(3);
    \draw (-2.5,0) node {$\scs \bullet$};
    \draw (-1.5,0) node {$\scs \bullet$};
    \draw (-0.75,0) node {$\scs \cdots$};
    \draw (0,0) node {$\scs \bullet$};
    \draw (1,0) node {$\scs \cdots$};
    \draw (2,0) node {$\scs \bullet$};
    \draw (-2.5,0) node[above] {$\scs p_1$};
    \draw (-1.5,0) node[above] {$\scs p_2$};
    \draw (0,0) node[above] {$\scs p_i$};
    \draw (2,0) node[above] {$\scs p_n$};
    \draw ({3*cos(220)}, {3*sin(220)}) node {$\scs \bullet$};
    \draw ({3*cos(240)}, {3*sin(240)}) node {$\scs \bullet$};
    \draw ({3*cos(300)}, {3*sin(300)}) node {$\scs \bullet$};
    \draw (0,-2.9) node {$\scs \cdots$};
    \draw ({3*cos(220)}, {3*sin(220)}) node[left] {$\scs d_1$};
    \draw ({3*cos(240)}, {3*sin(240)}) node[left] {$\scs d_2$};
    \draw ({3*cos(300)}, {3*sin(300)}) node[right] {$\scs d_m$};
    \begin{scope}[decoration={markings, mark=at position 0.75 with {\arrow{>}}}]
        \draw[postaction=decorate, thin] ({3*cos(220)}, {3*sin(220)}) to [out=0, in=-90] (0.5,0);
    \end{scope}
    \draw[thin] (0.5,0) to [out=90, in=0] (0,1);
    \draw[thin] (0,1) to [out=180, in=90] (-0.5,0);
    \draw[thin] ({3*cos(220)}, {3*sin(220)}) to [out=0, in=-90] (-0.5,0);
    \draw (0,-1.25) node[below] {$\scs \sigma_i$};
\end{tikzpicture}
\qquad \qquad \qquad \qquad
\begin{tikzpicture}[baseline=0, thick, scale=0.75, shift={(0,0)}]
    \draw (0,0) circle(3);
    \draw (-2.5,0) node {$\scs \bullet$};
    \draw (-1.5,0) node {$\scs \bullet$};
    \draw (-0.75,0) node {$\scs \cdots$};
    \draw (0,0) node {$\scs \bullet$};
    \draw (1,0) node {$\scs \cdots$};
    \draw (2,0) node {$\scs \bullet$};
    \draw (-2.5,0) node[above] {$\scs p_1$};
    \draw (-1.5,0) node[above] {$\scs p_2$};
    \draw (0,0) node[above] {$\scs p_i$};
    \draw (2,0) node[above] {$\scs p_n$};
    \draw ({3*cos(220)}, {3*sin(220)}) node {$\scs \bullet$};
    \draw ({3*cos(240)}, {3*sin(240)}) node {$\scs \bullet$};
    \draw ({3*cos(300)}, {3*sin(300)}) node {$\scs \bullet$};
    \draw (0,-2.9) node {$\scs \cdots$};
    \draw ({3*cos(220)}, {3*sin(220)}) node[left] {$\scs d_1$};
    \draw ({3*cos(240)}, {3*sin(240)}) node[left] {$\scs d_2$};
    \draw ({3*cos(300)}, {3*sin(300)}) node[right] {$\scs d_m$};
    \begin{scope}[decoration={markings, mark=at position 0.75 with {\arrow{>}}}]
        \draw[postaction=decorate,thin] ({3*cos(220)}, {3*sin(220)}) to [out=0, in=90] (-3/2, {3*sin(240)});
        \draw[thin] ({3*cos(220)}, {3*sin(220)}) to [out=45, in=180] (-0.5, -1);
        \draw[postaction=decorate,thin] ({3*cos(240)}, {3*sin(240)}) to [out=45, in=0] (-0.5, -1);
    \end{scope}
    \draw (0,-1.25) node[below] {$\scs \delta$};
\end{tikzpicture}
  \caption{Visual depiction of the loops $\sigma_i$ and $\delta$ in $\pi_1(C_{n,m})$.}  \label{fig:fundamental-group}
\end{figure}
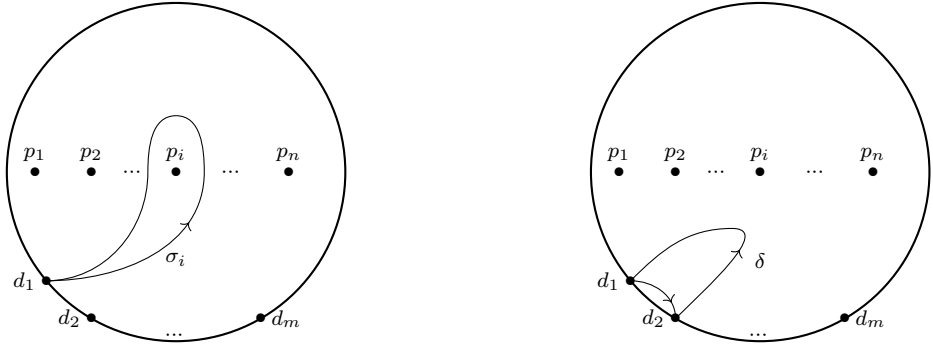

Consider the function $\epsilon: H_1(C_{n,m}) \cong\mathbb{Z}^n \oplus \mathbb{Z} \rightarrow \mathbb{Z} \oplus \mathbb{Z}$ that sends $(\mathtt{x}_1, \dots, \mathtt{x}_n,\mathtt{q}) \mapsto (\mathtt{x}_1 + \dots + \mathtt{x}_n, \mathtt{q})$. Now, define the local system $\phi$ as the composition of these two maps 
\begin{equation*}\phi: \pi_1(C_{n,m}) \xrightarrow{\rho} H_1(C_{n,m})\cong \mathbb{Z}^n \oplus \mathbb{Z} \xrightarrow{\epsilon} \mathbb{Z} \oplus \mathbb{Z}=\langle \mathtt{x} \rangle \oplus \langle \mathtt{q} \rangle.
\end{equation*}
Let $\pi: \tilde{C}_{n,m} \rightarrow C_{n,m}$ be the covering associated to $\ker \phi$. By construction, the group of deck transformations on this covering is $\operatorname{Deck}(\tilde{C}_{n,m})\cong \pi_1(C_{n,m})/\ker \phi \cong \langle \mathtt{x} \rangle \oplus \langle \mathtt{q} \rangle $. For the rest of the construction, we fix a basepoint $\tilde{d}\in \pi^{-1}(d)$ in $\tilde{C}_{n,m}$.

Recall that the braid group is the mapping class group of the punctured disc relative to its boundary $\Br_n^\Sigma = \operatorname{MCG}(\Sigma_n)$. Consequently, we have a braid group action on the configuration space $C_{n,m}$ which further induces an action on the fundamental group $\pi_1 (C_{n,m})$. In particular, one may observe that the braid group action is invariant under $\ker \phi$.

\begin{proposition}{\cite{ItoG, KT, Lawrence}}
The local system $\phi$ is invariant under the $\Br_n$-action and it can be lifted to an action onto the homology of the covering. Moreover, this action is compatible with the action of the deck transformations, yielding a homological representation of the braid group:
\begin{equation*}
\Br_n \rightarrow GL(H_m^{lf}(\tilde{C}_{n,m}, \mathbb{Z}))
\end{equation*}
where the homology of locally-finite chains $H_m^{lf}(\tilde{C}_{n,m}, \mathbb{Z})$ is a $\mathbb{Z}[\mathtt{x}^{\pm 1}, \mathtt{q}^{\pm 1}]$-module and the braid group acts as a $\mathbb{Z}[\mathtt{x}^{\pm 1}, \mathtt{q}^{\pm 1}]$-module automorphism.
\end{proposition}

\begin{remark}
The Borel-Moore homology $H_\bullet^{BM}(\tilde{C}_{n,m}, \mathbb{Z})$ and the homology of locally finite chains $H_\bullet^{lf}(\tilde{C}_{n,m},\mathbb{Z})$ on the covering space $\tilde{C}_{n,m}$ are isomorphic. We refer the reader to the Appendix of~\cite{Homological-Martel} for facts and properties. 
\end{remark}

\begin{remark}
The case when $m=1$ have a homologous construction, except we do not have the loop $\delta$ in $\pi_1(C_{n,1})$ which eradicates the $\mathtt{q}$ variable in $\mathbb{Z}[\mathtt{x}^{\pm 1}, \mathtt{q}^{\pm 1}]$.
\end{remark}

The Lawrence representation is defined as a $\mathbb{Z}[\mathtt{x}^{\pm 1}, \mathtt{q}^{\pm 1}]$-submodule $\mathcal{H}_{n,m} \subset H^{lf}_{m}(\tilde{C}_{n,m}, \mathbb{Z})$ generated by certain homology classes represented by geometric objects called \textit{multiforks}~\cite{Anghel23, Ito16}.

The subspace $\mathcal{H}_{n,m}$ has a natural basis which we briefly introduce here. First, define an indexing set
\begin{equation*}
E_{n,m}=\{ \mathbf{e}=(e_1, \dots, e_{n-1})\in \mathbb{Z}_{\geq 0}^{n-1} \mid e_1+\dots +e_{n-1} = m \}.
\end{equation*}
Denote the cardinality of this indexing set as 
\begin{equation*}
d_{n,m}=|E_{n,m}|=\binom{n+m-2}{m}.
\end{equation*}
Then, for each $\mathbf{e} \in E_{n,m}$, we can assign a multifork $\mathbb{F}_\mathbf{e} = \{ F_{1}, \dots, F_{m}\}:\Sigma_n^{\times m} \rightarrow C_{n,m}$ which we call the \emph{standard multifork} or the \emph{multifork associated to the partition $\mathbf{e}$}. Each standard multifork $\mathbb{F}_{\mathbf{e}}$ is associated to a homology class $[\tilde{\mathbb{F}}_{\mathbf{e}}] \in \mathcal{H}_{n,m}$. By abuse of notation, we use the notation $\mathbb{F}_\mathbf{e}$ to represent both the multifork and the homology class $[\tilde{\mathbb{F}}_{\mathbf{e}}]$. The set of all homology classes associated to standard multiforks $\{ \mathbb{F}_{\mathbf{e}} \}_{\mathbf{e} \in E_{n,m}}$ forms a free basis on $\mathcal{H}_{n,m}$.

\begin{proposition}{\cite{ItoG, KT, Lawrence}}
The subspace $\mathcal{H}_{n,m}$ is invariant under the action of the braid group and the basis $\{ \mathbb{F}_{\mathbf{e}} \}_{\mathbf{e} \in E_{n,m}}$ leads to a homological representation of the braid group
\begin{equation*}
L_{n,m}: \Br_n \rightarrow GL(\mathcal{H}_{n,m})=GL\left(d_{n,m};\mathbb{Z}[\mathtt{x}^{\pm 1},\mathtt{q}^{\pm 1}]\right)
\end{equation*}
called the \emph{Lawrence representation}.
\end{proposition}

\subsection{Truncated Lawrence representations} 
First, consider the subspace $\mathcal{H}_{n,m}^{\geq r}$ of $\mathcal{H}_{n,m}$ spanned by $\{ \mathbb{F}_{\mathbf{e}} \}_{\mathbf{e} \in E_{n,m}^{\geq r}}$ where 
\begin{equation*}
E_{n,m}^{\geq r} = \{ \mathbf{e}=(e_1, \dots, e_{n-1})\in E_{n,m} \mid e_i \geq r \text{ for some } i \} \subset E_{n,m}.
\end{equation*}
Next, define $\mathcal{H}_{n,m}^r = \mathcal{H}_{n,m}/\mathcal{H}_{n,m}^{\geq r}$. By abuse of notation again, we use $\mathbb{F}_{\mathbf{e}}$ to denote a standard multifork $\pi(\mathbb{F}_{\mathbf{e}})$ under the quotient $\pi: \mathcal{H}_{n,m} \rightarrow \mathcal{H}_{n,m}^r$. We define another indexing set
\begin{equation*}
E_{n,m}^{r} = \{ \mathbf{e}=(e_1, \dots, e_{n-1})\in E_{n,m} \mid e_i \leq r - 1 \text{ for all } i \} \subset E_{n,m}
\end{equation*}
and denote $d_{n,m}^r=|E_{n,m}^r|$.

\begin{proposition}{\cite[Proposition-Definition~3.2]{Ito16}}
When $\mathtt{q}$ specialized to $-e^{\frac{2\pi \sqrt{-1}}{r}}$, the space $\mathcal{H}_{n,m}^{\geq r}$ is a $\Br_n$-invariant subspace of $\mathcal{H}_{n,m}$. Thus, we have the following braid group representation on the quotient
\begin{equation*}
l_{n,m}: \Br_n \rightarrow GL(\mathcal{H}_{n,m}^{r}) = GL(d_{n,m}^r ; \mathbb{Z}[\mathtt{x}^{\pm 1}])
\end{equation*}
called the \emph{truncated Lawrence representation}. 
\end{proposition}

\begin{remark}
Ito~\cite{Ito16} introduced the truncated Lawrence representations as a means to derive a formula of the colored Alexander invariant in terms of homological representations. From a pure homological representation-theoretic point of view, these truncated Lawrence representations are interesting in its own right. However, it is now understood that it is not necessary to factor through the process of truncation if one plans to ultimately specialize at roots of unity: at these specializations, the multiforks labeled by elements in $E_{n,m}^{\geq r}$ vanish at Borel-Moore homology (see~\cite[Corollary~4.10]{Homological-Martel}).
\end{remark}

\begin{definition}{(Lawrence representations specialized at the $2r^{th}$-root of unity)} \label{def:l2r}

Let $r \geq 2$ be an integer and $q=\xi_{r}=e^{\frac{\pi \sqrt{-1}}{r}}$ the $2r^{th}$-root of unity. We refer the (truncated) Lawrence representation
\begin{equation*}
\mathcal{L}_{n,m,\alpha}^r:=l_{n,m}|_{\mathtt{x}=\xi_r^{-2(\alpha+r-1)}, \mathtt{q}=-\xi_r^2}
\end{equation*}
where $\mathtt{q}$ is specialized to $-\xi_r^2=-e^{\frac{2\pi \sqrt{-1}}{r}}$ and $\mathtt{x}$ is specialized to $\xi_r^{-2(\alpha+r-1)}$ as the \emph{Lawrence representation specialized at the $2r^{th}$-root of unity}. 
\end{definition}

\begin{definition}{(Lawrence representations specialized at roots of unity)}

We call the family of homological representations $\mathcal{L}_{n,m,\alpha}^r$ as the \emph{Lawrence representations specialized at roots of unity} or the \emph{$q$-specialized Lawrence representations} for short.
\end{definition}

\section{The unrolled quantum $\slt$}
In this section we recall the algebra $\overline{U}_{q}^H\mathfrak{sl}(2)$ and a category of modules over this algebra. Fix an integer $r\geq 2$ and let $q=e^\frac{\pi\sqrt{-1}}{r}$ be a $2r^{th}$-root of unity. Let $\mathbb{C}$ be the complex numbers and $\ddot{\mathbb{C}}=(\mathbb{C}\setminus \mathbb{Z})\cup r\mathbb{Z}.$ We use the notation $q^x=e^{\frac{\pi\sqrt{-1}}{r}x}$. For $n\in \N$, we also set 
\begin{equation*}
\qn{x}=q^x-q^{-x},\quad\qN{x}=\frac{\qn x}{\qn1},\quad\qn{n}!=\qn{n}\qn{n-1}\cdots\qn{1},\et\qN{n}!=\qN{n}\qN{n-1}\cdots\qN{1}.
\end{equation*}

\subsection{The Drinfel'd-Jimbo quantum group}
Let $\Uq$ be the $\mathbb{C}$-algebra given by generators $E, F, K, K^{-1}$ and relations:
\begin{align}\label{E:RelDCUqsl}
  KK^{-1}&=K^{-1}K=1, & KEK^{-1}&=q^2E, & KFK^{-1}&=q^{-2}F, &
  [E,F]&=\frac{K-K^{-1}}{q-q^{-1}}.
\end{align}
The algebra $\Uq$ is a Hopf algebra where the coproduct, counit, and antipode are defined by
\begin{align}\label{E:HopfAlgDCUqsl}
  \Delta(E)&= 1\otimes E + E\otimes K,
  &\varepsilon(E)&= 0,
  &S(E)&=-EK^{-1},
  \\
  \Delta(F)&=K^{-1} \otimes F + F\otimes 1,
  &\varepsilon(F)&=0,& S(F)&=-KF,
    \\
  \Delta(K)&=K\otimes K,
  &\varepsilon(K)&=1,
  & S(K)&=K^{-1}
.\label{E:HopfAlgDCUqsle}
\end{align}
Let $\UqMed$ be the algebra $\Uq$ modulo the relations $E^r=F^r=0$.

\subsection{A modified version of $\Uq$}\label{SS:UqH} 
Let $\UsltH$ be the $\mathbb{C}$-algebra given by generators $E$, $F$, $K$, $K^{-1}$, $H$ and relations in \eqref{E:RelDCUqsl} along with the relations:
\begin{align*}
  HK&=KH,
& [H,E]&=2E, & [H,F]&=-2F.
\end{align*}
The algebra $\UsltH$ is a Hopf algebra where the coproduct, counit, and antipode are defined in \eqref{E:HopfAlgDCUqsl}--\eqref{E:HopfAlgDCUqsle} and by
\begin{align*}
  \Delta(H)&=H\otimes 1 + 1 \otimes H,
  & \varepsilon(H)&=0,
  &S(H)&=-H.
\end{align*}
Define the \emph{unrolled quantum group} $\overline{U}_{q}^H\mathfrak{sl}(2)$ to be the Hopf algebra $\UsltH$ modulo the relations $E^r=F^r=0$.

Let $V$ be a finite-dimensional $\overline{U}_q^H\mathfrak{sl}(2)$-module. An eigenvalue $\lambda\in \C$ of the operator $H:V\to V$ is called a \emph{weight} of $V$ and the associated eigenspace is called a \emph{weight space}. A vector $v$ in the $\lambda$-eigenspace of $H$ is a \emph{weight vector} of \emph{weight} $\lambda$, i.e. $Hv=\lambda v$.  We call $V$ a \emph{weight module} if $V$ splits as a direct sum of weight spaces and $\qr^H=K$ as operators on $V$, i.e. $Kv=q^{\lambda} v$ for any vector $v$ of weight $\lambda$.  Let $\cat$ be the category of finite-dimensional weight $\overline{U}_q^H\mathfrak{sl}(2)$-modules.

Since $\overline{U}_q^H\mathfrak{sl}(2)$ is a Hopf algebra, $\cat$ is a tensor category where the unit $\mathbb{I}$ is the $1$-dimensional trivial module $\mathbb{C}$.  Moreover, $\cat$ is $\mathbb{C}$-linear: $\Hom$-sets are $\mathbb{C}$-modules, the composition and tensor product of morphisms are $\mathbb{C}$-bilinear, and $\End_\cat(\mathbb{I})=\C\Id_\mathbb{I}$.  When it is clear, we denote the unit $\mathbb{I}$ by $\mathbb{C}$.  We say a module $V$ is \emph{simple} if it has no proper submodules. For a module $V$ and a morphism $f\in\End_\cat(V)$, we write $\brk f_V=\lambda\in\C$ if $f-\lambda\Id_V$ is nilpotent.  If $V$ is simple, then Schur's lemma implies that $\End_\cat(V)=\C\Id_V$. Thus for $f\in \End_\cat(V)$, we have $f=\brk{f}_V \Id_V$.

We will now recall the fact that the category $\cat$ is a ribbon category. Let $V$ and $W$ be objects of $\cat$. Let $\{v_i\}$ be a basis of $V$ and $\{v_i^*\}$ be a dual basis of $V^*=\Hom_\C(V,\C)$. Then
\begin{align*}
  \coev_V \maps & \C \rightarrow V\otimes V^{*}, \text{ given by } 1 \mapsto \sum
  v_i\otimes v_i^*,  &
  \ev_V \maps & V^*\otimes V\rightarrow \C, \text{ given by }
  f\otimes w \mapsto f(w)
\end{align*}
are duality morphisms of $\cat$. In \cite{Oh}, Ohtsuki truncates the usual formula of the $h$-adic quantum $\slt$ $R$-matrix to define an operator on $V\otimes W$ by
\begin{equation}
  \label{eq:R}
  R=\qr^{H\otimes H/2} \sum_{n=0}^{r-1} \frac{\{1\}^{2n}}{\{n\}!}\qr^{n(n-1)/2}
  E^n\otimes F^n
\end{equation}
where $q^{H\otimes H/2}$ is the operator given by
\begin{equation*}
q^{H\otimes H/2}(v\otimes v') =q^{\lambda \lambda'/2}v\otimes v'
\end{equation*}
for weight vectors $v$ and $v'$ of weights $\lambda$ and
$\lambda'$ respectively. The $R$-matrix is not an element in $\Ubar\otimes \Ubar$.
However, the action of $R$ on the tensor product of two objects of
$\cat$ is a well-defined linear map.
Moreover, $R$ gives rise to a braiding $c_{V,W} \maps V\otimes W
\rightarrow W \otimes V$ on $\cat$ defined by $v\otimes w \mapsto
\tau(R(v\otimes w))$ where $\tau$ is the permutation $x\otimes
y\mapsto y\otimes x$.
This braiding follows from the invertibility of the $R$-matrix.  An explicit inverse
(see \cite[Section 2.1.2]{BDGG} and \cite{Oh}) is given by
\begin{equation}
  \label{eq:Rinverse}
  R^{-1}= \left(\sum_{n=0}^{r-1} (-1)^n  \frac{\{1\}^{2n}}{\{n\}!}\qr^{-n(n-1)/2}
  E^n\otimes F^n \right) \qr^{-H\otimes H/2}.
\end{equation}

Let $\theta$ be the operator given by
\begin{equation}
\theta=K^{r-1}\sum_{n=0}^{r-1}
\frac{\{1\}^{2n}}{\{n\}!}\qr^{n(n-1)/2} S(F^n)\qr^{-H^2/2}E^n
\end{equation}
where $q^{-H^2/2}$ is an operator defined on a weight vector $v_\lambda$ by
$q^{-H^2/2}.v_\lambda = q^{-\lambda^2/2}v_\lambda.$
Ohtsuki shows that the family of maps $\theta_V:V\rightarrow V$ in
$\cat$ defined by $v\mapsto \theta^{-1}v$ is a twist (see
\cite{Mur08,Oh}).

Now the ribbon structure on $\cat$ yields right duality morphisms
\begin{equation}\label{E:d'b'}
  \tev_{V}=\ev_{V}c_{V,V^*}(\theta_V\otimes\Id_{V^*})\text{ and }\tcoev_V =(\Id_{V^*}\otimes\theta_V)c_{V,V^*}\coev_V
\end{equation}
which are compatible with the left duality morphisms $\{\coev_V\}_V$ and
$\{\ev_V\}_V$. These duality morphisms are given explicitly by \begin{align*}
  \tcoev_{V} \maps & \C \rightarrow V^*\otimes V, \text{ where } 1 \mapsto
  \sum v_i^* \otimes K^{r-1}v_i, \\ \tev_{V} \maps & V\otimes V^*\rightarrow
  \C, \text{ where } v\otimes f \mapsto f(K^{1-r}v).
\end{align*}
The \emph{quantum dimension} $\qdim(V)$ of an object $V$ in $\cat$ is defined by
\begin{equation*}
\qdim(V)= \brk{\tev_V\circ \coev_V}_\mathbb{I}=\sum  v_i^*(K^{1-r}v_i) \ .
\end{equation*}

For $g\in\C/2\Z$, define
$\cat_{g}$ as the full subcategory of weight modules whose weights
are all in the class  $g$ (mod $2\Z$). Then $\cat=\{\cat_g\}_{g\in \C/2\Z}$ is a $\C/2\Z$-graded category (where $\C/2\Z$ is an additive group). Let $V\in\cat_g$ and $V'\in\cat_{g'}$. Then the weights of $V\otimes V'$ are congruent to $g+g' \mod 2\Z$,
and so the tensor product is in $\cat_{g+g'}$.  Also, if $g\neq g'$
 then $\Hom_\cat(V, V')=0$ since morphisms in $\cat$  preserve weights.
Finally, if $f\in V^*=\Hom_\C(V,\C)$, then by definition the action of
$H$ on $f$ is given by $(Hf)(v)=f(S(H)v)=-f(Hv)$
and so $V^{*}\in\cat_{-g}$.
We call the part of the category where $g=0,1$, the {\it singular } part of the category and refer to the objects in this part of the category as singular objects. The category $\cat_g$ is non-semisimple if $g$ is singular, otherwise $g$ is called {\it generic} and $\cat_g$ is semisimple.

We now consider the following classes of finite-dimensional highest weight modules. For each $n\in \{ 0, \dots, r-1 \}$ let $S_n$ be the usual $(n+1)$-dimensional simple highest weight $\overline{U}_q^H\mathfrak{sl}(2)$-module with highest weight $n$. The module $S_n$ is the highest weight module with a highest weight vector $s_0$ such that $Es_0 = 0$ and $Hs_0 = ns_0$. Then $\{ s_0, s_1, \dots, s_n \}$ is a basis of $S_n$ where \begin{equation}\label{E:BasisS}
H.s_i=(n-2i) s_i,\quad E.s_i= \frac{\qn i\qn{n+1-i}}{\qn1^2}
s_{i-1} ,\quad F.s_i=s_{i+1}.
\end{equation}
The quantum dimension of $S_n$ is $\qdim(S_n) = (-1)^n\frac{\qn{n+1}}{\qn{1}}$.

Next, for each $\alpha\in \C$, we let $V_\alpha$ be the $r$-dimensional
highest weight $\overline{U}_q^H\mathfrak{sl}(2)$-module of highest weight $\alpha + r - 1$.  The
module $V_\alpha$ has a basis $\{v_0, \dots, v_{r-1} \}$ whose action is
given by
\begin{equation}\label{E:BasisV}
H.v_i=(\alpha + r - 1-2i) v_i,\quad E.v_i= \frac{\qn i\qn{i-\alpha}}{\qn1^2}
v_{i-1} ,\quad F.v_i=v_{i+1}.
\end{equation}
For all $\alpha\in \C$, the quantum dimension of $V_\alpha$ is zero:
$$\qdim(V_\alpha)=\sum\limits_{i=0}^{r-1}v_i^*(K^{1-r}v_i) = q^{(1-r)(\alpha+r-1)} \frac{1-q^{2r(r-1)}}{1-q^{2(r-1)}}=0.$$

For $a\in \Z$, let $\mathbb{C}_{ra}^{H}$ be the one dimensional module in $\mathcal{C}_{\overline{0}}$ where both $E$ and $F$ act by zero and $H$ acts by $ra$. Every simple module of $\mathcal{C}$ is isomorphic to exactly one of the modules in the list:
\begin{itemize}
    \item $S_n \otimes \C_{ra}^{H}$, for $n=0,\dots, r-2$ and $a\in \Z$,
    \item $V_\alpha$ for $\alpha \in (\mathbb{C}\setminus \mathbb{Z}) \cup r\mathbb{Z}$.
\end{itemize}

A key ingredient in the construction of non-semisimple TQFTs from the representations of $\overline{U}_q^H\mathfrak{sl}(2)$ is the notion of a \emph{modified trace} $\mt:=\mt_V$~\cite{BCGP2, CGP14, GPT09}. Taking the modified trace of the identity gives a notion of \emph{modified dimension}, which can be viewed as renormalizing the representations whose usual quantum dimension is zero. The modified dimension of $V_\alpha$ is given by 
\begin{equation} \label{eq:md}
\md(\alpha) := \md(V_\alpha) =  (-1)^{r-1}\frac{r\qn{\alpha}}{\qn{r\alpha}}=(-1)^{r-1}r\frac{q^{\alpha}-q^{-\alpha}}{q^{r\alpha}-q^{-r\alpha}}
\end{equation}
for $\alpha \in \ddot{\mathbb{C}}$.

We will use graphical calculus to pictorially present the data of the category of representations $\mathcal{C}$. We refer the reader to~\cite{BK, Turaev-Book} for formal definitions of the language of graphical calculus. All diagrams are read from bottom to top. Our main focus is to study representations of the braid group $\Br_n$ on various morphism spaces in $\mathcal{C}$.

\section{Fusion Trees} \label{sec:fusion-trees}

We study ribbon graphs colored by simple objects from the \emph{generic} part of the category. For integers $n \geq 1$ and $m\geq 0$, let $\mathcal{H}_{n,m,\alpha}^r:= \Hom(V_{n\alpha+(n-1)(r-1)-2m}, V_{\alpha}^{\otimes n})$. The superscript $r$ is used to specify that the $\Hom$-space comes from the category constructed from the $2r^{th}$-root of unity. Note that this is nonzero if and only if $m\in \{ 0, 1, \dots, (n-1)(r-1) \}$. For our arguments, we choose specific morphisms
\begin{equation} \label{eq:conventions}
Y_{\gamma}^{\alpha,\beta}:=\hackcenter{\begin{tikzpicture}[ scale=1.1]
\begin{scope}[decoration={markings, mark=at position 0.5 with {\arrow{>}}}]
  \draw[ultra thick, black, postaction=decorate] (0,0) to (.6,-.6);
  \draw[ultra thick, black, postaction=decorate] (.6,-.6) to (.6,-1.2);
  \draw[ultra thick, black, postaction=decorate] (1.2,0) to (.6,-.6);
\end{scope}
   \node at (0,0.2) {$\alpha$};
   \node at (1.2,0.2) {$\beta$};
   \node at (.6,-1.4) {$\gamma$};
\end{tikzpicture} }
\in \Hom(V_\gamma, V_\alpha \otimes V_\beta),\qquad Y^{\gamma}_{\alpha,\beta}:=\hackcenter{\begin{tikzpicture}[ scale=1.1]
\begin{scope}[decoration={markings, mark=at position 0.75 with {\arrow{>}}}]
  \draw[ultra thick, black, postaction=decorate] (.6,0) to (.6,-.6);
  \draw[ultra thick, black, postaction=decorate] (.6,-.6) to (1.2,-1.2);
  \draw[ultra thick, black, postaction=decorate] (.6,-.6) to (0,-1.2);
\end{scope}
   \node at (0.6,0.2) {$\gamma$};
   \node at (0,-1.4) {$\alpha$};
   \node at (1.2,-1.4) {$\beta$};
\end{tikzpicture} }
\in \Hom(V_\alpha \otimes V_\beta, V_\gamma)
\end{equation}
coming from the convention of \cite{CGP14, CM}.

\subsection{$\alpha$-Fusion Trees}
First, we discuss the dimension of $\mathcal{H}_{n,m,\alpha}^r$ and enumerate its basis elements using fusion trees. Before proceeding, we introduce an indexing set:
\begin{equation*}
H_k = \{ k-1, k-3, \dots, -(k-1) \}.
\end{equation*}

\begin{lemma}
Suppose $\alpha,2\alpha, \dots, n\alpha\in \mathbb{C}\setminus \mathbb{Z}$ (i.e. generic values). Then, the dimension of $\mathcal{H}_{n,m,\alpha}^r$ is $d_{n,m}^r = |E_{n,m}^r|$.
\end{lemma}

\begin{proof}
For generic values of $\alpha$ and $\beta$ with $\alpha+\beta$ generic, recall that the tensor product structure is given by $$V_\alpha \otimes V_\beta \cong \bigoplus\limits_{\gamma \in \alpha + \beta + H_r} V_\gamma = \bigoplus\limits_{m=0}^{r-1} V_{\alpha + \beta + (r-1) - 2m}.$$ We will use the shorthand notation that $nV_\beta:=V_\beta^{\oplus n}$.

We claim that $V_\alpha^{\otimes n}\cong \bigoplus\limits_{m=0}^{(n-1)(r-1)}d_{n,m}^r V_{n\alpha + (n-1)(r-1)-2m}$ to conclude that $\dim \mathcal{H}_{n,m,\alpha}^r = d_{n,m}^r$ via Schur's lemma. We approach the claim inductively.

When $n=2$, we have $V_\alpha \otimes V_\alpha \cong \bigoplus\limits_{m=0}^{r-1} V_{2\alpha + (r-1) - 2m} = \bigoplus\limits_{m=0}^{r-1} d_{2,m}^{r} V_{2\alpha + (r-1) - 2m}$. Next, 
\begin{align*}
V_\alpha^{\otimes (k+1)} & \cong  \left( \bigoplus\limits_{m=0}^{(k-1)(r-1)}d_{k,m}^r V_{k\alpha +(k-1)(r-1) - 2m} \right) \otimes V_\alpha\\
& =   \bigoplus\limits_{m=0}^{(k-1)(r-1)}d_{k,m}^r \left( V_{k\alpha +(k-1)(r-1) - 2m}  \otimes V_\alpha \right)\\
& = \bigoplus\limits_{m=0}^{(k-1)(r-1)}d_{k,m}^r \left( \bigoplus\limits_{j=0}^{r-1} V_{(k+1)\alpha + k(r-1) -2(m +j)} \right)\\
& = \bigoplus\limits_{m=0}^{k(r-1)} \left( \sum\limits_{j=0}^{r-1} d_{k,m-j}^r \right) V_{(k+1)\alpha + k (r-1) - 2m}\\
& = \bigoplus\limits_{m=0}^{k(r-1)} d_{k+1,m}^r V_{(k+1)\alpha + k (r-1)-2m} .
\end{align*}
The second to last equality comes from the fact that elements in $E_{k+1,m}^r$ are determined by appending one additional entry of value $j$ to elements in $E_{k,m-j}^r$. Therefore, 
\begin{equation*}
d_{k+1,m}^r = |E_{k+1,m}^r| = \sum\limits_{j=0}^{r-1}|E_{k,m-j}^r| = \sum\limits_{j=0}^{r-1} d_{k,m-j}^{r}.
\end{equation*}
Note that, by definition of $E_{n,m}^r$, if $m-j<0$ then $d_{k,m-j}^r = 0$.
\end{proof}

It is useful to have the notion of a $\mathcal{H}_{n,m,\alpha}^r$ path. This is a path $$\mathbf{p}=(\alpha=\gamma_0, \gamma_1, \dots, \gamma_{n-2}, n\alpha + (n-1)(r-1) - 2m=\gamma_{n-1})$$ in the triangle of multinomial univariate coefficients from $\alpha$ to $n \alpha + (n-1)(r-1) - 2m$ by taking $(n-1)$ total steps where each step moves in the $(k,-1)$ direction where $k\in H_{r}$.

Here, by multinomial univariate coefficients, we mean the coefficient of $x^m$ in the expansion of $(1+x+x^2 + \dots + x^{r-1})^{n}$. In particular, when $r=2$, we retrieve the familiar binomial coefficients and Pascal's triangle. The triangle of multinomial univariate coefficients will provide us a combinatorial way to enumerate and explicitly write down the full basis of $\mathcal{H}_{n,m,\alpha}^r$. The corresponding Bratteli diagram will recover back the values of $d_{n,m}^r$. We demonstrate with a concrete example.

\textbf{Example: $r=4$.} Note that $H_4 = \{ 3, 1, -1, -3 \}$. The triangle of quadinomial coefficients \eqref{eq:triangle} gives the values of $d_{n,m}^4$.
\begin{equation} \label{eq:triangle}
 \hackcenter{
\begin{tikzpicture}[yscale=1, scale=0.7,  decoration={markings, mark=at position 0.6 with {\arrow{>}};},]
\node at (0,0) {$\scs 1 $};
\node at (-3,-1) {$\scs 1$};
\node at (-1,-1) {$\scs 1$};
\node at (1,-1) {$\scs 1$};
\node at (3,-1) {$\scs 1$};
\node at (-6,-2) {$\scs 1$};
\node at (-4,-2) {$\scs 2$};
\node at (-2,-2) {$\scs 3$};
\node at (0,-2) {$\scs 4$};
\node at (2,-2) {$\scs 3$};
\node at (4,-2) {$\scs 2$};
\node at (6,-2) {$\scs 1$};
\node at (0,-3) {$\scs \cdots$};
\node at (-8,-4) {$\scs d_{n,3(n-1)}^4$};
\node at (-4,-4) {$\scs \cdots$};
\node at (0,-4) {$\scs d_{n,m}^4$};
\node at (4,-4) {$\scs \cdots$};
\node at (8,-4) {$\scs d_{n,0}^4  $};
\end{tikzpicture}}
\end{equation}
One may convert each entry to admissible entries of the fusion channels to obtain the following triangle \eqref{quadrinomial-triangle1}. That is:
\begin{enumerate}
    \item A row will correspond to a level $n$, starting with $n=1$ at the top, indicated by the coefficient of $\alpha$.
    \item Within each row, each entry will placed on the $x$-position corresponding to their constant values.
    \item At level $n$, the terms correspond to the simple object $V_{n\alpha+(n-1)(r-1)-2m}$ that appears in the semisimple decomposition of $V_\alpha^{\otimes n}$. 
\end{enumerate}
For example, at level $n=3$, $V_\alpha^{\otimes3}\cong V_{3\alpha+6} \oplus 2V_{3\alpha+4} \oplus 3V_{3\alpha+2} \oplus 4V_{3\alpha}\oplus 3V_{3\alpha-2}\oplus 2V_{3\alpha-4} \oplus V_{3\alpha-6}$. Here, $3\alpha-4$ will be placed on the $x$-position of $-4$.
\begin{equation}\label{quadrinomial-triangle1}
 \hackcenter{
\begin{tikzpicture}[yscale=1, scale=0.7,  decoration={markings, mark=at position 0.6 with {\arrow{>}};},]
\draw[<->, thick] (-7,1) -- (7,1) node[right] {$\scs x$};
\foreach \x in {-6,-5,-4,-3,-2,-1,1,0,1,2,3,4,5,6}
   \draw (\x,1.1) -- (\x,0.9);
\node[above] at (0,1) {$\scs 0$};
\node[above] at (1,1) {$\scs 1$};
\node[above] at (2,1) {$\scs 2$};
\node[above] at (3,1) {$\scs 3$};
\node[above] at (4,1) {$\scs 4$};
\node[above] at (5,1) {$\scs 5$};
\node[above] at (6,1) {$\scs 6$};
\node[above] at (-1,1) {$\scs -1$};
\node[above] at (-2,1) {$\scs -2$};
\node[above] at (-3,1) {$\scs -3$};
\node[above] at (-4,1) {$\scs -4$};
\node[above] at (-5,1) {$\scs -5$};
\node[above] at (-6,1) {$\scs -6$};
\node at (0,0) {$\scs \alpha $};
\node at (-3,-1) {$\scs 2\alpha-3$};
\node at (-1,-1) {$\scs 2\alpha-1$};
\node at (1,-1) {$\scs 2\alpha+1$};
\node at (3,-1) {$\scs 2\alpha+3$};
\node at (-6,-2) {$\scs 3\alpha-6$};
\node at (-4,-2) {$\scs 3\alpha-4$};
\node at (-2,-2) {$\scs 3\alpha-2$};
\node at (0,-2) {$\scs 3\alpha$};
\node at (2,-2) {$\scs 3\alpha+2$};
\node at (4,-2) {$\scs 3\alpha+4$};
\node at (6,-2) {$\scs 3\alpha+6$};
\node at (0,-3) {$\scs \vdots$};
\node at (-8,-4) {$\scs n\alpha-3(n-1)$};
\node at (-4,-4) {$\scs \cdots$};
\node at (0,-4) {$\scs n\alpha+3(n-1)-2m$};
\node at (4,-4) {$\scs \cdots$};
\node at (8,-4) {$\scs n\alpha+3(n-1)$};
\node at (7,-3) {$\scs \ddots$};
\node at (-7,-3) {$\scs \mathinner{\reflectbox{$\ddots$}}$};
\end{tikzpicture}}
\end{equation}
We will also label such paths by direction sequences given elements of $E_{n,m}^r$ which indicates that at the $i^{th}$ step the path moves by $(r-1-2e_i)$ in the $x$-direction. For example, there is the path $(\alpha,2\alpha-3,3\alpha-2,4\alpha + 1)$ given in \eqref{quadrinomial-triangle2}.  The corresponding direction sequence is $(3,1,0)\in E_{4,4}^4$.
\begin{equation}\label{quadrinomial-triangle2}
 \hackcenter{
\begin{tikzpicture}[yscale=1, scale=0.7,  decoration={markings, mark=at position 0.6 with {\arrow{>}};},]
\node at (0,0) {$\scs \alpha $};
\node at (-3,-1) {$\scs 2\alpha-3$};
\node at (-1,-1) {$\scs 2\alpha-1$};
\node at (1,-1) {$\scs 2\alpha+1$};
\node at (3,-1) {$\scs 2\alpha+3$};
\node at (-6,-2) {$\scs 3\alpha-6$};
\node at (-4,-2) {$\scs 3\alpha-4$};
\node at (-2,-2) {$\scs 3\alpha-2$};
\node at (0,-2) {$\scs 3\alpha$};
\node at (2,-2) {$\scs 3\alpha+2$};
\node at (4,-2) {$\scs 3\alpha+4$};
\node at (6,-2) {$\scs 3\alpha+6$};
\node at (-9,-3) {$\scs 4\alpha-9$};
\node at (-7,-3) {$\scs 4\alpha-7$};
\node at (-5,-3) {$\scs 4\alpha-5$};
\node at (-3,-3) {$\scs 4\alpha - 3$};
\node at (-1,-3) {$\scs 4\alpha-1$};
\node at (1,-3) {$\scs 4\alpha+1$};
\node at (3,-3) {$\scs 4\alpha+3$};
\node at (5,-3) {$\scs 4\alpha+5$};
\node at (7,-3) {$\scs 4\alpha+7$};
\node at (9,-3) {$\scs 4\alpha+9$};
\draw[postaction={decorate}] (-.2,-.2) to  (-2.8,-.8);
\draw[postaction={decorate}] (-2.8,-1.2) to  (-2.2,-1.8);
\draw[postaction={decorate}] (-1.8,-2.2) to  (.8,-2.8);
\end{tikzpicture}}
\end{equation}
An explicit formula for translating between a $\mathcal{H}_{n,m,\alpha}^r$ path $\mathbf{p}$ and $\mathbf{e} \in E_{n,m}^r$ is provided later in Equation~\eqref{eq:bijection}.

\begin{lemma} \label{lemma:basis}
A basis of the vector space $\mathcal{H}_{n,m,\alpha}^r$ can be enumerated by paths $\alpha$ to $n \alpha + (n-1)(r-1) - 2m$. In particular, the basis is also enumerated by the indexing set $E_{n,m}^r$.
\end{lemma}

\begin{proof}
The bijection maps the path $\mathbf{p}=(\alpha, \gamma_1, \dots, \gamma_{n-2}, n\alpha + (n-1)(r-1)-2m)$ to the homomorphism encoded by the tree
\begin{equation} \label{eq:tree}
Y_{\mathbf{p}}:=\hackcenter{\begin{tikzpicture}[scale=1]
\begin{scope}[decoration={markings, mark=at position 0.5 with {\arrow{>}}}]
  \draw[ultra thick, black, postaction=decorate] (0,0) to (0.6,-0.6);
  \draw[ultra thick, black, postaction=decorate] (0.6,-0.6) to (1.2,-1.2);
  \draw[ultra thick, dotted] (1.2,-1.2) to (1.8,-1.8);
  \draw[ultra thick, black, postaction=decorate] (1.8,-1.8) to (2.4,-2.4);
  \draw[ultra thick, black, postaction=decorate] (1.2,0) to (.6,-.6);
  \draw[ultra thick, black, postaction=decorate] (2.4,0) to (1.2,-1.2);
  \draw[ultra thick, black, postaction=decorate] (3.6,0) to (1.8,-1.8);
  \draw[ultra thick, black, postaction=decorate] (4.8,0) to (2.4,-2.4);
  \draw[ultra thick, black, postaction=decorate] (2.4,-2.4) to (2.4,-3.3);
\end{scope}
   \node at (0,0.2) {$\alpha$};
   \node at (1.2,0.2) {$\alpha$};
   \node at (2.4,0.2) {$\alpha$};
   \node at (3,0.2) {$\cdots$};
   \node at (3.6,0.2) {$\alpha$};
   \node at (4.8,0.2) {$\alpha$};
   \node at (.7,-1.1) {$\gamma_1$};
   \node at (1.75,-2.3) {$\gamma_{n-2}$};
   \node at (4.65,-3) {$n\alpha + (n-1)(r-1)-2m$};
\end{tikzpicture} }.
\end{equation}
For the latter statement, suppose we are given an element $(e_1, \dots, e_{n-1}) \in E_{n,m}^r$. We correspond this element to the path $(\gamma_0, \gamma_1, \dots, \gamma_{n-1})$ where $\gamma_k = (k+1)\alpha + k(r-1) - 2 \sum\limits_{i=1}^k e_i$.
\end{proof}

Denote by $P_{n,m,\alpha}^r$ the set of all paths in $\mathcal{H}_{n,m,\alpha}^r$. Lemma~\ref{lemma:basis} tells us that this indexing set is completely characterized by sending each element of $E_{n,m}^r$ to the corresponding path as given in the set bijection defined above. That is:
\begin{align} \label{eq:bijection}
\begin{split}
E_{n,m}^r & \leftrightarrow P_{n,m,\alpha}^r\\
\mathbf{e} =(e_1, \dots, e_{n-1}) & \mapsto \mathbf{p}=(\gamma_0, \gamma_1, \dots, \gamma_{n-1}) \text{ where } \gamma_k = (k+1)\alpha + k(r-1) - 2 \sum\limits_{i=1}^k e_i.
\end{split}
\end{align}

\begin{definition}{($\alpha$-Fusion Tree)}
We refer to an element in $\mathcal{H}_{n,m,\alpha}^r$ of the form~\eqref{eq:tree} as an \emph{$\alpha$-fusion tree}. By Lemma~\ref{lemma:basis}, the basis of $\mathcal{H}_{n,m,\alpha}^r$ is enumerated by the set of all possible $\alpha$-fusion trees $\{ Y_{\mathbf{p}} \}_{\mathbf{p} \in P_{n,m,\alpha}^r}$ residing in the $\Hom$-space.

The braid group representation $\rho_{\gamma_{n-1}}: \Br_n \rightarrow \End(\mathcal{H}_{n,m,\alpha}^r)$ is defined by the following map
\begin{equation} \label{eq:braid}
\sigma_i \mapsto \hackcenter{\begin{tikzpicture}[scale=0.7]
\begin{scope}[decoration={markings, mark=at position 1 with {\arrow{>}}}]
  \draw[ultra thick, black, postaction=decorate] (0,1.2) to (1.2,0);
  \draw[line width=10pt, white] (1.2,1.2) to (0,0);
  \draw[ultra thick, black, postaction=decorate] (1.2,1.2) to (0,0);
  \draw[ultra thick, black, postaction=decorate] (-3, 1.2) to (-3, 0);
  \draw[ultra thick, black, postaction=decorate] (-1, 1.2) to (-1, 0);
  \draw[ultra thick, black, postaction=decorate] (4.2, 1.2) to (4.2, 0);
  \draw[ultra thick, black, postaction=decorate] (2.2, 1.2) to (2.2, 0);
  \draw (-3,1.5) node {$\scs \alpha_1$};
  \draw (-1,1.5) node {$\scs \alpha_{i-1}$};
  \draw (-0,1.5) node {$\scs \alpha_i$};
  \draw (1.2,1.5) node {$\scs \alpha_{i+1}$};
  \draw (2.2,1.5) node {$\scs \alpha_{i+2}$};
  \draw (4.2,1.5) node {$\scs \alpha_{n}$};
  \draw (-2,0.6) node {$\cdots$};
  \draw (3.2,0.6) node {$\cdots$};
\end{scope}
\end{tikzpicture} }
\end{equation}
which pictorially stacks on top of an $\alpha$-fusion tree. The subscript $i$ of $\alpha_i$ in Equation~\eqref{eq:braid} indicate the positioning index of $\alpha$.
\end{definition}

We are now ready to establish an identification between the space of $\alpha$-fusion trees and homological representations.

\begin{theorem}\label{thm:fusiontree-homological}
The braid group representation on $\mathcal{H}_{n,m,\alpha}^r$ is projectively isomorphic to the Lawrence representation specialized at a root of unity $\mathcal{L}_{n,m,\alpha}^{r}$.
\end{theorem}

\begin{proof}
Since morphisms in this category preserve weights, respect the action of the quantum group, and $V_{\gamma_{n-1}}$ is simple, any map $Y\in \mathcal{H}_{n,m,\alpha}^r$ is determined by the image $w_Y:=Y(v_0)$ of the highest weight vector $v_0 \in V_{\gamma_{n-1}}$ of weight $n(\alpha + r-1) - 2m$. In other words, $\mathcal{H}_{n,m,\alpha}^r$ as a vector space is isomorphic to the highest weight subspace $W_{n,m}:=\mathrm{Ker}(K - q^{n(\alpha+r-1)-2m}\id)\cap \mathrm{Ker}E \subset V_{\alpha}^{\otimes n}$ and the $\mathbb{C}$-span of pairwise distinct $Y_{\mathbf{p}}$ and $Y_{\mathbf{p'}}$ corresponds to pairwise linearly independent one-dimensional subspaces of $W_{n,m}$. Furthermore, we also have an isomorphism of braid group representations because the braid group action on an element $Y\in \mathcal{H}_{n,m,\alpha}^r$ is completely determined by the output of the braid acting on its highest weight vector $w_Y \in W_{n,m}$.

By~\cite[Theorem~4.2]{Ito16} (more generally, see~\cite[Corollary~7.1]{Homological-Martel}), the braid group representation on the highest weight subspace $W_{n,m}$ is projectively isomorphic to the truncated Lawrence representation $l_{n,m}$ specialized at $t=-q^2$ and $x=q^{-2\lambda}$ where $\lambda = \alpha + r - 1$ is the highest weight of $V_\alpha$. We clarify that our isomorphism is projective as we do not alter the $R$-matrix by $q^{-\frac{H \otimes H}{2}}$.
\end{proof}

\begin{remark}
Quantum states of a quantum mechanical system correspond to an element of the \emph{projective} Hilbert space $P(\mathcal{H}_{n,m,\alpha}^r)$. Consequently, we are interested in studying \emph{projective} representations of the braid group. So the \emph{projective} isomorphism above is not a concern in the context of quantum computation.
\end{remark}

\subsection{Graphical Calculus} \label{subsec:graphical-calculus}
As mentioned earlier, the original quantum interpretation of the (truncated) Lawrence representations come from certain highest weight subspaces of tensor powers of the Verma modules associated to quantum $\slt$. For the remainder of this section, we leverage our new interpretation of the $q$-specialized Lawrence representations as fusion trees to reprove Ito's formula through the usage of graphical calculus techniques.

We introduce four graphical moves.\par
    \centerline{\makebox[4cm]{Pinch Move}  (Lemma~\ref{lemma:fusion-identity})\par}
    \centerline{\makebox[4cm]{Mirror Reflection}  (Lemma~\ref{lemma:fusion-channel})\par}
    \centerline{\makebox[4cm]{Braid Pop}  (Lemma~\ref{lemma:character-identity})\par}
    \centerline{\makebox[4cm]{Pruning}  (Lemma~\ref{lemma:isotoped-bubble})\par}

\begin{lemma}[Pinch Move]  \label{lemma:fusion-identity}
For generic values $\alpha_1, \dots, \alpha_n$ such that $\sum\limits_{i=1}^{j} \alpha_i$ are also generic for $j=2, \dots, n$, we have the following identity:
\begin{align*}
\id_{V_{\alpha_1} \otimes \cdots \otimes V_{\alpha_n}}
&= \sum\limits_{\substack{\gamma_1 \in \alpha_1 + \alpha_2 + H_r \\ \gamma_2 \in \gamma_1 + \alpha_3 + H_r \\ \vdots \\ \gamma_{n-1} \in \gamma_{n-2} + \alpha_n + H_r}} \md(\gamma_1) \md(\gamma_2) \cdots \md(\gamma_{n-1})
\begin{tikzpicture}[baseline=0, thick, scale=0.6, shift={(0,0)}]
\begin{scope}[decoration={markings, mark=at position 0.5 with {\arrow{>}}}]
    \draw[ultra thick, black, postaction=decorate] (0,3.5) to (0.8125,2.75);
    \draw[ultra thick, black, postaction=decorate] (1,3.5) to (0.8125,2.75);
    \draw[ultra thick, black, postaction=decorate] (0.8125,2.75) to (1.625,2);
    \draw[ultra thick, black, postaction=decorate] (2,3.5) to (1.625,2);
    \draw[ultra thick, black, dotted] (1.625, 2) to (2.4375, 1.25);
    \draw[ultra thick, black, postaction=decorate] (3, 3.5) to (2.4375, 1.25);
    \draw[ultra thick, black, postaction=decorate] (2.4375,1.25) to (3.25,0.5);
    \draw[ultra thick, black, postaction=decorate] (4, 3.5) to (3.25,0.5);
    \draw[ultra thick, black, postaction=decorate] (3.25,0.5) to (3.25,-0.5);
\end{scope}
\draw (0,3.5) node[above] {$\scs \alpha_1$};
\draw (1,3.5) node[above] {$\scs \alpha_2$};
\draw (2,3.5) node[above] {$\scs \alpha_3$};
\draw (3,3.5) node[above] {$\scs \cdots$};
\draw (4,3.5) node[above] {$\scs \alpha_{n}$};
\draw (0.75,2) node {$\scs \gamma_{1}$};
\draw (2.5,0.5) node {$\scs \gamma_{n-2}$};
\draw (3.5,0) node[right] {$\scs \gamma_{n-1}$};
\begin{scope}[decoration={markings, mark=at position 0.75 with {\arrow{>}}}]
    \draw[ultra thick, black, postaction=decorate] (0.8125,-2.75) to (0,-3.5);
    \draw[ultra thick, black, postaction=decorate] (0.8125,-2.75) to (1,-3.5);
    \draw[ultra thick, black, postaction=decorate] (1.625,-2) to (0.8125,-2.75);
    \draw[ultra thick, black, postaction=decorate] (1.625,-2) to (2,-3.5);
    \draw[ultra thick, black, dotted] (2.4375,-1.25) to (1.625,-2);
    \draw[ultra thick, black, postaction=decorate] (2.4375,-1.25) to (3,-3.5);
    \draw[ultra thick, black, postaction=decorate] (3.25,-0.5) to (2.4375,-1.25);
    \draw[ultra thick, black, postaction=decorate] (3.25,-0.5) to (4,-3.5);
\end{scope}
\draw (0,-3.5) node[below] {$\scs \alpha_1$};
\draw (1,-3.5) node[below] {$\scs \alpha_2$};
\draw (2,-3.5) node[below] {$\scs \alpha_3$};
\draw (3,-3.5) node[below] {$\scs \cdots$};
\draw (4,-3.5) node[below] {$\scs \alpha_{n}$};
\draw (0.75,-2) node {$\scs \gamma_{1}$};
\draw (2.5,-0.5) node {$\scs \gamma_{n-2}$};
\end{tikzpicture}.
\end{align*}
\end{lemma}

\begin{proof}
Since $\alpha_1$, $\alpha_2$, and $\alpha_1+\alpha_2$ are generic, the base case $n=2$ of the induction is satisfied by~\cite[Equation (Ni)]{CGP14}.
\begin{align*}
\id_{V_{\alpha_1} \otimes V_{\alpha_2}}
= \begin{tikzpicture}[baseline=0, thick, scale=0.75, shift={(0,0)}]
\begin{scope}[decoration={markings, mark=at position 0.5 with {\arrow{>}}}]
    \draw[ultra thick, black, postaction=decorate] (0.5,1.2) to (0.5,-1.2);
    \draw[ultra thick, black, postaction=decorate] (1.5,1.2) to (1.5,-1.2);
\end{scope}
\draw (0.5,0) node[left] {$\scs \alpha_1$};
\draw (1.5,0) node[right] {$\scs \alpha_2$};
\end{tikzpicture}
&= \sum\limits_{\gamma_1 \in \alpha_1 + \alpha_2 + H_r} \md(\gamma_1)
\begin{tikzpicture}[baseline=0, thick, scale=0.75, shift={(0,0)}]
\begin{scope}[decoration={markings, mark=at position 0.75 with {\arrow{>}}}]
    \draw[ultra thick, black, postaction=decorate] (0,1.2) to (.6,0.6);
    \draw[ultra thick, black, postaction=decorate] (.6,0.6) to (.6,-0.6);
    \draw[ultra thick, black, postaction=decorate] (1.2,1.2) to (.6,0.6);
    \draw[ultra thick, black, postaction=decorate] (.6,-0.6) to (0,-1.2);
    \draw[ultra thick, black, postaction=decorate] (.6,-0.6) to (1.2,-1.2);
\end{scope}
\draw (0,1.25) node[above] {$\scs \alpha_1$};
\draw (1,1.25) node[above] {$\scs \alpha_2$};
\draw (0.5,0.2) node[right] {$\scs \gamma_1$};
\draw (0,-1.25) node[below] {$\scs \alpha_1$};
\draw (1,-1.25) node[below] {$\scs \alpha_2$};
\end{tikzpicture}.
\end{align*}
Next, observe that applying (Ni) to the red-dashed box gives us the following equality. Since $\sum\limits_{i=1}^{n-1} \alpha_i$ is assumed to be generic, any admissible choices of $\gamma_{n-2}$ has to be generic as well.
\begin{align*}
\hackcenter{\begin{tikzpicture}[baseline=0, thick, scale=0.6, shift={(0,0)}]
\begin{scope}[decoration={markings, mark=at position 0.5 with {\arrow{>}}}]
    \draw[ultra thick, black, postaction=decorate] (0,3.5) to (0.8125,2.75);
    \draw[ultra thick, black, postaction=decorate] (1,3.5) to (0.8125,2.75);
    \draw[ultra thick, black, postaction=decorate] (0.8125,2.75) to (1.625,2);
    \draw[ultra thick, black, postaction=decorate] (2,3.5) to (1.625,2);
    \draw[ultra thick, black, dotted] (1.625, 2) to (2.4375, 1.25);
    \draw[ultra thick, black, postaction=decorate] (3, 3.5) to (2.4375, 1.25);
    \draw[ultra thick, black, postaction=decorate] (2.4375,1.25) to (2.4375,-1.25);
    \draw[ultra thick, black, postaction=decorate] (4,3.5) to (4,-3.5);
\end{scope}
\draw (0,3.5) node[above] {$\scs \alpha_1$};
\draw (1,3.5) node[above] {$\scs \alpha_2$};
\draw (2,3.5) node[above] {$\scs \cdots$};
\draw (3,3.5) node[above] {$\scs \alpha_{n-1}$};
\draw (0.75,2) node {$\scs \gamma_{1}$};
\draw[red,dashed] (0.75,0.8) rectangle (5.25,-0.8);
\draw (2.5,0.5) node[left] {$\scs \gamma_{n-2}$};
\draw (4,0.5) node[right] {$\scs \alpha_{n}$};
\begin{scope}[decoration={markings, mark=at position 0.75 with {\arrow{>}}}]
    \draw[ultra thick, black, postaction=decorate] (0.8125,-2.75) to (0,-3.5);
    \draw[ultra thick, black, postaction=decorate] (0.8125,-2.75) to (1,-3.5);
    \draw[ultra thick, black, postaction=decorate] (1.625,-2) to (0.8125,-2.75);
    \draw[ultra thick, black, postaction=decorate] (1.625,-2) to (2,-3.5);
    \draw[ultra thick, black, dotted] (2.4375,-1.25) to (1.625,-2);
    \draw[ultra thick, black, postaction=decorate] (2.4375,-1.25) to (3,-3.5);
\end{scope}
\draw (0,-3.5) node[below] {$\scs \alpha_1$};
\draw (1,-3.5) node[below] {$\scs \alpha_2$};
\draw (2,-3.5) node[below] {$\scs \cdots$};
\draw (3,-3.5) node[below] {$\scs \alpha_{n-1}$};
\draw (0.75,-2) node {$\scs \gamma_{1}$};
\end{tikzpicture} }
&~=~\sum\limits_{\gamma_{n-1} \in \gamma_{n-2} + \alpha_n + H_r} \md(\gamma_{n-1})
\hackcenter{\begin{tikzpicture}[baseline=0, thick, scale=0.6, shift={(0,0)}]
\begin{scope}[decoration={markings, mark=at position 0.5 with {\arrow{>}}}]
    \draw[ultra thick, black, postaction=decorate] (0,3.5) to (0.8125,2.75);
    \draw[ultra thick, black, postaction=decorate] (1,3.5) to (0.8125,2.75);
    \draw[ultra thick, black, postaction=decorate] (0.8125,2.75) to (1.625,2);
    \draw[ultra thick, black, postaction=decorate] (2,3.5) to (1.625,2);
    \draw[ultra thick, black, dotted] (1.625, 2) to (2.4375, 1.25);
    \draw[ultra thick, black, postaction=decorate] (3, 3.5) to (2.4375, 1.25);
    \draw[ultra thick, black, postaction=decorate] (2.4375,1.25) to (3.25,0.5);
    \draw[ultra thick, black, postaction=decorate] (4, 3.5) to (3.25,0.5);
    \draw[ultra thick, black, postaction=decorate] (3.25,0.5) to (3.25,-0.5);
\end{scope}
\draw (0,3.5) node[above] {$\scs \alpha_1$};
\draw (1,3.5) node[above] {$\scs \alpha_2$};
\draw (2,3.5) node[above] {$\scs \alpha_3$};
\draw (3,3.5) node[above] {$\scs \cdots$};
\draw (4,3.5) node[above] {$\scs \alpha_{n}$};
\draw (0.75,2) node {$\scs \gamma_{1}$};
\draw (2.5,0.5) node {$\scs \gamma_{n-2}$};
\draw (3.5,0) node[right] {$\scs \gamma_{n-1}$};
\begin{scope}[decoration={markings, mark=at position 0.75 with {\arrow{>}}}]
    \draw[ultra thick, black, postaction=decorate] (0.8125,-2.75) to (0,-3.5);
    \draw[ultra thick, black, postaction=decorate] (0.8125,-2.75) to (1,-3.5);
    \draw[ultra thick, black, postaction=decorate] (1.625,-2) to (0.8125,-2.75);
    \draw[ultra thick, black, postaction=decorate] (1.625,-2) to (2,-3.5);
    \draw[ultra thick, black, dotted] (2.4375,-1.25) to (1.625,-2);
    \draw[ultra thick, black, postaction=decorate] (2.4375,-1.25) to (3,-3.5);
    \draw[ultra thick, black, postaction=decorate] (3.25,-0.5) to (2.4375,-1.25);
    \draw[ultra thick, black, postaction=decorate] (3.25,-0.5) to (4,-3.5);
\end{scope}
\draw (0,-3.5) node[below] {$\scs \alpha_1$};
\draw (1,-3.5) node[below] {$\scs \alpha_2$};
\draw (2,-3.5) node[below] {$\scs \alpha_3$};
\draw (3,-3.5) node[below] {$\scs \cdots$};
\draw (4,-3.5) node[below] {$\scs \alpha_{n}$};
\draw (0.75,-2) node {$\scs \gamma_{1}$};
\draw (2.5,-0.5) node {$\scs \gamma_{n-2}$};
\end{tikzpicture} }
\end{align*}
The inductive step immediately follows from the previous equality.
\end{proof}

We explain that Lemma~\ref{lemma:fusion-identity} fails without the $\sum_{i=1}^j \alpha_i$ generic assumption. Recall that one may only guarantee a semisimple decomposition if the object lives in the generic part of the category. As a counterexample, consider the case when $r=2$ and choose any generic value $\alpha$. Then, $V_\alpha$ and $V_{-\alpha}$ are generic objects whereas $V_\alpha \otimes V_{-\alpha}\in\mathcal{C}_{\overline{0}}$ will live in the singular part of the category. Their tensor product, by~\cite[Corollary 6.3]{CGP2}, is equivalent to $V_\alpha \otimes V_{-\alpha}\cong V_0 \otimes V_0$. Since $V_0 = S_{1}$, using~\cite[Lemma 8.1]{CGP2}, we can deduce that their tensor product is $V_0 \otimes V_0 \cong P_0$ non-semisimple. Therefore we cannot derive an identity involving ribbon graphs colored by simple objects from the generic part of the category like above.

\begin{lemma}[Mirror Reflection]  \label{lemma:fusion-channel}
Suppose that the following graph is a nonzero morphism in $\mathrm{Hom}(V_{\psi_0}, V_{\gamma_0})$:
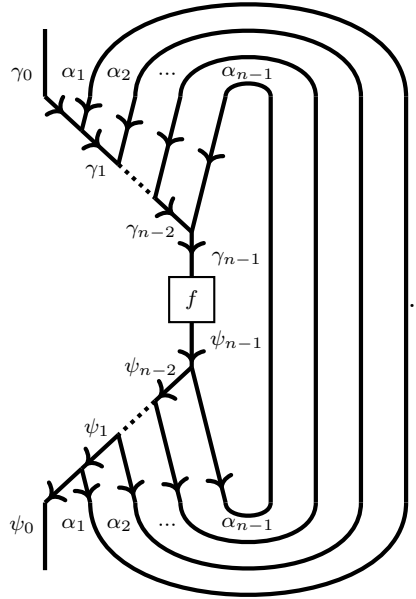
\begin{equation*}
\hackcenter{\begin{tikzpicture}[baseline=0, thick, scale=0.6, shift={(0,0)}]
\begin{scope}[decoration={markings, mark=at position 0.5 with {\arrow{>}}}]
    \draw[ultra thick, black, postaction=decorate] (0,4.5) to (0.8125,3.75);
    \draw[ultra thick, black, postaction=decorate] (1,4.5) to (0.8125,3.75);
    \draw[ultra thick, black, postaction=decorate] (0.8125,3.75) to (1.625,3);
    \draw[ultra thick, black, postaction=decorate] (2,4.5) to (1.625,3);
    \draw[ultra thick, black, dotted] (1.625, 3) to (2.4375, 2.25);
    \draw[ultra thick, black, postaction=decorate] (3, 4.5) to (2.4375, 2.25);
    \draw[ultra thick, black, postaction=decorate] (2.4375,2.25) to (3.25,1.5);
    \draw[ultra thick, black, postaction=decorate] (4, 4.5) to (3.25,1.5);
    \draw[ultra thick, black, postaction=decorate] (3.25,1.5) to (3.25,0.5);
\end{scope}
\draw (-0.5,5) node {$\scs \gamma_0$};
\draw (0.65,5) node {$\scs \alpha_1$};
\draw (1.65,5) node {$\scs \alpha_2$};
\draw (2.75,5) node {$\scs \cdots$};
\draw (4.5,5) node {$\scs \alpha_{n-1}$};
\draw (1.15,2.85) node {$\scs \gamma_{1}$};
\draw (2.35,1.5) node {$\scs \gamma_{n-2}$};
\draw (4.25,0.9) node {$\scs \gamma_{n-1}$};
\draw (2.75,-0.5) rectangle (3.75,0.5);
\draw (3.25,0) node {$\scs f$};
\begin{scope}[decoration={markings, mark=at position 0.9 with {\arrow{>}}}]
    \draw[ultra thick, black, postaction=decorate] (3.25,-0.5) to (3.25,-1.5);
    \draw[ultra thick, black, postaction=decorate] (0.8125,-3.75) to (0,-4.5);
    \draw[ultra thick, black, postaction=decorate] (0.8125,-3.75) to (1,-4.5);
    \draw[ultra thick, black, postaction=decorate] (1.625,-3) to (0.8125,-3.75);
    \draw[ultra thick, black, postaction=decorate] (1.625,-3) to (2,-4.5);
    \draw[ultra thick, black, dotted] (2.4375,-2.25) to (1.625,-3);
    \draw[ultra thick, black, postaction=decorate] (2.4375,-2.25) to (3,-4.5);
    \draw[ultra thick, black, postaction=decorate] (3.25,-1.5) to (2.4375,-2.25);
    \draw[ultra thick, black, postaction=decorate] (3.25,-1.5) to (4,-4.5);
\end{scope}
\draw (-0.5,-5) node {$\scs \psi_0$};
\draw (0.65,-5) node {$\scs \alpha_1$};
\draw (1.65,-5) node {$\scs \alpha_2$};
\draw (2.75,-5) node {$\scs \cdots$};
\draw (4.5,-5) node {$\scs \alpha_{n-1}$};
\draw (1.15,-2.85) node {$\scs \psi_{1}$};
\draw (2.35,-1.5) node {$\scs \psi_{n-2}$};
\draw (4.25,-0.9) node {$\scs \psi_{n-1}$};
\draw[ultra thick, black] (4,4.5) to [out=90,in=90] (5,4.5);
\draw[ultra thick, black] (4,-4.5) to [out=-90,in=-90] (5,-4.5);
\draw[ultra thick, black] (5,4.5) to (5,-4.5);
\draw[ultra thick, black] (3,4.5) to [out=90,in=90] (6,4.5);
\draw[ultra thick, black] (3,-4.5) to [out=-90,in=-90] (6,-4.5);
\draw[ultra thick, black] (6,4.5) to (6,-4.5);
\draw[ultra thick, black] (2,4.5) to [out=90,in=90] (7,4.5);
\draw[ultra thick, black] (2,-4.5) to [out=-90,in=-90] (7,-4.5);
\draw[ultra thick, black] (7,4.5) to (7,-4.5);
\draw[ultra thick, black] (1,4.5) to [out=90,in=90] (8,4.5);
\draw[ultra thick, black] (1,-4.5) to [out=-90,in=-90] (8,-4.5);
\draw[ultra thick, black] (8,4.5) to (8,-4.5);
\draw[ultra thick, black] (0,4.5) to (0,6);
\draw[ultra thick, black] (0,-4.5) to (0,-6);
\end{tikzpicture} }.
\end{equation*}
Then, $\gamma_i = \psi_i$ for $i=0,\dots, n-1$.
\end{lemma}

\begin{proof}
The diagram in the statement can be isotoped to the following picture
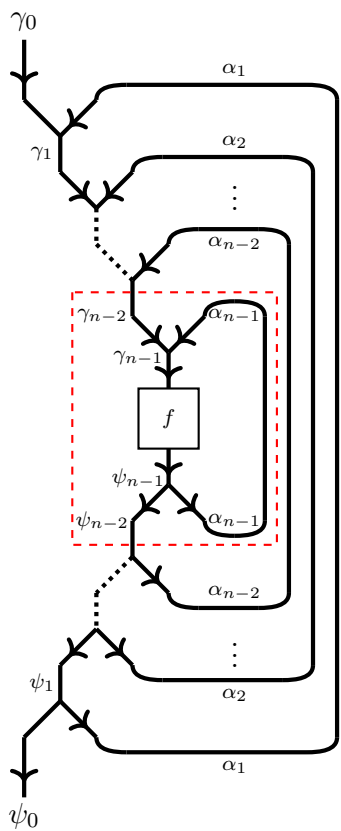
\begin{equation*}
\hackcenter{\begin{tikzpicture}[baseline=0, thick, scale=0.8, shift={(0,0)}]
\draw (-0.5,-0.5) rectangle (0.5,0.5);
\draw (0,0) node {$\scs f$};
\begin{scope}[decoration={markings, mark=at position 0.75 with {\arrow{>}}}]
\begin{scope}[shift={(-0.6,1.7)}]
    \draw[ultra thick, black, postaction=decorate] (0,0) to (0.6,-0.6);
    \draw[ultra thick, black, postaction=decorate] (1.2,0) to (0.6,- 0.6);
    \draw[ultra thick, black, postaction=decorate] (0.6,-0.6) to (0.6,-1.2);
    \draw[ultra thick, black] (2.2,0) to [out=90, in=0] (1.7,0.25);
    \draw[ultra thick, black] (1.7,0.25) to [out=180, in=90] (1.2,0);
\end{scope}
\begin{scope}[shift={(-1.2,2.9)}]
    \draw[ultra thick, black, dotted] (0,0) to (0.6,-0.6);
    \draw[ultra thick, black, postaction=decorate] (1.2,0) to (0.6,- 0.6);
    \draw[ultra thick, black] (0.6,-0.6) to (0.6,-1.2);
    \draw[ultra thick, black] (3.2,0) to [out=90, in=0] (2.7,0.25);
    \draw[ultra thick, black] (2.7,0.25) to (1.7,0.25);
    \draw[ultra thick, black] (1.7,0.25) to [out=180, in=90] (1.2,0);
\end{scope}
\begin{scope}[shift={(-1.8,4.1)}]
    \draw[ultra thick, black, postaction=decorate] (0,0) to (0.6,-0.6);
    \draw[ultra thick, black, postaction=decorate] (1.2,0) to (0.6,- 0.6);
    \draw[ultra thick, black, dotted] (0.6,-0.6) to (0.6,-1.2);
    \draw[ultra thick, black] (4.2,0) to [out=90, in=0] (3.7,0.25);
    \draw[ultra thick, black] (3.7,0.25) to (1.7,0.25);
    \draw[ultra thick, black] (1.7,0.25) to [out=180, in=90] (1.2,0);
\end{scope}
\begin{scope}[shift={(-2.4,5.3)}]
    \draw[ultra thick, black, postaction=decorate] (0,1) to (0,0);
    \draw[ultra thick, black] (0,0) to (0.6,-0.6);
    \draw[ultra thick, black, postaction=decorate] (1.2,0) to (0.6,- 0.6);
    \draw[ultra thick, black] (0.6,-0.6) to (0.6,-1.2);
    \draw[ultra thick, black] (5.2,0) to [out=90, in=0] (4.7,0.25);
    \draw[ultra thick, black] (4.7,0.25) to (1.7,0.25);
    \draw[ultra thick, black] (1.7,0.25) to [out=180, in=90] (1.2,0);
\end{scope}
\end{scope}
\begin{scope}[decoration={markings, mark=at position 0.75 with {\arrow{>}}}]
\begin{scope}[shift={(-0.6,-0.5)}]
    \draw[ultra thick, black, postaction=decorate] (0.6,0) to (0.6,-0.6);
    \draw[ultra thick, black, postaction=decorate] (0.6,-0.6) to (0,-1.2);
    \draw[ultra thick, black, postaction=decorate] (0.6,-0.6) to (1.2,-1.2);
    \draw[ultra thick, black] (2.2,-1.2) to [out=-90, in=0] (1.7,-1.45);
    \draw[ultra thick, black] (1.7,-1.45) to [out=-180, in=-90] (1.2,-1.2);
\end{scope}
\begin{scope}[shift={(-1.2,-1.7)}]
    \draw[ultra thick, black] (0.6,0) to (0.6,-0.6);
    \draw[ultra thick, black, dotted] (0.6,-0.6) to (0,-1.2);
    \draw[ultra thick, black, postaction=decorate] (0.6,-0.6) to (1.2,-1.2);
    \draw[ultra thick, black] (3.2,-1.2) to [out=-90, in=0] (2.7,-1.45);
    \draw[ultra thick, black] (2.7,-1.45) to (1.7,-1.45);
    \draw[ultra thick, black] (1.7,-1.45) to [out=-180, in=-90] (1.2,-1.2);
\end{scope}
\begin{scope}[shift={(-1.8,-2.9)}]
    \draw[ultra thick, black, dotted] (0.6,0) to (0.6,-0.6);
    \draw[ultra thick, black, postaction=decorate] (0.6,-0.6) to (0,-1.2);
    \draw[ultra thick, black, postaction=decorate] (0.6,-0.6) to (1.2,-1.2);
    \draw[ultra thick, black] (4.2,-1.2) to [out=-90, in=0] (3.7,-1.45);
    \draw[ultra thick, black] (3.7,-1.45) to (1.7,-1.45);
    \draw[ultra thick, black] (1.7,-1.45) to [out=-180, in=-90] (1.2,-1.2);
\end{scope}
\begin{scope}[shift={(-2.4,-4.1)}]
    \draw[ultra thick, black] (0.6,0) to (0.6,-0.6);
    \draw[ultra thick, black] (0.6,-0.6) to (0,-1.2);
    \draw[ultra thick, black, postaction=decorate] (0,-1.2) to (0,-2.2);
    \draw[ultra thick, black, postaction=decorate] (0.6,-0.6) to (1.2,-1.2);
    \draw[ultra thick, black] (5.2,-1.2) to [out=-90, in=0] (4.7,-1.45);
    \draw[ultra thick, black] (4.7,-1.45) to (1.7,-1.45);
    \draw[ultra thick, black] (1.7,-1.45) to [out=-180, in=-90] (1.2,-1.2);
\end{scope}
\end{scope}
\draw (-2.4,6.6) node {$\gamma_0$};
\draw (1.1,5.8) node {$\scs \alpha_1$};
\draw (1.1,4.6) node {$\scs \alpha_2$};
\draw (1.1,3.8) node {$\scs \vdots$};
\draw (1.1,2.9) node {$\scs \alpha_{n-2}$};
\draw (1.1,1.7) node {$\scs \alpha_{n-1}$};
\draw (-2.1,4.4) node {$\scs \gamma_{1}$};
\draw (-1.1,1.7) node {$\scs \gamma_{n-2}$};
\draw (-0.5,1) node {$\scs \gamma_{n-1}$};
\draw[ultra thick, black] (1.6, 1.7) to (1.6, -1.7);
\draw[ultra thick, black] (2, 2.9) to (2, -2.9);
\draw[ultra thick, black] (2.4, 4.1) to (2.4, -4.1);
\draw[ultra thick, black] (2.8, 5.3) to (2.8, -5.3);
\draw (-2.4,-6.6) node {$\psi_0$};
\draw (1.1,-5.8) node {$\scs \alpha_1$};
\draw (1.1,-4.6) node {$\scs \alpha_2$};
\draw (1.1,-3.8) node {$\scs \vdots$};
\draw (1.1,-2.9) node {$\scs \alpha_{n-2}$};
\draw (1.1,-1.7) node {$\scs \alpha_{n-1}$};
\draw (-2.1,-4.4) node {$\scs \psi_{1}$};
\draw (-1.1,-1.7) node {$\scs \psi_{n-2}$};
\draw (-0.5,-1) node {$\scs \psi_{n-1}$};
\draw[red,dashed] (-1.6,2.1) rectangle (1.8,-2.1);
\end{tikzpicture} }
\end{equation*}
by pulling the coevaluation
\begin{tikzpicture}[baseline=0, thick, scale=0.5, shift={(0,0)}]
\begin{scope}[decoration={markings, mark=at position 1 with {\arrow{>}}}]
    \draw[postaction=decorate] (0,0.5) to [out=-90, in=180] (0.5,0) to [out=0, in=-90] (1,0.5);
\end{scope}
\end{tikzpicture}~and evaluation 
\begin{tikzpicture}[baseline=0, thick, scale=0.5, shift={(0,0)}]
\begin{scope}[decoration={markings, mark=at position 1 with {\arrow{>}}}]
    \draw[postaction=decorate] (1,0) to [out=90, in=0] (0.5,0.5) to [out=180, in=90] (0,0);
\end{scope}
\end{tikzpicture}
~closer to the $f$ coupon. This diagram makes it clear that we can extract the intermediate morphisms $V_{\psi_i} \rightarrow V_{\gamma_i}$. For example, see the red-dashed rectangle to obtain $V_{\psi_{n-2}} \rightarrow V_{\gamma_{n-2}}$. Observe that these intermediate morphisms must be nonzero since we assumed that the full diagram evaluates to a nonzero morphism in $\mathrm{Hom}(V_{\psi_0}, V_{\gamma_0})$ under the ribbon functor. By the simplicity of $V_\alpha$ when $\alpha$ is generic, Schur's Lemma on a nonzero morphism between simples $V_{\psi_i} \rightarrow V_{\gamma_i}$ then informs us that $\gamma_i = \psi_i$ for $i=0, \dots, n-1$.
\end{proof}

\begin{lemma}[Braid Pop] \label{lemma:character-identity}
For any $\beta_n \in \Br_n$, we have the following identity:
\begin{equation*}
\sum\limits_{(\alpha, \gamma_1, \dots, \gamma_{n-1})\in P_{n,m,\alpha}^r} \md(\gamma_1)\cdots \md(\gamma_{n-1})
~\hackcenter{\begin{tikzpicture}[baseline=0, thick, scale=0.5, shift={(0,0)}]
\draw (0.5,.5) rectangle (5.5,-.5);
\draw (3,0) node {$\scs \beta_n$};
\begin{scope}[shift={(1,-5.5)}]
\begin{scope}[decoration={markings, mark=at position 0.6 with {\arrow{>}}}]
    \draw[ultra thick, black, postaction=decorate] (0,4.5) to (0.8125,3.75);
    \draw[ultra thick, black, postaction=decorate] (1,4.5) to (0.8125,3.75);
    \draw[ultra thick, black, postaction=decorate] (0.8125,3.75) to (1.625,3);
    \draw[ultra thick, black, postaction=decorate] (2,4.5) to (1.625,3);
    \draw[ultra thick, black, dotted] (1.625, 3) to (2.4375, 2.25);
    \draw[ultra thick, black, postaction=decorate] (3, 4.5) to (2.4375, 2.25);
    \draw[ultra thick, black, postaction=decorate] (2.4375,2.25) to (3.25,1.5);
    \draw[ultra thick, black, postaction=decorate] (4, 4.5) to (3.25,1.5);
    \draw[ultra thick, black, postaction=decorate] (3.25,1.5) to (3.25,0.5);
\end{scope}
\end{scope}
\begin{scope}[shift={(1,5.5)}]
\begin{scope}[decoration={markings, mark=at position 0.9 with {\arrow{>}}}]
    \draw[ultra thick, black, postaction=decorate] (3.25,-0.5) to (3.25,-1.5);
    \draw[ultra thick, black, postaction=decorate] (0.8125,-3.75) to (0,-4.5);
    \draw[ultra thick, black, postaction=decorate] (0.8125,-3.75) to (1,-4.5);
    \draw[ultra thick, black, postaction=decorate] (1.625,-3) to (0.8125,-3.75);
    \draw[ultra thick, black, postaction=decorate] (1.625,-3) to (2,-4.5);
    \draw[ultra thick, black, dotted] (2.4375,-2.25) to (1.625,-3);
    \draw[ultra thick, black, postaction=decorate] (2.4375,-2.25) to (3,-4.5);
    \draw[ultra thick, black, postaction=decorate] (3.25,-1.5) to (2.4375,-2.25);
    \draw[ultra thick, black, postaction=decorate] (3.25,-1.5) to (4,-4.5);
\end{scope}
\end{scope}
\draw[ultra thick, black] (1,0.5) to (1,1);
\draw[ultra thick, black] (2,0.5) to (2,1);
\draw[ultra thick, black] (3,0.5) to (3,1);
\draw[ultra thick, black] (4,0.5) to (4,1);
\draw[ultra thick, black] (5,0.5) to (5,1);
\draw[ultra thick, black] (1,-0.5) to (1,-1);
\draw[ultra thick, black] (2,-0.5) to (2,-1);
\draw[ultra thick, black] (3,-0.5) to (3,-1);
\draw[ultra thick, black] (4,-0.5) to (4,-1);
\draw[ultra thick, black] (5,-0.5) to (5,-1);
\draw (0.7,0.75) node {$\scs \alpha$};
\draw (1.7,0.75) node {$\scs \alpha$};
\draw (2.7,0.75) node {$\scs \alpha$};
\draw (3.5,0.75) node {$\scs \cdots$};
\draw (4.3,0.75) node {$\scs \alpha$};
\draw (5.3,0.75) node {$\scs \alpha$};
\draw (0.7,-0.75) node {$\scs \alpha$};
\draw (1.7,-0.75) node {$\scs \alpha$};
\draw (2.7,-0.75) node {$\scs \alpha$};
\draw (3.5,-0.75) node {$\scs \cdots$};
\draw (4.3,-0.75) node {$\scs \alpha$};
\draw (5.3,-0.75) node {$\scs \alpha$};
\draw (2.2,-2.5) node {$\scs \gamma_1$};
\draw (3.3,-4) node {$\scs \gamma_{n-2}$};
\draw (4.25,-5.25) node {$\scs \gamma_{n-1}$};
\draw (2.2,2.5) node {$\scs \gamma_1$};
\draw (3.3,4) node {$\scs \gamma_{n-2}$};
\draw (4.25,5.25) node {$\scs \gamma_{n-1}$};
\end{tikzpicture} }
~=~
\chi_{\rho_{\gamma_{n-1}}}(\beta_n)
\hackcenter{\begin{tikzpicture}[baseline=0, thick, scale=0.5, shift={(0,0)}]
\begin{scope}[shift={(1,-5.5)}]
\begin{scope}[decoration={markings, mark=at position 0.6 with {\arrow{>}}}]
    \draw[ultra thick, black, postaction=decorate] (0,5) to (0,-5);
\end{scope}
\draw (1,0) node {$\gamma_{n-1}$};
\end{scope}
\end{tikzpicture} }
.
\end{equation*}
where $\chi_{\rho_{\gamma_{n-1}}}$ is the character of $\rho_{\gamma_{n-1}}: \Br_n \rightarrow \End(\mathcal{H}_{n,m,\alpha}^r)$.
\end{lemma}

\begin{proof}
Recall that our conventions aid us the bubble pop move:
\begin{equation*}
\hackcenter{\begin{tikzpicture}[scale=1.1]
\begin{scope}[decoration={markings, mark=at position 0.6 with {\arrow{>}}}]
  \draw[ultra thick, black, postaction=decorate] (0,0) to (.6,-.6);
  \draw[ultra thick, black, postaction=decorate] (.6,-.6) to (.6,-1.2);
  \draw[ultra thick, black, postaction=decorate] (1.2,0) to (.6,-.6);
  \draw[ultra thick, black] (.6,.6) to (0,0);
  \draw[ultra thick, black, postaction=decorate] (.6,1.2) to (.6,.6);
  \draw[ultra thick, black] (.6,.6) to (1.2,0);
\end{scope}
   \node at (0,0.3) {$\alpha$};
   \node at (1.2,0.3) {$\beta$};
   \node at (.6,-1.4) {$\gamma$};
   \node at (.6,1.4) {$\gamma'$};
\end{tikzpicture} }~=~ \frac{\delta_{\gamma\gamma'}}{\md(\gamma)}~\id_{V_\gamma}
\quad
~\Rightarrow~
\quad
\hackcenter{\begin{tikzpicture}[scale=0.8]
\begin{scope}[decoration={markings, mark=at position 0.75 with {\arrow{>}}}]
  \draw[ultra thick, black, postaction=decorate] (0,0) to (0.6,-0.6);
  \draw[ultra thick, black, postaction=decorate] (0.6,-0.6) to (1.2,-1.2);
  \draw[ultra thick, dotted] (1.2,-1.2) to (1.8,-1.8);
  \draw[ultra thick, black, postaction=decorate] (1.8,-1.8) to (2.4,-2.4);
  \draw[ultra thick, black, postaction=decorate] (1.2,0) to (.6,-.6);
  \draw[ultra thick, black, postaction=decorate] (2.4,0) to (1.2,-1.2);
  \draw[ultra thick, black, postaction=decorate] (3.6,0) to (1.8,-1.8);
  \draw[ultra thick, black, postaction=decorate] (4.8,0) to (2.4,-2.4);
  \draw[ultra thick, black, postaction=decorate] (2.4,-2.4) to (2.4,-3.3);
  \draw[ultra thick, black] (0.6,0.6) to (0,0);
  \draw[ultra thick, black, postaction=decorate] (1.2,1.2) to (0.6,0.6);
  \draw[ultra thick, dotted] (1.8,1.8) to (1.2,1.2);
  \draw[ultra thick, black, postaction=decorate] (2.4,2.4) to (1.8,1.8);
  \draw[ultra thick, black] (.6,.6) to (1.2,0);
  \draw[ultra thick, black] (1.2,1.2) to (2.4,0);
  \draw[ultra thick, black] (1.8,1.8) to (3.6,0);
  \draw[ultra thick, black] (2.4,2.4) to (4.8,0);
  \draw[ultra thick, black, postaction=decorate] (2.4,3.3) to (2.4,2.4);
\end{scope}
   \node at (-0.3,0) {$\alpha$};
   \node at (0.8,0) {$\alpha$};
   \node at (2,0) {$\alpha$};
   \node at (2.9,0) {$\cdots$};
   \node at (3.9,0) {$\alpha$};
   \node at (5.1,0) {$\alpha$};
   \node at (.7,-1.25) {$\psi_1$};
   \node at (1.7,-2.4) {$\psi_{n-2}$};
   \node at (2.4,-3.5) {$\psi_{n-1}$};
   \node at (.7,1.25) {$\gamma_1$};
   \node at (1.7,2.4) {$\gamma_{n-2}$};
   \node at (2.4,3.5) {$\gamma_{n-1}$};
\end{tikzpicture} }
~=~
\frac{\delta_{\gamma_1\psi_1} \cdots \delta_{\gamma_{n-1}\psi_{n-1}}}{\md(\gamma_1) \cdots \md(\gamma_{n-1})}~\id_{V_{\gamma_{n-1}}}
\end{equation*}
where $\delta_{\gamma\gamma'}$ is the Kronecker delta.
\end{proof}

\begin{lemma}[Pruning]  \label{lemma:isotoped-bubble}
For any path $\mathbf{p}=(\alpha,\gamma_1, \dots, \gamma_{n-1})\in P_{n,m,\alpha}^r$, we have the following identity:
\begin{equation}
\hackcenter{\begin{tikzpicture}[baseline=0, thick, scale=0.6, shift={(0,0)}]
\begin{scope}[shift={(0,-0.5)}]
\begin{scope}[decoration={markings, mark=at position 0.5 with {\arrow{>}}}]
    \draw[ultra thick, black, postaction=decorate] (0,4.5) to (0.8125,3.75);
    \draw[ultra thick, black, postaction=decorate] (1,4.5) to (0.8125,3.75);
    \draw[ultra thick, black, postaction=decorate] (0.8125,3.75) to (1.625,3);
    \draw[ultra thick, black, postaction=decorate] (2,4.5) to (1.625,3);
    \draw[ultra thick, black, dotted] (1.625, 3) to (2.4375, 2.25);
    \draw[ultra thick, black, postaction=decorate] (3, 4.5) to (2.4375, 2.25);
    \draw[ultra thick, black, postaction=decorate] (2.4375,2.25) to (3.25,1.5);
    \draw[ultra thick, black, postaction=decorate] (4, 4.5) to (3.25,1.5);
    \draw[ultra thick, black, postaction=decorate] (3.25,1.5) to (3.25,-0.5);
\end{scope}
\draw (-0.5,5) node {$\scs \alpha$};
\draw (0.65,5) node {$\scs \alpha$};
\draw (1.65,5) node {$\scs \alpha$};
\draw (2.75,5) node {$\scs \cdots$};
\draw (4.5,4.5) node {$\scs \alpha$};
\draw (1.15,2.85) node {$\scs \gamma_{1}$};
\draw (2.35,1.5) node {$\scs \gamma_{n-2}$};
\draw (4.25,0.9) node {$\scs \gamma_{n-1}$};
\end{scope}
\begin{scope}[shift={(0,0.5)}]
\begin{scope}[decoration={markings, mark=at position 0.9 with {\arrow{>}}}]
    \draw[ultra thick, black, postaction=decorate] (0.8125,-3.75) to (0,-4.5);
    \draw[ultra thick, black, postaction=decorate] (0.8125,-3.75) to (1,-4.5);
    \draw[ultra thick, black, postaction=decorate] (1.625,-3) to (0.8125,-3.75);
    \draw[ultra thick, black, postaction=decorate] (1.625,-3) to (2,-4.5);
    \draw[ultra thick, black, dotted] (2.4375,-2.25) to (1.625,-3);
    \draw[ultra thick, black, postaction=decorate] (2.4375,-2.25) to (3,-4.5);
    \draw[ultra thick, black, postaction=decorate] (3.25,-1.5) to (2.4375,-2.25);
    \draw[ultra thick, black, postaction=decorate] (3.25,-1.5) to (4,-4.5);
\end{scope}
\draw (-0.5,-5) node {$\scs \alpha$};
\draw (0.65,-5) node {$\scs \alpha$};
\draw (1.65,-5) node {$\scs \alpha$};
\draw (2.75,-5) node {$\scs \cdots$};
\draw (4.5,-4.5) node {$\scs \alpha$};
\draw (1.15,-2.85) node {$\scs \gamma_{1}$};
\draw (2.35,-1.5) node {$\scs \gamma_{n-2}$};
\end{scope}
\draw[ultra thick, black] (4,4) to [out=90,in=90] (5,4);
\draw[ultra thick, black] (4,-4) to [out=-90,in=-90] (5,-4);
\draw[ultra thick, black] (5,4) to (5,-4);
\draw[ultra thick, black] (3,4) to [out=90,in=90] (6,4);
\draw[ultra thick, black] (3,-4) to [out=-90,in=-90] (6,-4);
\draw[ultra thick, black] (6,4) to (6,-4);
\draw[ultra thick, black] (2,4) to [out=90,in=90] (7,4);
\draw[ultra thick, black] (2,-4) to [out=-90,in=-90] (7,-4);
\draw[ultra thick, black] (7,4) to (7,-4);
\draw[ultra thick, black] (1,4) to [out=90,in=90] (8,4);
\draw[ultra thick, black] (1,-4) to [out=-90,in=-90] (8,-4);
\draw[ultra thick, black] (8,4) to (8,-4);
\draw[ultra thick, black] (0,4) to (0,5.5);
\draw[ultra thick, black] (0,-4) to (0,-5.5);
\end{tikzpicture} }~=~\frac{1}{\md(\alpha)\md(\gamma_1) \cdots \md(\gamma_{n-2})} 
\hackcenter{\begin{tikzpicture}[baseline=0, thick, scale=0.5, shift={(0,0)}]
\begin{scope}[shift={(1,-5.5)}]
\begin{scope}[decoration={markings, mark=at position 0.6 with {\arrow{>}}}]
    \draw[ultra thick, black, postaction=decorate] (0,6) to (0,-6);
\end{scope}
\draw (0.5,0) node {$\alpha$};
\end{scope}
\end{tikzpicture} } .
\end{equation}
\end{lemma}

\begin{proof}
The proof is obvious once the diagram is isotoped to
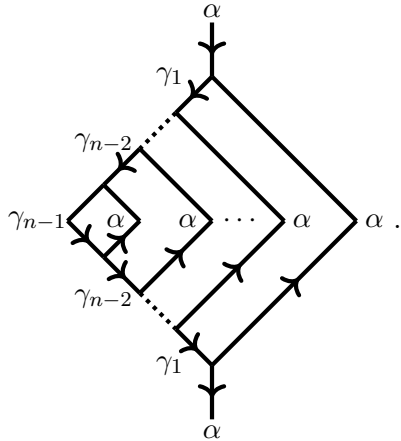
\begin{equation*}
\hackcenter{\begin{tikzpicture}[scale=0.8]
\begin{scope}[decoration={markings, mark=at position 0.6 with {\arrow{>}}}]
  \draw[ultra thick, black, postaction=decorate] (0,0) to (0.6,-0.6);
  \draw[ultra thick, black, postaction=decorate] (0.6,-0.6) to (1.2,-1.2);
  \draw[ultra thick, dotted] (1.2,-1.2) to (1.8,-1.8);
  \draw[ultra thick, black, postaction=decorate] (1.8,-1.8) to (2.4,-2.4);
  \draw[ultra thick, black, postaction=decorate] (.6,-.6) to (1.2,0);
  \draw[ultra thick, black, postaction=decorate] (1.2,-1.2) to (2.4,0);
  \draw[ultra thick, black, postaction=decorate] (1.8,-1.8) to (3.6,0);
  \draw[ultra thick, black, postaction=decorate] (2.4,-2.4) to (4.8,0);
  \draw[ultra thick, black, postaction=decorate] (2.4,-2.4) to (2.4,-3.3);
  \draw[ultra thick, black] (0.6,0.6) to (0,0);
  \draw[ultra thick, black, postaction=decorate] (1.2,1.2) to (0.6,0.6);
  \draw[ultra thick, dotted] (1.8,1.8) to (1.2,1.2);
  \draw[ultra thick, black, postaction=decorate] (2.4,2.4) to (1.8,1.8);
  \draw[ultra thick, black] (.6,.6) to (1.2,0);
  \draw[ultra thick, black] (1.2,1.2) to (2.4,0);
  \draw[ultra thick, black] (1.8,1.8) to (3.6,0);
  \draw[ultra thick, black] (2.4,2.4) to (4.8,0);
  \draw[ultra thick, black, postaction=decorate] (2.4,3.3) to (2.4,2.4);
\end{scope}
   \node at (-0.5,0) {$\gamma_{n-1}$};
   \node at (0.8,0) {$\alpha$};
   \node at (2,0) {$\alpha$};
   \node at (2.9,0) {$\cdots$};
   \node at (3.9,0) {$\alpha$};
   \node at (5.1,0) {$\alpha$};
   \node at (.6,-1.3) {$\gamma_{n-2}$};
   \node at (1.7,-2.4) {$\gamma_1$};
   \node at (2.4,-3.5) {$\alpha$};
   \node at (.6,1.3) {$\gamma_{n-2}$};
   \node at (1.7,2.4) {$\gamma_{1}$};
   \node at (2.4,3.5) {$\alpha$};
\end{tikzpicture} }.
\end{equation*}
\end{proof}

\subsection{Quantum Knot Invariants} We are now ready to study quantum knot invariants using the graphical calculi developed above.

\begin{theorem} \label{thm:CGP-formula}
Let $K$ be a knot represented as a (Markov) closure of an $n$-braid $\beta_n$. Then,
\begin{equation}
N_r^{\alpha}(K)=\sum\limits_{\gamma \in n\alpha + H_{(n-1)(r-1)+1}}\md(\gamma) \chi_{\rho_\gamma}(\beta_n).
\end{equation}
\end{theorem}

\begin{proof}
Let $K_\alpha$ be the ribbon graph of the knot $K$ colored by $V_\alpha$. Recall the definition of the renormalized quantum knot invariant~\cite{GPT09}:
\begin{equation} \label{eq:definition}
N_r^\alpha(K) =\mt(K_\alpha)= \md(\alpha)~\left\langle \hackcenter{\begin{tikzpicture}[baseline=0, thick, scale=0.75, shift={(0,0)}]
\draw (0.5,.5) rectangle (4.5,-.5);
\draw (2.5,0) node {$\beta_n$};
\draw (3,1) node {$\cdots$};
\draw (3,-1) node {$\cdots$};
\draw (5.5,0) node {$\cdots$};
\draw (1,3) node[above] {$\alpha$};
\draw (1,-3) node[below] {$\alpha$};
\draw (2,1) node[left] {$\alpha$};
\draw (4,1) node[left] {$\alpha$};
\draw[ultra thick, black] (4,0.5) to (4, 1);
\draw[ultra thick, black] (4,1) to [out=90,in=90] (5,1);
\draw[ultra thick, black] (4,-1) to [out=-90,in=-90] (5,-1);
\draw[ultra thick, black] (2,0.5) to (2, 1.5);
\draw[ultra thick, black] (2,1.5) to [out=90,in=90] (6,1.5);
\draw[ultra thick, black] (2,-1.5) to [out=-90,in=-90] (6,-1.5);
\draw[ultra thick, black] (2,-0.5) to (2, -1.5);
\draw[ultra thick, black] (4,-0.5) to (4, -1);
\begin{scope}[decoration={markings, mark=at position 0.5 with {\arrow{>}}}]
    \draw[ultra thick, black, postaction=decorate] (1,3) to (1, 0.5);
    \draw[ultra thick, black, postaction=decorate] (1,-0.5) to (1, -3);
    \draw[ultra thick, black, postaction=decorate] (5,-1) to (5,1);
    \draw[ultra thick, black, postaction=decorate] (6,-1.5) to (6,1.5);
\end{scope}
\end{tikzpicture} } \right\rangle \in \mathbb{C}.
\end{equation}
Applying Lemma~\ref{lemma:fusion-identity}, Lemma~\ref{lemma:fusion-channel}, Lemma~\ref{lemma:character-identity}, and then Lemma~\ref{lemma:isotoped-bubble} gives us the following identity
\begin{align*}
\hackcenter{\begin{tikzpicture}[baseline=0, thick, scale=0.6, shift={(0,0)}]
\draw (0.5,.5) rectangle (4.5,-.5);
\draw (2.5,0) node {$\beta_n$};
\draw (2.7,1) node {$\cdots$};
\draw (2.7,-1) node {$\cdots$};
\draw (5.5,0) node {$\cdots$};
\draw (1,3) node[above] {$\alpha$};
\draw (1,-3) node[below] {$\alpha$};
\draw (2,1) node[left] {$\alpha$};
\draw (4,1) node[left] {$\alpha$};
\draw[ultra thick, black] (4,0.5) to (4, 1);
\draw[ultra thick, black] (4,1) to [out=90,in=90] (5,1);
\draw[ultra thick, black] (4,-1) to [out=-90,in=-90] (5,-1);
\draw[ultra thick, black] (2,0.5) to (2, 1.5);
\draw[ultra thick, black] (2,1.5) to [out=90,in=90] (6,1.5);
\draw[ultra thick, black] (2,-1.5) to [out=-90,in=-90] (6,-1.5);
\draw[ultra thick, black] (2,-0.5) to (2, -1.5);
\draw[ultra thick, black] (4,-0.5) to (4, -1);
\begin{scope}[decoration={markings, mark=at position 0.5 with {\arrow{>}}}]
    \draw[ultra thick, black, postaction=decorate] (1,3) to (1, 0.5);
    \draw[ultra thick, black, postaction=decorate] (1,-0.5) to (1, -3);
    \draw[ultra thick, black, postaction=decorate] (5,-1) to (5,1);
    \draw[ultra thick, black, postaction=decorate] (6,-1.5) to (6,1.5);
\end{scope}
\end{tikzpicture} } &= \sum\limits_{\gamma_{n-1}\in n\alpha + H_{(n-1)(r-1)+1}}\frac{\md(\gamma_{n-1})}{\md(\alpha)}\chi_{\rho_{\gamma_{n-1}}}(\beta_n) \id_{V_\alpha}.
\end{align*}
\end{proof}
Although the proof relied on a specific choice of morphism conventions, the resulting formula is now independent of this choice because the traces of similar matrices are equal.

Next, we discuss the \emph{Akutsu-Deguchi-Ohtsuki (ADO) invariants}~\cite{ADO92}. The ADO invariants are the most studied non-semisimple quantum link invariants and they play a prominent role in quantum topology. For example, the Kashaev invariant, which is the special case of the ADO invariant, is the key ingredient for the statement of the Volume conjecture~\cite{CGP14}.

\begin{remark}
We clarify the terminology used in this paper. The original definition of the ADO invariant utilizes a state sum formulation~\cite{ADO92}. Murakami then reconstructed a framed version of the ADO invariant using the universal $R$-matrix from quantum groups~\cite{Mur08}. He suggested this invariant to be named the \emph{colored Alexander invariant}. At a similar timeframe, Geer and Patureau-Mirand introduced the notion of the modified dimensions under the pseudonym ``fake quantum dimensions'' in~\cite{GP08b} which were later used to introduce the renormalized quantum link invariant with Turaev in~\cite{GPT09}. They showed that their renormalized invariant $N_r^\alpha(K)$ recover the colored Alexander polynomial. In the literature, the ADO invariant and the colored Alexander invariant are often synonymous. Here, we specifically refer the ADO invariant as the unframed version of the colored Alexander polynomial without the modified dimension factor $\md(\alpha)$. That is, for a knot $K$ represented as a closure of an $n$-braid $\beta_n$,
\begin{equation}
    q^{-\frac{\alpha^2 - (r-1)^2}{2} writhe(\beta_n)}N_r^{\alpha}(K)=\md(\alpha)\operatorname{ADO}_r(K)|_{x=q^{2\alpha - 2}}
\end{equation}
where the coefficient $q^{-\frac{\alpha^2 - (r-1)^2}{2} writhe(\beta_n)}$ is the framing correction (see~\cite[Section 3.2]{Mur08}). 
\end{remark}

This brings us to the following Corollary.
\begin{corollary}
Let $K$ be a knot represented as a closure of an $n$-braid $\beta_n$. Then, 
\begin{equation}
\operatorname{ADO}_r(K)|_{x=q^{2\alpha - 2}}=\frac{q^{-\frac{\alpha^2 - (r-1)^2}{2} writhe(\beta_n)}}{\md(\alpha)}\sum\limits_{\gamma \in n\alpha + H_{(n-1)(r-1)+1}}\md(\gamma) \chi_{\rho_\gamma}(\beta_n).    
\end{equation}
In particular, when $r=2$, we have 
\begin{equation}
\Delta_{K}(q^{2\alpha - 2}) = \frac{q^{-\frac{\alpha^2 - 1}{2} writhe(\beta_n)}}{\md(\alpha)}\sum\limits_{\gamma \in n\alpha + H_{n}} \md(\gamma)\chi_{\rho_\gamma}(\beta_n)
\end{equation} where $\Delta_K(x)=\operatorname{ADO}_2(K)$ is the Alexander polynomial of $K$. 
\end{corollary}

\begin{theorem}[Ito's formula~\cite{Ito16}]\label{thm:Ito}
Let $K$ be a knot represented as a closure of an $n$-braid $\beta_n$. Then, the colored Alexander invariant formula is given by
\begin{equation}
N_r^{\alpha}(K)= q^{\frac{(\alpha+r-1)^2}{2}writhe(\beta_n)}\sum\limits_{m=0}^{(n-1)(r-1)}\md(n\alpha+(n-1)(r-1)-2m) \chi_{\mathcal{L}_{n,m,\alpha}^r}(\beta_n) .
\end{equation}
In particular,
\begin{equation}
\Delta_K(q^{2\alpha-2})=\frac{q^{(\alpha+1) writhe(\beta_n)}}{\md(\alpha)}\sum\limits_{m=0}^{n-1}\md(n\alpha+n-1-2m) \chi_{\mathcal{L}_{n,m,\alpha}^2}(\beta_n) 
\end{equation}
where $\Delta_K(x)\in \mathbb{Z}[x,x^{-1}]$ is the Alexander polynomial of $K$. 
\end{theorem}

\begin{proof}
Ito's colored Alexander invariant formula~\cite{Ito16} from homological representations follows immediately from combining our renormalized quantum knot invariant formula from $\alpha$-fusion trees (Theorem~\ref{thm:CGP-formula}) and the following identity
\begin{equation*}
q^{-\frac{(\alpha+r-1)^2}{2}writhe(\beta_n)} \chi_{\rho_{\gamma}}(\beta_n) = \chi_{\mathcal{L}_{n,m,\alpha}^r}(\beta_n)
\end{equation*}
obtained from the identification theorem between the space of $\alpha$-fusion trees and the $q$-specialized Lawrence representations (Theorem~\ref{thm:fusiontree-homological}).
\end{proof}

\section{Fusion Trees in the Hermitian Category}
We now study fusion trees in $\Hom$-spaces equipped with a non-degenerate Hermitian pairing. First, consider the subcategory $\mathcal{D}$ of $\mathcal{C}$ generated by the following set:
\begin{equation*}
\mathrm{A}=\left\{V_\alpha, S_n, P_i, \mathbb{C}_{a r}^H \mid \alpha \in(\mathbb{R} \backslash \mathbb{Z}) \cup r \mathbb{Z}, n, i \in\{0, \cdots, r-2\}, a \in \mathbb{Z}\right\}.
\end{equation*}
Let $\mathcal{C}^{\dagger}$ be the full subcategory of $\mathcal{C}$ whose objects are \emph{Hermitian $\overline{U}_{q}^H\mathfrak{sl}(2)$-modules} in the language of~\cite{GLPMS}. Then, let $\mathcal{D}^\dagger$ be the full subcategory of $\mathcal{C}^\dagger$ whose objects are in $\mathcal{D}$. It was shown in~\cite[Theorem 4.18]{GLPMS} that $\mathcal{D}^\dagger$ is a \emph{Hermitian ribbon category} in the sense of~\cite[Definition 3.1]{GLPMS}.

\subsection{Hermitian Data}
We restrict our attention to ribbon graphs colored by simple objects from the generic part of the category $\mathcal{D}^\dagger$. The following Lemma tells us that the $\Hom$-spaces we consider in $\mathcal{D}^\dagger$ are equipped with a non-degenerate Hermitian pairing which will be utilized for the rest of the paper.

\begin{lemma}{\cite[Lemma 4.20]{GLPMS}} \label{def:hermitian-pairing}
For any objects $V,W$ of $\mathcal{D}^\dagger$ with $V$ projective, the pairing
\begin{equation*} \langle f , g \rangle \mapsto \mt_V(f^\dagger g): \Hom_{\mathcal{D}^\dagger}(V,W) \times \Hom_{\mathcal{D}^\dagger}(V,W) \rightarrow \mathbb{C}\end{equation*}
is a non-degenerate Hermitian pairing.
\end{lemma}

A key ingredient in computing $(Y^{\alpha,\beta}_\gamma)^\dagger$ is the map $\mathrm{X}_{V_\alpha,V_\beta}$ defined in~\cite{GLPMS}. With respect to our conventions, we outline some useful properties.
\begin{proposition} \label{prop:X-map}
We have the following identity:
\begin{equation*}
\mathrm{X}_{V_{\alpha},V_{\beta}}Y_{\gamma}^{\alpha,\beta} = Y_{\gamma}^{\beta,\alpha}.
\end{equation*}
\end{proposition}

\begin{proof}
Recall that
\begin{equation}
\mathrm{X}_{V_{\alpha},V_{\beta}}=\left( \sqrt{\theta}_{V_\beta, V_\alpha} \right)^{-1} c_{V_\alpha, V_\beta} \left( \sqrt{\theta}_{V_\alpha} \otimes \sqrt{\theta}_{V_\beta} \right).
\end{equation}
The map $\sqrt{\theta}_{V_\alpha}$ can be extracted from~\cite[Table (28)]{GLPMS}. Thus, the computation of $\mathrm{X}_{V_{\alpha},V_{\beta}}$ amounts to solving for $\sqrt{\theta}_{V_\beta, V_\alpha}$. For generic values $\alpha$ and $\beta$, we have the equation
\begin{equation}
\left(\sqrt{\theta}_{V_\beta \otimes V_\alpha}\right)^{\pm 1} = \sum\limits_{\gamma \in \alpha + \beta + H_r} q^{\pm \frac{\gamma^2 - (r-1)^2}{4}} \md(\gamma) \hackcenter{\begin{tikzpicture}[scale=1.1]
\begin{scope}[decoration={markings, mark=at position 0.7 with {\arrow{>}}}]
  \draw[ultra thick, black, postaction=decorate] (0,.9) to (.6,.3);
  \draw[ultra thick, black, postaction=decorate] (.6,.3) to (.6,-.3);
  \draw[ultra thick, black, postaction=decorate] (1.2,.9) to (.6,0.3);
  \draw[ultra thick, black, postaction=decorate] (.6,-.3) to (0,-0.9);
  \draw[ultra thick, black, postaction=decorate] (.6,-.3) to (1.2,-0.9);
\end{scope}
   \node at (0,1.1) {$\beta$};
   \node at (1.2,1.1) {$\alpha$};
   \node at (0.9,0) {$\gamma$};
   \node at (0,-1.1) {$\beta$};
   \node at (1.2,-1.1) {$\alpha$};
\end{tikzpicture} }
\end{equation}
which can be deduced from the following commutative diagram
\begin{equation*}
\begin{tikzcd}
V_\beta \otimes V_\alpha \arrow[rrr, "\sqrt{\theta}_{V_{\beta}\otimes V_{\alpha}}"] \arrow[d, "U^{-1}", swap] & & & V_\beta \otimes V_\alpha \\
\bigoplus\limits_{\gamma \in \alpha + \beta + H_r} V_{\gamma} \arrow[rrr, "\bigoplus\limits_{\gamma \in \alpha + \beta + H_r} \sqrt{\theta}_{V_{\gamma}}", swap] & & & \bigoplus\limits_{\gamma \in \alpha + \beta + H_r} V_{\gamma} \arrow[u, "U", swap]
\end{tikzcd}
\end{equation*}
with the vertical isomorphism $U$ chosen as
\begin{equation*}
U=\begin{bmatrix}
\hackcenter{\begin{tikzpicture}[scale=1.1]
\begin{scope}[decoration={markings, mark=at position 0.5 with {\arrow{>}}}]
  \draw[ultra thick, black, postaction=decorate] (0,0) to (.6,-.6);
  \draw[ultra thick, black, postaction=decorate] (.6,-.6) to (.6,-1.2);
  \draw[ultra thick, black, postaction=decorate] (1.2,0) to (.6,-.6);
\end{scope}
   \node at (0,0.2) {$\beta$};
   \node at (1.2,0.2) {$\alpha$};
   \node at (.6,-1.4) {$\gamma_0$};
\end{tikzpicture} } & \cdots & \hackcenter{\begin{tikzpicture}[ scale=1.1]
\begin{scope}[decoration={markings, mark=at position 0.5 with {\arrow{>}}}]
  \draw[ultra thick, black, postaction=decorate] (0,0) to (.6,-.6);
  \draw[ultra thick, black, postaction=decorate] (.6,-.6) to (.6,-1.2);
  \draw[ultra thick, black, postaction=decorate] (1.2,0) to (.6,-.6);
\end{scope}
   \node at (0,0.2) {$\beta$};
   \node at (1.2,0.2) {$\alpha$};
   \node at (.6,-1.4) {$\gamma_{r-1}$};
\end{tikzpicture} }
\end{bmatrix} \quad \iff \quad U^{-1} = \begin{bmatrix}
\md(\gamma_0) \hackcenter{\begin{tikzpicture}[scale=1.1]
\begin{scope}[decoration={markings, mark=at position 0.75 with {\arrow{>}}}]
  \draw[ultra thick, black, postaction=decorate] (0.6,0) to (.6,-.6);
  \draw[ultra thick, black, postaction=decorate] (.6,-.6) to (0,-1.2);
  \draw[ultra thick, black, postaction=decorate] (.6,-.6) to (1.2,-1.2);
\end{scope}
   \node at (0.6,0.2) {$\gamma_0$};
   \node at (1.2,-1.4) {$\alpha$};
   \node at (0,-1.4) {$\beta$};
\end{tikzpicture} }\\
\\
\vdots \\
\\
\md(\gamma_{r-1})\hackcenter{\begin{tikzpicture}[scale=1.1]
\begin{scope}[decoration={markings, mark=at position 0.75 with {\arrow{>}}}]
  \draw[ultra thick, black, postaction=decorate] (0.6,0) to (.6,-.6);
  \draw[ultra thick, black, postaction=decorate] (.6,-.6) to (0,-1.2);
  \draw[ultra thick, black, postaction=decorate] (.6,-.6) to (1.2,-1.2);
\end{scope}
   \node at (0.6,0.2) {$\gamma_{r-1}$};
   \node at (1.2,-1.4) {$\alpha$};
   \node at (0,-1.4) {$\beta$};
\end{tikzpicture} }
\end{bmatrix}
\end{equation*}
where $\gamma_{k} = \beta + \alpha + r-1 - 2k \in \beta + \alpha+H_r$. Then, applying~\cite[Equation (Ng)]{CGP14} and a bubble pop gives our desired solution:
\begin{equation}
\begin{split}
\mathrm{X}_{V_\alpha,V_\beta}Y_{\gamma}^{\alpha,\beta} &= \sum\limits_{\gamma_k \in \alpha + \beta + H_r} q^{\frac{\alpha^2 + \beta^2-\gamma_k^2 - (r-1)^2}{4}} \md(\gamma)
\hackcenter{\begin{tikzpicture}[scale=0.8]
\begin{scope}[decoration={markings, mark=at position 0.7 with {\arrow{>}}}]
  \draw[ultra thick, black] (0,-0.9) to (1.2,-2.1);
  \draw[line width=10pt, white] (1.2,-0.9) to (0,-2.1);
  \draw[ultra thick, black] (1.2,-0.9) to (0,-2.1);
  \draw[ultra thick, black, postaction=decorate] (0,-2.1) to (.6,-2.7);
  \draw[ultra thick, black, postaction=decorate] (1.2,-2.1) to (.6,-2.7);
  \draw[ultra thick, black, postaction=decorate] (0.6,-2.7) to (.6,-3.3);
  \draw[ultra thick, black, postaction=decorate] (0,.9) to (.6,.3);
  \draw[ultra thick, black, postaction=decorate] (.6,.3) to (.6,-.3);
  \draw[ultra thick, black, postaction=decorate] (1.2,.9) to (.6,0.3);
  \draw[ultra thick, black] (.6,-.3) to (0,-0.9);
  \draw[ultra thick, black] (.6,-.3) to (1.2,-0.9);
\end{scope}
   \node at (0,1.1) {$\beta$};
   \node at (1.2,1.1) {$\alpha$};
   \node at (1.1,0) {$\gamma_k$};
   \node at (-0.2,-2.5) {$\alpha$};
   \node at (1.4,-2.5) {$\beta$};
   \node at (.6,-3.6) {$\gamma$};
\end{tikzpicture} } = \hackcenter{\begin{tikzpicture}[scale=1.1]
\begin{scope}[decoration={markings, mark=at position 0.5 with {\arrow{>}}}]
  \draw[ultra thick, black, postaction=decorate] (0,0) to (.6,-.6);
  \draw[ultra thick, black, postaction=decorate] (.6,-.6) to (.6,-1.2);
  \draw[ultra thick, black, postaction=decorate] (1.2,0) to (.6,-.6);
\end{scope}
   \node at (0,0.2) {$\beta$};
   \node at (1.2,0.2) {$\alpha$};
   \node at (.6,-1.4) {$\gamma$};
\end{tikzpicture} }.
\end{split}
\end{equation}
\end{proof}

We now review the following decomposition.

\begin{definition}[Clebsch-Gordan decomposition~\cite{CM}]
The maps $Y^{\alpha,\beta}_\gamma$ and $Y_{\alpha,\beta}^\gamma$ can be expressed as a \emph{Clebsch-Gordan decomposition}:
\begin{equation}
Y_{\gamma}^{\alpha,\beta}(v_c^\gamma) = \sum\limits_{a+b-c=(\alpha+\beta+r-1-\gamma)/2} A_{a,b,c}^{\alpha,\beta,\gamma} v_a^\alpha \otimes v_b^\beta
\end{equation}
\begin{equation}
Y_{\alpha,\beta}^\gamma (v_a^\alpha \otimes v_b^\beta) = B^{\alpha,\beta,\gamma}_{a,b,c} v_c^\gamma
\end{equation}
where the coefficient $A^{\alpha,\beta,\gamma}_{a,b,c}$ is called the \emph{Clebsch-Gordan quantum coefficient} (CGQC).
\end{definition}

We refer the reader to~\cite[Subsection 6.2]{CGP14} to obtain the prescription for computing the coefficients: this requires translating the data of~\cite{CM} to~\cite{CGP14} via the functor $U\text{-cat} \rightarrow \mathcal{C}$ (see~\cite[Subsection 6.2]{CGP14} for the definition of $U\text{-cat}$). 

\begin{notation}
There is an isomorphism $w_\alpha: V_\alpha \rightarrow V_{-\alpha}^*$ provided in this prescription, that we will use later, represented in graphical calculus as follows:
\begin{equation}
w_\alpha=\hackcenter{\begin{tikzpicture}[scale=1.1]
\begin{scope}[decoration={markings, mark=at position 0.9 with {\arrow{>}}}]
  \draw (-0.3,-0.3) rectangle (0.3,0.3);
  \draw (0,0) node {$w_\alpha$};
  \draw[ultra thick, black, postaction=decorate] (0,0.3) to (0,0.9);
  \draw[ultra thick, black, postaction=decorate] (0,-0.3) to (0,-0.9);
  \node at (0,-1.1) {$\alpha$};
  \node at (0,1.1) {$-\alpha$};
\end{scope}
\end{tikzpicture} } ~=:~ \hackcenter{\begin{tikzpicture}[scale=1.1]
\begin{scope}[decoration={markings, mark=at position 0.9 with {\arrow{>}}}]
  \draw[ultra thick, black, postaction=decorate] (0,0) to (0,0.9);
  \draw[ultra thick, black, postaction=decorate] (0,0) to (0,-0.9);
  \node at (0,0) [circle,fill,inner sep=2pt]{};
  \node at (0,-1.1) {$\alpha$};
  \node at (0,1.1) {$-\alpha$};
\end{scope}
\end{tikzpicture} }, \qquad 
\id_{V_\alpha}=\hackcenter{\begin{tikzpicture}[scale=1.1]
\begin{scope}[decoration={markings, mark=at position 0.6 with {\arrow{>}}}]
  \draw[ultra thick, black, postaction=decorate] (0,0.9) to (0,-0.9);
  \node at (0,-1.1) {$\alpha$};
  \node at (0,1.1) {$\alpha$};
\end{scope}
\end{tikzpicture} } ~=~ \hackcenter{\begin{tikzpicture}[scale=1.1]
\begin{scope}[decoration={markings, mark=at position 0.65 with {\arrow{>}}}]
  \draw[ultra thick, black, postaction=decorate] (0,0.9) to (0,0.4);
  \draw[ultra thick, black, postaction=decorate] (0,-0.4) to (0,0.4);
  \draw[ultra thick, black, postaction=decorate] (0,-0.4) to (0,-0.9);
  \node at (0,0.4) [circle,fill,inner sep=2pt]{};
  \node at (0,-0.4) [circle,fill,inner sep=2pt]{};
  \node at (0,-1.1) {$\alpha$};
  \node at (0.5,0) {$-\alpha$};
  \node at (0,1.1) {$\alpha$};
\end{scope}
\end{tikzpicture} } = w_\alpha^{-1} \circ w_\alpha.
\end{equation}
\end{notation}
Note that $B^{\alpha,\beta,\gamma}_{a,b,c}$ can be expressed in terms of $A^{\alpha,\beta,\gamma}_{a,b,c}$. In this article, we do not work with explicit values of these coefficients, but rather provide an algorithm for computing $(Y_{\gamma}^{\alpha,\beta})^{\dagger}$ in terms of $Y_{\alpha,\beta}^{\gamma}$.

\begin{lemma} \label{lemma:dagger}
Let $\gamma=\alpha+\beta+r-1-2k$ for $k\in \{ 0, 1, \dots, r-1 \}$. Then, \begin{equation*}(Y_{\gamma}^{\alpha,\beta})^\dagger = C^{\alpha,\beta}_\gamma  Y_{\alpha,\beta}^\gamma\end{equation*} where \begin{equation*}C^{\alpha,\beta}_\gamma:= \frac{\overline{\left( A_{k,0,0}^{\beta,\alpha,\gamma} \prod\limits_{i=0}^k [i][i-\beta]\right)}}{B_{0,k,0}^{\alpha,\beta,\gamma}}\end{equation*} and $\overline{(-)}$ denotes the complex conjugation of $(-)$.
\end{lemma}

\begin{proof}
Since $\gamma = \alpha + \beta + r - 1 - 2k$, it follows that $Y_{\alpha,\beta}^{\gamma}(v_0^\alpha \otimes v_k^\beta) \neq 0$.

On one hand, 
\begin{align*}
((Y_{\gamma}^{\alpha,\beta})^\dagger (v_0^\alpha \otimes v_k^\beta), v_0^\gamma)_{V_\gamma} & = (C_{\gamma}^{\alpha,\beta}  Y_{\alpha,\beta}^\gamma (v_0^\alpha \otimes v_k^\beta), v_0^\gamma)_{V_\gamma} \\
& = (C^{\alpha,\beta}_\gamma  B^{\alpha,\beta,\gamma}_{0,k,0} v_0^\gamma, v_0^\gamma)_{V_\gamma} \\
& = \overline{C^{\alpha,\beta}_\gamma  B^{\alpha,\beta,\gamma}_{0,k,0}}.
\end{align*}
On the other hand, utilizing Proposition~\ref{prop:X-map},
\begin{align*}
((Y_{\gamma}^{\alpha,\beta})^\dagger (v_0^\alpha \otimes v_k^\beta), v_0^\gamma)_{V_\gamma} & = (v_0^\alpha \otimes v_k^\beta, Y_{\gamma}^{\alpha,\beta}(v_0^\gamma))_{V_\alpha \otimes V_\beta}\\
& = (v_0^\alpha \otimes v_k^\beta, \tau X_{V_\alpha, V_\beta} Y_{\gamma}^{\alpha,\beta}(v_0^\gamma))_{p}\\
& = (v_0^\alpha \otimes v_k^\beta, \tau Y_{\gamma}^{\beta, \alpha} (v_0^\gamma))_{p} \\
& = A^{\beta,\alpha,\gamma}_{k,0,0}(v_0^\alpha \otimes v_k^\beta, \tau (v_k^\beta \otimes v_0^\alpha))_{p} \\
& = A^{\beta,\alpha,\gamma}_{k,0,0} (v_0^\alpha, v_0^\alpha)_{V_\alpha}(v_k^\beta, v_k^\beta)_{V_\beta}\\
& = A^{\beta,\alpha,\gamma}_{k,0,0}\prod\limits_{i=0}^k [i][i-\beta].
\end{align*}
\end{proof}

\begin{proposition} \label{prop:dagger-coeff}
We have the following identity
\begin{equation}
\left( \hackcenter{\begin{tikzpicture}[scale=1]
\begin{scope}[decoration={markings, mark=at position 0.5 with {\arrow{>}}}]
  \draw[ultra thick, black, postaction=decorate] (0,0) to (0.6,-0.6);
  \draw[ultra thick, black, postaction=decorate] (0.6,-0.6) to (1.2,-1.2);
  \draw[ultra thick, dotted] (1.2,-1.2) to (1.8,-1.8);
  \draw[ultra thick, black, postaction=decorate] (1.8,-1.8) to (2.4,-2.4);
  \draw[ultra thick, black, postaction=decorate] (1.2,0) to (.6,-.6);
  \draw[ultra thick, black, postaction=decorate] (2.4,0) to (1.2,-1.2);
  \draw[ultra thick, black, postaction=decorate] (3.6,0) to (1.8,-1.8);
  \draw[ultra thick, black, postaction=decorate] (4.8,0) to (2.4,-2.4);
  \draw[ultra thick, black, postaction=decorate] (2.4,-2.4) to (2.4,-3.3);
\end{scope}
   \node at (0,0.2) {$\alpha_1$};
   \node at (1.2,0.2) {$\alpha_2$};
   \node at (2.4,0.2) {$\alpha_3$};
   \node at (2.9,0.2) {$\cdots$};
   \node at (3.6,0.2) {$\alpha_{n-1}$};
   \node at (4.8,0.2) {$\alpha_n$};
   \node at (.7,-1.1) {$\gamma_1$};
   \node at (1.75,-2.3) {$\gamma_{n-2}$};
   \node at (2.4,-3.4) {$\gamma_{n-1}$};
\end{tikzpicture} } \right)^{\dagger} ~=~ C^{\alpha_1, \dots, \alpha_n}_{\gamma_1, \dots, \gamma_{n-1}}
\hackcenter{\begin{tikzpicture}[scale=0.9]
\begin{scope}[decoration={markings, mark=at position 0.6 with {\arrow{>}}}]
  \draw[ultra thick, black, postaction=decorate] (0.6,0.6) to (0,0);
  \draw[ultra thick, black, postaction=decorate] (1.2,1.2) to (0.6,0.6);
  \draw[ultra thick, dotted] (1.8,1.8) to (1.2,1.2);
  \draw[ultra thick, black, postaction=decorate] (2.4,2.4) to (1.8,1.8);
  \draw[ultra thick, black, postaction=decorate] (.6,.6) to (1.2,0);
  \draw[ultra thick, black, postaction=decorate] (1.2,1.2) to (2.4,0);
  \draw[ultra thick, black, postaction=decorate] (1.8,1.8) to (3.6,0);
  \draw[ultra thick, black, postaction=decorate] (2.4,2.4) to (4.8,0);
  \draw[ultra thick, black, postaction=decorate] (2.4,3.3) to (2.4,2.4);
\end{scope}
   \node at (0,-0.2) {$\alpha_1$};
   \node at (1.2,-0.2) {$\alpha_2$};
   \node at (2.4,-0.2) {$\alpha_3$};
   \node at (2.9,0) {$\cdots$};
   \node at (3.6,-0.2) {$\alpha_{n-1}$};
   \node at (4.8,-0.2) {$\alpha_{n}$};
   \node at (.6,1.1) {$\gamma_1$};
   \node at (1.75,2.5) {$\gamma_{n-2}$};
   \node at (2.4,3.5) {$\gamma_{n-1}$};
\end{tikzpicture} }
\end{equation}
where $C^{\alpha_1, \dots, \alpha_n}_{\gamma_1, \dots, \gamma_{n-1}} := \prod\limits_{j=1}^{n-1} C^{\gamma_{j-1}, \alpha_{j+1}}_{\gamma_j}$ such that $\gamma_0 := \alpha_1$.
\end{proposition}

\begin{proof}
This follows immediately from Lemma~\ref{lemma:dagger} and the definition and properties of $\dagger$ in a Hermitian ribbon category.
\end{proof}

\begin{lemma} \label{lemma:pairing}
We have the following expression:
\begin{equation*}
\left\langle \hackcenter{\begin{tikzpicture}[scale=1]
\begin{scope}[decoration={markings, mark=at position 0.5 with {\arrow{>}}}]
  \draw[ultra thick, black, postaction=decorate] (0,0) to (0.6,-0.6);
  \draw[ultra thick, black, postaction=decorate] (0.6,-0.6) to (1.2,-1.2);
  \draw[ultra thick, dotted] (1.2,-1.2) to (1.8,-1.8);
  \draw[ultra thick, black, postaction=decorate] (1.8,-1.8) to (2.4,-2.4);
  \draw[ultra thick, black, postaction=decorate] (1.2,0) to (.6,-.6);
  \draw[ultra thick, black, postaction=decorate] (2.4,0) to (1.2,-1.2);
  \draw[ultra thick, black, postaction=decorate] (3.6,0) to (1.8,-1.8);
  \draw[ultra thick, black, postaction=decorate] (4.8,0) to (2.4,-2.4);
  \draw[ultra thick, black, postaction=decorate] (2.4,-2.4) to (2.4,-3.3);
\end{scope}
   \node at (0,0.2) {$\alpha_1$};
   \node at (1.2,0.2) {$\alpha_2$};
   \node at (2.4,0.2) {$\alpha_3$};
   \node at (2.9,0.2) {$\cdots$};
   \node at (3.6,0.2) {$\alpha_{n-1}$};
   \node at (4.8,0.2) {$\alpha_n$};
   \node at (.7,-1.1) {$\gamma_1$};
   \node at (1.75,-2.3) {$\gamma_{n-2}$};
   \node at (2.4,-3.4) {$\gamma_{n-1}$};
\end{tikzpicture} },
\hackcenter{\begin{tikzpicture}[scale=1]
\begin{scope}[decoration={markings, mark=at position 0.5 with {\arrow{>}}}]
  \draw[ultra thick, black, postaction=decorate] (0,0) to (0.6,-0.6);
  \draw[ultra thick, black, postaction=decorate] (0.6,-0.6) to (1.2,-1.2);
  \draw[ultra thick, dotted] (1.2,-1.2) to (1.8,-1.8);
  \draw[ultra thick, black, postaction=decorate] (1.8,-1.8) to (2.4,-2.4);
  \draw[ultra thick, black, postaction=decorate] (1.2,0) to (.6,-.6);
  \draw[ultra thick, black, postaction=decorate] (2.4,0) to (1.2,-1.2);
  \draw[ultra thick, black, postaction=decorate] (3.6,0) to (1.8,-1.8);
  \draw[ultra thick, black, postaction=decorate] (4.8,0) to (2.4,-2.4);
  \draw[ultra thick, black, postaction=decorate] (2.4,-2.4) to (2.4,-3.3);
\end{scope}
   \node at (0,0.2) {$\alpha_1$};
   \node at (1.2,0.2) {$\alpha_2$};
   \node at (2.4,0.2) {$\alpha_3$};
   \node at (2.9,0.2) {$\cdots$};
   \node at (3.6,0.2) {$\alpha_{n-1}$};
   \node at (4.8,0.2) {$\alpha_n$};
   \node at (.7,-1.1) {$\psi_1$};
   \node at (1.75,-2.3) {$\psi_{n-2}$};
   \node at (2.4,-3.4) {$\gamma_{n-1}$};
\end{tikzpicture} }
\right\rangle ~=~ C^{\alpha_1, \dots, \alpha_n}_{\gamma_1, \dots, \gamma_{n-1}} \prod\limits_{j=1}^{n-2} \frac{\delta_{\gamma_j\psi_j}}{\md(\gamma_{j})}.
\end{equation*}
\end{lemma}

\begin{proof}
This follows from definition of the Hermitian pairing, Proposition~\ref{prop:dagger-coeff}, and applying a sequence of bubble pops.
\end{proof}

\begin{remark}
Recall that our Hermitian category $\mathcal{D}^\dagger$ only contain the $V_\alpha$ representations parametrized by $\alpha \in (\mathbb{R} \setminus \mathbb{Z}) \cup r \mathbb{Z}$. Moreover, real $\alpha$ values admit real modified dimensions, i.e. $\md(\alpha) \in \mathbb{R}$. Furthermore, the Hermitian pairing on fusion trees also admit real values (see~\cite{neglecton}).
\end{remark}

\subsection{Special $\alpha$-Fusion Trees} Inspired by Anghel's topological model, we introduce a new type of fusion tree.

\begin{definition}{(Special $\alpha$-Fusion Tree)} \label{def:special-fusion-tree}
Let $\mathbf{p}=(\alpha, \gamma_1, \dots, \gamma_{n-1})\in P_{n,m,\alpha}^r$ be a path. We refer to an element of the form
\begin{equation}
Y_{\mathbf{p}}^s:=
\hackcenter{\begin{tikzpicture}[baseline=0, thick, scale=0.5, shift={(0,0)}]
\begin{scope}[shift={(1,-5.5)}]
\begin{scope}[decoration={markings, mark=at position 0.6 with {\arrow{>}}}]
    \draw[ultra thick, black, postaction=decorate] (0,4.5) to (0.8125,3.75);
    \draw[ultra thick, black, postaction=decorate] (1,4.5) to (0.8125,3.75);
    \draw[ultra thick, black, postaction=decorate] (0.8125,3.75) to (1.625,3);
    \draw[ultra thick, black, postaction=decorate] (2,4.5) to (1.625,3);
    \draw[ultra thick, black, dotted] (1.625, 3) to (2.4375, 2.25);
    \draw[ultra thick, black, postaction=decorate] (3, 4.5) to (2.4375, 2.25);
    \draw[ultra thick, black, postaction=decorate] (2.4375,2.25) to (3.25,1.5);
    \draw[ultra thick, black, postaction=decorate] (4, 4.5) to (3.25,1.5);
    \draw[ultra thick, black, postaction=decorate] (5, 4.5) to (4.0625,0.75);
    \draw[ultra thick, black, postaction=decorate] (3.25,1.5) to (4.0625,0.75);
    \draw[ultra thick, black, postaction=decorate] (6, 4.5) to (4.875,0);
    \draw[ultra thick, black, postaction=decorate] (4.0625,0.75) to (4.875,0);
    \draw[ultra thick, black, postaction=decorate] (7, 4.5) to (5.6875,-0.75);
    \draw[ultra thick, black, dotted] (4.875,0) to (5.6875,-0.75);
    \draw[ultra thick, black, postaction=decorate] (8, 4.5) to (6.5,-1.5);
    \draw[ultra thick, black, postaction=decorate] (5.6875,-0.75) to (6.5,-1.5);
    \draw[ultra thick, black, postaction=decorate] (6.5,-1.5) to (6.5,-2.5);
\end{scope}
\end{scope}
\draw (1,-0.8) node {$\scs \alpha$};
\draw (2,-0.8) node {$\scs \alpha$};
\draw (3,-0.8) node {$\scs \alpha$};
\draw (3.5,-1.25) node {$\scs \cdots$};
\draw (4,-0.8) node {$\scs \alpha$};
\draw (5,-0.8) node {$\scs \alpha$};
\draw (6,-0.8) node {$\scs -\alpha$};
\draw (7,-0.8) node {$\scs -\alpha$};
\draw (7.5,-1.25) node {$\scs \cdots$};
\draw (8,-0.8) node {$\scs -\alpha$};
\draw (9,-0.8) node {$\scs -\alpha$};
\draw (2.2,-2.5) node {$\scs \gamma_1$};
\draw (3.3,-4) node {$\scs \gamma_{n-2}$};
\draw (4.25,-4.75) node {$\scs \gamma_{n-1}$};
\draw (5.2,-5.6) node {$\scs \gamma_{n-2}$};
\draw (6.5,-7) node {$\scs \gamma_{1}$};
\draw (7.5,-8.2) node {$\scs \alpha$};
\end{tikzpicture} } \in \Hom(V_{\alpha}, V_{\alpha}^{\otimes n} \otimes V_{-\alpha}^{\otimes (n-1)}) 
\end{equation}
as a \emph{special $\alpha$-fusion tree}.
\end{definition}

To motivate this definition, we provide an alternative proof of Theorem~\ref{thm:CGP-formula} in which special $\alpha$-fusion trees naturally appear.

\begin{proof}
First, we identify the partial braid closure in the following way using $w_\alpha$
\begin{equation} \label{eq:special-identification}
\hackcenter{\begin{tikzpicture}[baseline=0, thick, scale=0.75, shift={(0,0)}]
\draw (0.5,.5) rectangle (4.5,-.5);
\draw (2.5,0) node {$\beta_n$};
\draw (3,1) node {$\cdots$};
\draw (3,-1) node {$\cdots$};
\draw (5.5,0) node {$\cdots$};
\draw (1,3) node[above] {$\alpha$};
\draw (1,-3) node[below] {$\alpha$};
\draw (2,1) node[left] {$\alpha$};
\draw (4,1) node[left] {$\alpha$};
\draw[ultra thick, black] (4,0.5) to (4, 1);
\draw[ultra thick, black] (4,1) to [out=90,in=90] (5,1);
\draw[ultra thick, black] (4,-1) to [out=-90,in=-90] (5,-1);
\draw[ultra thick, black] (2,0.5) to (2, 1.5);
\draw[ultra thick, black] (2,1.5) to [out=90,in=90] (6,1.5);
\draw[ultra thick, black] (2,-1.5) to [out=-90,in=-90] (6,-1.5);
\draw[ultra thick, black] (2,-0.5) to (2, -1.5);
\draw[ultra thick, black] (4,-0.5) to (4, -1);
\begin{scope}[decoration={markings, mark=at position 0.5 with {\arrow{>}}}]
    \draw[ultra thick, black, postaction=decorate] (1,3) to (1, 0.5);
    \draw[ultra thick, black, postaction=decorate] (1,-0.5) to (1, -3);
    \draw[ultra thick, black, postaction=decorate] (5,-1) to (5,1);
    \draw[ultra thick, black, postaction=decorate] (6,-1.5) to (6,1.5);
\end{scope}
\end{tikzpicture} }
~=~
\hackcenter{\begin{tikzpicture}[baseline=0, thick, scale=0.75, shift={(0,0)}]
\draw (0.5,.5) rectangle (4.5,-.5);
\draw (2.5,0) node {$\beta_n$};
\draw (3,.85) node {$\cdots$};
\draw (3,-.85) node {$\cdots$};
\draw (5.5,0) node {$\cdots$};
\draw (1,3) node[above] {$\alpha$};
\draw (1,-3) node[below] {$\alpha$};
\draw (1.9,.85) node[left] {$\alpha$};
\draw (3.9,.85) node[left] {$\alpha$};
\draw[ultra thick, black] (4,1) to [out=90,in=90] (5,1);
\draw[ultra thick, black] (4,-1) to [out=-90,in=-90] (5,-1);
\draw[ultra thick, black] (2,1.5) to [out=90,in=90] (6,1.5);
\draw[ultra thick, black] (2,-1.5) to [out=-90,in=-90] (6,-1.5);
\draw[ultra thick, black] (1,3) to (1, 0.5);
\draw[ultra thick, black] (1,-0.5) to (1, -3);
\draw[ultra thick, black] (4,1) to (4, 0.5);
\draw[ultra thick, black] (4,-0.5) to (4, -1);
\draw[ultra thick, black] (5,1) to (5,-1);
\draw[ultra thick, black] (6,1.5) to (6,-1.5);
\draw[ultra thick, black] (2,-0.5) to (2, -1.5);
\draw[ultra thick, black] (2,1.5) to (2, 0.5);
\begin{scope}[decoration={markings, mark=at position 0.75 with {\arrow{>}}}]
    \draw[ultra thick, black, postaction=decorate] (1,1.1) to (1, 0.6);
    \draw[ultra thick, black, postaction=decorate] (1,-0.6) to (1, -1.1);
    \draw[ultra thick, black, postaction=decorate] (2,1.1) to (2, 0.6);
    \draw[ultra thick, black, postaction=decorate] (2,-0.6) to (2, -1.1);
    \draw[ultra thick, black, postaction=decorate] (4,1.1) to (4, 0.6);
    \draw[ultra thick, black, postaction=decorate] (4,-0.6) to (4, -1.1);
    \draw[ultra thick, black, postaction=decorate] (5,1.1) to (5, 0.6);
    \draw[ultra thick, black, postaction=decorate] (5,-0.6) to (5, -1.1);
    \draw[ultra thick, black, postaction=decorate] (6,1.1) to (6, 0.6);
    \draw[ultra thick, black, postaction=decorate] (6,-0.6) to (6, -1.1);
\end{scope}
\node at (4.5,1.3) [circle,fill,inner sep=2pt]{};
\node at (4.5,-1.3) [circle,fill,inner sep=2pt]{};
\node at (4,2.67) [circle,fill,inner sep=2pt]{};
\node at (4,-2.67) [circle,fill,inner sep=2pt]{};
\draw[red,dashed] (0.5,0.6) rectangle (6.5,1.1);
\draw[red,dashed] (0.5,-0.6) rectangle (6.5,-1.1);
\end{tikzpicture} } .
\end{equation}
This time, we apply pinch moves (Lemma~\ref{lemma:fusion-identity}) on $\id_{V_\alpha}^{\otimes n} \otimes \id_{V_{-\alpha}}^{\otimes (n-1)}$ (see the red-dashed rectangle in Equation~\eqref{eq:special-identification}). Then, we isotope the center piece

\begin{equation}
\hackcenter{\begin{tikzpicture}[baseline=0, thick, scale=0.5, shift={(0,0)}]
\draw (0.5,.5) rectangle (5.5,-.5);
\draw (3,0) node {$\scs \beta_n$};
\begin{scope}[shift={(1,-5.5)}]
\begin{scope}[decoration={markings, mark=at position 0.6 with {\arrow{>}}}]
    \draw[ultra thick, black, postaction=decorate] (0,4.5) to (0.8125,3.75);
    \draw[ultra thick, black, postaction=decorate] (1,4.5) to (0.8125,3.75);
    \draw[ultra thick, black, postaction=decorate] (0.8125,3.75) to (1.625,3);
    \draw[ultra thick, black, postaction=decorate] (2,4.5) to (1.625,3);
    \draw[ultra thick, black, dotted] (1.625, 3) to (2.4375, 2.25);
    \draw[ultra thick, black, postaction=decorate] (3, 4.5) to (2.4375, 2.25);
    \draw[ultra thick, black, postaction=decorate] (2.4375,2.25) to (3.25,1.5);
    \draw[ultra thick, black, postaction=decorate] (4, 4.5) to (3.25,1.5);
    \draw[ultra thick, black, postaction=decorate] (5, 4.5) to (4.0625,0.75);
    \draw[ultra thick, black, postaction=decorate] (3.25,1.5) to (4.0625,0.75);
    \draw[ultra thick, black, postaction=decorate] (6, 4.5) to (4.875,0);
    \draw[ultra thick, black, postaction=decorate] (4.0625,0.75) to (4.875,0);
    \draw[ultra thick, black, postaction=decorate] (7, 4.5) to (5.6875,-0.75);
    \draw[ultra thick, black, dotted] (4.875,0) to (5.6875,-0.75);
    \draw[ultra thick, black, postaction=decorate] (8, 4.5) to (6.5,-1.5);
    \draw[ultra thick, black, postaction=decorate] (5.6875,-0.75) to (6.5,-1.5);
    \draw[ultra thick, black, postaction=decorate] (6.5,-1.5) to (6.5,-2.5);
\end{scope}
\end{scope}
\begin{scope}[shift={(1,5.5)}]
\begin{scope}[decoration={markings, mark=at position 0.9 with {\arrow{>}}}]
    \draw[ultra thick, black, postaction=decorate] (6.5,2.5) to (6.5,1.5);
    \draw[ultra thick, black, postaction=decorate] (0.8125,-3.75) to (0,-4.5);
    \draw[ultra thick, black, postaction=decorate] (0.8125,-3.75) to (1,-4.5);
    \draw[ultra thick, black, postaction=decorate] (1.625,-3) to (0.8125,-3.75);
    \draw[ultra thick, black, postaction=decorate] (1.625,-3) to (2,-4.5);
    \draw[ultra thick, black, dotted] (2.4375,-2.25) to (1.625,-3);
    \draw[ultra thick, black, postaction=decorate] (2.4375,-2.25) to (3,-4.5);
    \draw[ultra thick, black, postaction=decorate] (3.25,-1.5) to (2.4375,-2.25);
    \draw[ultra thick, black, postaction=decorate] (3.25,-1.5) to (4,-4.5);
    \draw[ultra thick, black, postaction=decorate] (4.0625,-0.75) to (3.25,-1.5);
    \draw[ultra thick, black, postaction=decorate] (4.0625,-0.75) to (5,-4.5);
    \draw[ultra thick, black, postaction=decorate] (4.875,0) to (4.0625,-0.75);
    \draw[ultra thick, black, postaction=decorate] (4.875,0) to (6,-4.5);
    \draw[ultra thick, black, dotted] (5.6875,.75) to (4.875,0);
    \draw[ultra thick, black, postaction=decorate] (5.6875,0.75) to (7,-4.5);
    \draw[ultra thick, black, postaction=decorate] (6.5,1.5) to (5.6875,0.75);
    \draw[ultra thick, black, postaction=decorate] (6.5,1.5) to (8,-4.5);
\end{scope}
\end{scope}
\draw[ultra thick, black] (1,0.5) to (1,1);
\draw[ultra thick, black] (2,0.5) to (2,1);
\draw[ultra thick, black] (3,0.5) to (3,1);
\draw[ultra thick, black] (4,0.5) to (4,1);
\draw[ultra thick, black] (5,0.5) to (5,1);
\draw[ultra thick, black] (1,-0.5) to (1,-1);
\draw[ultra thick, black] (2,-0.5) to (2,-1);
\draw[ultra thick, black] (3,-0.5) to (3,-1);
\draw[ultra thick, black] (4,-0.5) to (4,-1);
\draw[ultra thick, black] (5,-0.5) to (5,-1);
\draw[ultra thick, black] (6,1) to (6,-1);
\draw[ultra thick, black] (7,1) to (7,-1);
\draw[ultra thick, black] (8,1) to (8,-1);
\draw[ultra thick, black] (9,1) to (9,-1);
\draw (0.7,0.75) node {$\scs \alpha$};
\draw (1.7,0.75) node {$\scs \alpha$};
\draw (2.7,0.75) node {$\scs \alpha$};
\draw (3.5,0.75) node {$\scs \cdots$};
\draw (4.3,0.75) node {$\scs \alpha$};
\draw (5.3,0.75) node {$\scs \alpha$};
\draw (6.5,0) node {$\scs -\alpha$};
\draw (7.5,0) node {$\scs \cdots$};
\draw (8.5,0) node {$\scs -\alpha$};
\draw (9.5,0) node {$\scs -\alpha$};
\draw (0.7,-0.75) node {$\scs \alpha$};
\draw (1.7,-0.75) node {$\scs \alpha$};
\draw (2.7,-0.75) node {$\scs \alpha$};
\draw (3.5,-0.75) node {$\scs \cdots$};
\draw (4.3,-0.75) node {$\scs \alpha$};
\draw (5.3,-0.75) node {$\scs \alpha$};
\draw (2.2,-2.5) node {$\scs \psi_1$};
\draw (3.3,-4) node {$\scs \psi_{n-2}$};
\draw (4.25,-4.75) node {$\scs \psi_{n-1}$};
\draw (5.3,-5.6) node {$\scs \psi_{n}$};
\draw (6.5,-7) node {$\scs \psi_{2n-3}$};
\draw (7.5,-8.2) node {$\scs \psi_{2n-2}$};
\draw (2.2,2.5) node {$\scs \gamma_1$};
\draw (3.3,4) node {$\scs \gamma_{n-2}$};
\draw (4.25,4.75) node {$\scs \gamma_{n-1}$};
\draw (5.3,5.6) node {$\scs \gamma_{n}$};
\draw (6.5,7) node {$\scs \gamma_{2n-3}$};
\draw (7.5,8.2) node {$\scs \gamma_{2n-2}$};
\end{tikzpicture} }
~=~
\hackcenter{\begin{tikzpicture}[baseline=0, thick, scale=0.6, shift={(0,0)}]
\begin{scope}[decoration={markings, mark=at position 0.5 with {\arrow{>}}}]
    \draw[ultra thick, black, postaction=decorate] (0,4.5) to (0.8125,3.75);
    \draw[ultra thick, black, postaction=decorate] (0.8125,3.75) to (1.625,3);
    \draw[ultra thick, black, dotted] (1.625, 3) to (2.4375, 2.25);
    \draw[ultra thick, black, postaction=decorate] (2.4375,2.25) to (3.25,1.5);
    \draw[ultra thick, black, postaction=decorate] (3.25,1.5) to (3.25,0.5);
\end{scope}
\begin{scope}[decoration={markings, mark=at position 0.8 with {\arrow{>}}}]
    \draw[ultra thick, black, postaction=decorate] (0.8125,3.75) to (1,4.5);
    \draw[ultra thick, black, postaction=decorate] (1.625,3) to (2,4.5);
    \draw[ultra thick, black, postaction=decorate] (2.4375, 2.25) to (3, 4.5);
    \draw[ultra thick, black, postaction=decorate] (3.25,1.5) to (4, 4.5);
\end{scope}
\draw (0,6.2) node {$\scs \gamma_{2n-2}$};
\draw (0.65,5) node {$\scs -\alpha$};
\draw (1.65,5) node {$\scs -\alpha$};
\draw (2.75,5) node {$\scs \cdots$};
\draw (4.5,5) node {$\scs -\alpha$};
\draw (0.8,2.9) node {$\scs \gamma_{2n-3}$};
\draw (2.5,1.5) node {$\scs \gamma_{n}$};
\draw (4.25,0.9) node {$\scs \gamma_{n-1}$};
\draw (2.75,-0.5) rectangle (3.75,0.5);
\draw (3.25,0) node {$\scs f$};
\begin{scope}[decoration={markings, mark=at position 0.9 with {\arrow{>}}}]
    \draw[ultra thick, black, postaction=decorate] (3.25,-0.5) to (3.25,-1.5);
    \draw[ultra thick, black, postaction=decorate] (0.8125,-3.75) to (0,-4.5);
    \draw[ultra thick, black, postaction=decorate] (1.625,-3) to (0.8125,-3.75);
    \draw[ultra thick, black, dotted] (2.4375,-2.25) to (1.625,-3);
    \draw[ultra thick, black, postaction=decorate] (3.25,-1.5) to (2.4375,-2.25);
\end{scope}
\begin{scope}[decoration={markings, mark=at position 0.3 with {\arrow{>}}}]
    \draw[ultra thick, black, postaction=decorate] (1,-4.5) to (0.8125,-3.75);
    \draw[ultra thick, black, postaction=decorate] (2,-4.5) to (1.625,-3);
    \draw[ultra thick, black, postaction=decorate] (3,-4.5) to (2.4375,-2.25);
    \draw[ultra thick, black, postaction=decorate] (4,-4.5) to (3.25,-1.5);
\end{scope}
\draw (0,-6.2) node {$\scs \psi_{2n-2}$};
\draw (0.65,-5) node {$\scs -\alpha$};
\draw (1.65,-5) node {$\scs -\alpha$};
\draw (2.75,-5) node {$\scs \cdots$};
\draw (4.5,-5) node {$\scs -\alpha$};
\draw (0.8,-2.9) node {$\scs \psi_{2n-3}$};
\draw (2.5,-1.5) node {$\scs \psi_{n}$};
\draw (4.25,-0.9) node {$\scs \psi_{n-1}$};
\draw[ultra thick, black] (4,4.5) to [out=90,in=90] (5,4.5);
\draw[ultra thick, black] (4,-4.5) to [out=-90,in=-90] (5,-4.5);
\draw[ultra thick, black] (5,4.5) to (5,-4.5);
\draw[ultra thick, black] (3,4.5) to [out=90,in=90] (6,4.5);
\draw[ultra thick, black] (3,-4.5) to [out=-90,in=-90] (6,-4.5);
\draw[ultra thick, black] (6,4.5) to (6,-4.5);
\draw[ultra thick, black] (2,4.5) to [out=90,in=90] (7,4.5);
\draw[ultra thick, black] (2,-4.5) to [out=-90,in=-90] (7,-4.5);
\draw[ultra thick, black] (7,4.5) to (7,-4.5);
\draw[ultra thick, black] (1,4.5) to [out=90,in=90] (8,4.5);
\draw[ultra thick, black] (1,-4.5) to [out=-90,in=-90] (8,-4.5);
\draw[ultra thick, black] (8,4.5) to (8,-4.5);
\draw[ultra thick, black] (0,4.5) to (0,6);
\draw[ultra thick, black] (0,-4.5) to (0,-6);
\end{tikzpicture} }~,~ f = \hackcenter{\begin{tikzpicture}[baseline=0, thick, scale=0.5, shift={(0,0)}]
\draw (0.5,.5) rectangle (5.5,-.5);
\draw (3,0) node {$\scs \beta_n$};
\begin{scope}[shift={(1,-5.5)}]
\begin{scope}[decoration={markings, mark=at position 0.6 with {\arrow{>}}}]
    \draw[ultra thick, black, postaction=decorate] (0,4.5) to (0.8125,3.75);
    \draw[ultra thick, black, postaction=decorate] (1,4.5) to (0.8125,3.75);
    \draw[ultra thick, black, postaction=decorate] (0.8125,3.75) to (1.625,3);
    \draw[ultra thick, black, postaction=decorate] (2,4.5) to (1.625,3);
    \draw[ultra thick, black, dotted] (1.625, 3) to (2.4375, 2.25);
    \draw[ultra thick, black, postaction=decorate] (3, 4.5) to (2.4375, 2.25);
    \draw[ultra thick, black, postaction=decorate] (2.4375,2.25) to (3.25,1.5);
    \draw[ultra thick, black, postaction=decorate] (4, 4.5) to (3.25,1.5);
    \draw[ultra thick, black, postaction=decorate] (3.25,1.5) to (3.25,0.5);
\end{scope}
\end{scope}
\begin{scope}[shift={(1,5.5)}]
\begin{scope}[decoration={markings, mark=at position 0.9 with {\arrow{>}}}]
    \draw[ultra thick, black, postaction=decorate] (3.25,-0.5) to (3.25,-1.5);
    \draw[ultra thick, black, postaction=decorate] (0.8125,-3.75) to (0,-4.5);
    \draw[ultra thick, black, postaction=decorate] (0.8125,-3.75) to (1,-4.5);
    \draw[ultra thick, black, postaction=decorate] (1.625,-3) to (0.8125,-3.75);
    \draw[ultra thick, black, postaction=decorate] (1.625,-3) to (2,-4.5);
    \draw[ultra thick, black, dotted] (2.4375,-2.25) to (1.625,-3);
    \draw[ultra thick, black, postaction=decorate] (2.4375,-2.25) to (3,-4.5);
    \draw[ultra thick, black, postaction=decorate] (3.25,-1.5) to (2.4375,-2.25);
    \draw[ultra thick, black, postaction=decorate] (3.25,-1.5) to (4,-4.5);
\end{scope}
\end{scope}
\draw[ultra thick, black] (1,0.5) to (1,1);
\draw[ultra thick, black] (2,0.5) to (2,1);
\draw[ultra thick, black] (3,0.5) to (3,1);
\draw[ultra thick, black] (4,0.5) to (4,1);
\draw[ultra thick, black] (5,0.5) to (5,1);
\draw[ultra thick, black] (1,-0.5) to (1,-1);
\draw[ultra thick, black] (2,-0.5) to (2,-1);
\draw[ultra thick, black] (3,-0.5) to (3,-1);
\draw[ultra thick, black] (4,-0.5) to (4,-1);
\draw[ultra thick, black] (5,-0.5) to (5,-1);
\draw (0.7,0.75) node {$\scs \alpha$};
\draw (1.7,0.75) node {$\scs \alpha$};
\draw (2.7,0.75) node {$\scs \alpha$};
\draw (3.5,0.75) node {$\scs \cdots$};
\draw (4.3,0.75) node {$\scs \alpha$};
\draw (5.3,0.75) node {$\scs \alpha$};
\draw (0.7,-0.75) node {$\scs \alpha$};
\draw (1.7,-0.75) node {$\scs \alpha$};
\draw (2.7,-0.75) node {$\scs \alpha$};
\draw (3.5,-0.75) node {$\scs \cdots$};
\draw (4.3,-0.75) node {$\scs \alpha$};
\draw (5.3,-0.75) node {$\scs \alpha$};
\draw (2.2,-2.5) node {$\scs \psi_1$};
\draw (3.3,-4) node {$\scs \psi_{n-2}$};
\draw (4.25,-5.25) node {$\scs \psi_{n-1}$};
\draw (2.2,2.5) node {$\scs \gamma_1$};
\draw (3.3,4) node {$\scs \gamma_{n-2}$};
\draw (4.25,5.25) node {$\scs \gamma_{n-1}$};
\end{tikzpicture} }
\end{equation}
to apply a mirror reflection (Lemma~\ref{lemma:fusion-channel}) and attain that $\gamma_{i}=\psi_i$ for $i = n-1, n, \dots, 2n-2$.

Next, we isotope the bottom piece to obtain the following identity
\begin{equation}
\begin{split}
\hackcenter{\begin{tikzpicture}[scale=0.8]
\begin{scope}[decoration={markings, mark=at position 0.9 with {\arrow{>}}}]
    \draw[ultra thick, black, postaction=decorate] (6.5,2.5) to (6.5,1.5);
    \draw[ultra thick, black] (0.8125,-3.75) to (0,-4.5);
    \draw[ultra thick, black, postaction=decorate] (0.8125,-3.75) to (1,-4.5);
    \draw[ultra thick, black, postaction=decorate] (1.625,-3) to (0.8125,-3.75);
    \draw[ultra thick, black, postaction=decorate] (1.625,-3) to (2,-4.5);
    \draw[ultra thick, black, dotted] (2.4375,-2.25) to (1.625,-3);
    \draw[ultra thick, black, postaction=decorate] (2.4375,-2.25) to (3,-4.5);
    \draw[ultra thick, black, postaction=decorate] (3.25,-1.5) to (2.4375,-2.25);
    \draw[ultra thick, black, postaction=decorate] (3.25,-1.5) to (4,-4.5);
    \draw[ultra thick, black, postaction=decorate] (4.0625,-0.75) to (3.25,-1.5);
    \draw[ultra thick, black, postaction=decorate] (4.0625,-0.75) to (5,-4.5);
    \draw[ultra thick, black, postaction=decorate] (4.875,0) to (4.0625,-0.75);
    \draw[ultra thick, black, postaction=decorate] (4.875,0) to (6,-4.5);
    \draw[ultra thick, black, dotted] (5.6875,.75) to (4.875,0);
    \draw[ultra thick, black, postaction=decorate] (5.6875,0.75) to (7,-4.5);
    \draw[ultra thick, black, postaction=decorate] (6.5,1.5) to (5.6875,0.75);
    \draw[ultra thick, black, postaction=decorate] (6.5,1.5) to (8,-4.5);
    \draw[ultra thick, black] (1,-4.5) to [out=-90,in=-90] (8,-4.5);
    \draw[ultra thick, black] (2,-4.5) to [out=-90,in=-90] (7,-4.5);
    \draw[ultra thick, black] (3,-4.5) to [out=-90,in=-90] (6,-4.5);
    \draw[ultra thick, black] (4,-4.5) to [out=-90,in=-90] (5,-4.5);
    \draw[ultra thick, black, postaction=decorate] (0,-4.5) to (0,-7);
    \node at (4.5,-4.8) [circle,fill,inner sep=2pt]{};
    \node at (4.5,-5.35) [circle,fill,inner sep=2pt]{};
    \node at (4.5,-5.95) [circle,fill,inner sep=2pt]{};
    \node at (4.5,-6.55) [circle,fill,inner sep=2pt]{};
\end{scope}
    \node at (6.5,2.7) {$\gamma_{2n-2}$};
    \node at (5.7,1.5) {$\gamma_{2n-3}$};
    \node at (4.3,0) {$\gamma_{n}$};
    \node at (3.4,-0.7) {$\gamma_{n-1}$};
    \node at (2.5,-1.4) {$\gamma_{n-2}$};
    \node at (1.2,-3) {$\gamma_{1}$};
    \node at (0.7,-4.8) {$\alpha$};
    \node at (1.7,-4.8) {$\alpha$};
    \node at (2.7,-4.8) {$\alpha$};
    \node at (3.7,-4.8) {$\alpha$};
    \node at (5.2,-4.8) {$-\alpha$};
    \node at (6.3,-4.8) {$-\alpha$};
    \node at (7.4,-4.8) {$-\alpha$};
    \node at (8.4,-4.8) {$-\alpha$};
    \node at (0,-7.2) {$\alpha$};
    \node at (2.3,-3.5) {$\scs \cdots$};
    \node at (6.3,-3.5) {$\scs \cdots$};
\end{tikzpicture} }
&~=~\hackcenter{\begin{tikzpicture}[scale=0.9]
\begin{scope}[decoration={markings, mark=at position 0.6 with {\arrow{>}}}]
  \draw[ultra thick, black, postaction=decorate] (0,0) to (0.6,-0.6);
  \draw[ultra thick, black, postaction=decorate] (0.6,-0.6) to (1.2,-1.2);
  \draw[ultra thick, dotted] (1.2,-1.2) to (1.8,-1.8);
  \draw[ultra thick, black, postaction=decorate] (1.8,-1.8) to (2.4,-2.4);
  \draw[ultra thick, black, postaction=decorate] (.6,-.6) to (1.2,0);
  \draw[ultra thick, black, postaction=decorate] (1.2,-1.2) to (2.4,0);
  \draw[ultra thick, black, postaction=decorate] (1.8,-1.8) to (3.6,0);
  \draw[ultra thick, black, postaction=decorate] (2.4,-2.4) to (4.8,0);
  \draw[ultra thick, black, postaction=decorate] (2.4,-2.4) to (2.4,-3.3);
  \draw[ultra thick, black] (0.6,0.6) to (0,0);
  \draw[ultra thick, black, postaction=decorate] (1.2,1.2) to (0.6,0.6);
  \draw[ultra thick, dotted] (1.8,1.8) to (1.2,1.2);
  \draw[ultra thick, black, postaction=decorate] (2.4,2.4) to (1.8,1.8);
  \draw[ultra thick, black, postaction=decorate] (.6,.6) to (1.2,0);
  \draw[ultra thick, black, postaction=decorate] (1.2,1.2) to (2.4,0);
  \draw[ultra thick, black, postaction=decorate] (1.8,1.8) to (3.6,0);
  \draw[ultra thick, black, postaction=decorate] (2.4,2.4) to (4.8,0);
  \draw[ultra thick, black, postaction=decorate] (2.4,3.3) to (2.4,2.4);
  \node at (1.2,0) [circle,fill,inner sep=2pt]{};
  \node at (2.4,0) [circle,fill,inner sep=2pt]{};
  \node at (3.6,0) [circle,fill,inner sep=2pt]{};
  \node at (4.8,0) [circle,fill,inner sep=2pt]{};
\end{scope}
   \node at (-0.5,0) {$\gamma_{n-1}$};
   \node at (1.2,-0.5) {$\alpha$};
   \node at (2.3,-0.5) {$\alpha$};
   \node at (2.9,0) {$\cdots$};
   \node at (3.5,-0.5) {$\alpha$};
   \node at (4.8,-0.5) {$\alpha$};
   \node at (1.2,0.5) {$-\alpha$};
   \node at (2.3,0.5) {$-\alpha$};
   \node at (3.5,0.5) {$-\alpha$};
   \node at (4.8,0.5) {$-\alpha$};
   \node at (.6,-1.3) {$\gamma_{n-2}$};
   \node at (1.7,-2.4) {$\gamma_1$};
   \node at (2.4,-3.5) {$\alpha$};
   \node at (.6,1.3) {$\gamma_{n}$};
   \node at (1.7,2.5) {$\gamma_{2n-3}$};
   \node at (2.4,3.5) {$\gamma_{2n-2}$};
\end{tikzpicture} }\\
&~=~ \frac{\delta_{\alpha\gamma_{2n-2}} \delta_{\gamma_{1}\gamma_{2n-3}}\cdots \delta_{\gamma_{n-2}\gamma_{n}}}{\md(\alpha)\md(\gamma_1) \cdots \md (\gamma_{n-2})}.
\end{split}
\end{equation}
Notice that the top piece admits an analogous identity as above.

Putting these pieces together and rearranging the finite sum gives us that Equation~\eqref{eq:special-identification} is equal to
\begin{equation} \label{eq:special-alpha}
\begin{split}
\sum\limits_{\substack{\gamma_1 \in 2\alpha + H_r \\ \gamma_2 \in \gamma_1 + \alpha + H_r \\ \vdots \\ \gamma_{n-1} \in \gamma_{n-2} + \alpha + H_r}} 
\md(\gamma_1)^2 \cdots \md(\gamma_{n-1})^2
\hackcenter{\begin{tikzpicture}[baseline=0, thick, scale=0.5, shift={(0,0)}]
\draw (0.5,.5) rectangle (5.5,-.5);
\draw (3,0) node {$\scs \beta_n$};
\begin{scope}[shift={(1,-5.5)}]
\begin{scope}[decoration={markings, mark=at position 0.6 with {\arrow{>}}}]
    \draw[ultra thick, black, postaction=decorate] (0,4.5) to (0.8125,3.75);
    \draw[ultra thick, black, postaction=decorate] (1,4.5) to (0.8125,3.75);
    \draw[ultra thick, black, postaction=decorate] (0.8125,3.75) to (1.625,3);
    \draw[ultra thick, black, postaction=decorate] (2,4.5) to (1.625,3);
    \draw[ultra thick, black, dotted] (1.625, 3) to (2.4375, 2.25);
    \draw[ultra thick, black, postaction=decorate] (3, 4.5) to (2.4375, 2.25);
    \draw[ultra thick, black, postaction=decorate] (2.4375,2.25) to (3.25,1.5);
    \draw[ultra thick, black, postaction=decorate] (4, 4.5) to (3.25,1.5);
    \draw[ultra thick, black, postaction=decorate] (5, 4.5) to (4.0625,0.75);
    \draw[ultra thick, black, postaction=decorate] (3.25,1.5) to (4.0625,0.75);
    \draw[ultra thick, black, postaction=decorate] (6, 4.5) to (4.875,0);
    \draw[ultra thick, black, postaction=decorate] (4.0625,0.75) to (4.875,0);
    \draw[ultra thick, black, postaction=decorate] (7, 4.5) to (5.6875,-0.75);
    \draw[ultra thick, black, dotted] (4.875,0) to (5.6875,-0.75);
    \draw[ultra thick, black, postaction=decorate] (8, 4.5) to (6.5,-1.5);
    \draw[ultra thick, black, postaction=decorate] (5.6875,-0.75) to (6.5,-1.5);
    \draw[ultra thick, black, postaction=decorate] (6.5,-1.5) to (6.5,-2.5);
\end{scope}
\end{scope}
\begin{scope}[shift={(1,5.5)}]
\begin{scope}[decoration={markings, mark=at position 0.9 with {\arrow{>}}}]
    \draw[ultra thick, black, postaction=decorate] (6.5,2.5) to (6.5,1.5);
    \draw[ultra thick, black, postaction=decorate] (0.8125,-3.75) to (0,-4.5);
    \draw[ultra thick, black, postaction=decorate] (0.8125,-3.75) to (1,-4.5);
    \draw[ultra thick, black, postaction=decorate] (1.625,-3) to (0.8125,-3.75);
    \draw[ultra thick, black, postaction=decorate] (1.625,-3) to (2,-4.5);
    \draw[ultra thick, black, dotted] (2.4375,-2.25) to (1.625,-3);
    \draw[ultra thick, black, postaction=decorate] (2.4375,-2.25) to (3,-4.5);
    \draw[ultra thick, black, postaction=decorate] (3.25,-1.5) to (2.4375,-2.25);
    \draw[ultra thick, black, postaction=decorate] (3.25,-1.5) to (4,-4.5);
    \draw[ultra thick, black, postaction=decorate] (4.0625,-0.75) to (3.25,-1.5);
    \draw[ultra thick, black, postaction=decorate] (4.0625,-0.75) to (5,-4.5);
    \draw[ultra thick, black, postaction=decorate] (4.875,0) to (4.0625,-0.75);
    \draw[ultra thick, black, postaction=decorate] (4.875,0) to (6,-4.5);
    \draw[ultra thick, black, dotted] (5.6875,.75) to (4.875,0);
    \draw[ultra thick, black, postaction=decorate] (5.6875,0.75) to (7,-4.5);
    \draw[ultra thick, black, postaction=decorate] (6.5,1.5) to (5.6875,0.75);
    \draw[ultra thick, black, postaction=decorate] (6.5,1.5) to (8,-4.5);
\end{scope}
\end{scope}
\draw[ultra thick, black] (1,0.5) to (1,1);
\draw[ultra thick, black] (2,0.5) to (2,1);
\draw[ultra thick, black] (3,0.5) to (3,1);
\draw[ultra thick, black] (4,0.5) to (4,1);
\draw[ultra thick, black] (5,0.5) to (5,1);
\draw[ultra thick, black] (1,-0.5) to (1,-1);
\draw[ultra thick, black] (2,-0.5) to (2,-1);
\draw[ultra thick, black] (3,-0.5) to (3,-1);
\draw[ultra thick, black] (4,-0.5) to (4,-1);
\draw[ultra thick, black] (5,-0.5) to (5,-1);
\draw[ultra thick, black] (6,1) to (6,-1);
\draw[ultra thick, black] (7,1) to (7,-1);
\draw[ultra thick, black] (8,1) to (8,-1);
\draw[ultra thick, black] (9,1) to (9,-1);
\draw (0.7,0.75) node {$\scs \alpha$};
\draw (1.7,0.75) node {$\scs \alpha$};
\draw (2.7,0.75) node {$\scs \alpha$};
\draw (3.5,0.75) node {$\scs \cdots$};
\draw (4.3,0.75) node {$\scs \alpha$};
\draw (5.3,0.75) node {$\scs \alpha$};
\draw (6.5,0) node {$\scs -\alpha$};
\draw (7.5,0) node {$\scs \cdots$};
\draw (8.5,0) node {$\scs -\alpha$};
\draw (9.5,0) node {$\scs -\alpha$};
\draw (0.7,-0.75) node {$\scs \alpha$};
\draw (1.7,-0.75) node {$\scs \alpha$};
\draw (2.7,-0.75) node {$\scs \alpha$};
\draw (3.5,-0.75) node {$\scs \cdots$};
\draw (4.3,-0.75) node {$\scs \alpha$};
\draw (5.3,-0.75) node {$\scs \alpha$};
\draw (2.2,-2.5) node {$\scs \gamma_1$};
\draw (3.3,-4) node {$\scs \gamma_{n-2}$};
\draw (4.25,-4.75) node {$\scs \gamma_{n-1}$};
\draw (5.2,-5.6) node {$\scs \gamma_{n-2}$};
\draw (6.5,-7) node {$\scs \gamma_{1}$};
\draw (7.5,-8.2) node {$\scs \alpha$};
\draw (2.2,2.5) node {$\scs \gamma_1$};
\draw (3.3,4) node {$\scs \gamma_{n-2}$};
\draw (4.25,4.75) node {$\scs \gamma_{n-1}$};
\draw (5.2,5.6) node {$\scs \gamma_{n-2}$};
\draw (6.5,7) node {$\scs \gamma_{1}$};
\draw (7.5,8.2) node {$\scs \alpha$};
\end{tikzpicture} }\\
~=~
\sum\limits_{\gamma_{n-1}\in n\alpha + H_{(n-1)(r-1)+1}}\frac{\md(\gamma_{n-1})}{\md(\alpha)}\chi_{\rho_{\gamma_{n-1}}}(\beta_n) \id_{V_\alpha}.
\end{split}
\end{equation}
\end{proof}

These two proofs of Theorem~\ref{thm:CGP-formula} reflect the pros and cons between $\alpha$-fusion trees and special $\alpha$-fusion trees. First, observe that $\alpha$-fusion trees appear as a graphical subcomponent of special $\alpha$-fusion trees. While the special $\alpha$-fusion trees lives in a single $\Hom$-space $\mathcal{H}^s := \Hom(V_\alpha, V_\alpha^{\otimes{n}} \otimes V_{-\alpha}^{\otimes(n-1)})$, the collection of special $\alpha$-fusion trees do not form a basis of $\mathcal{H}^s$ nor forms a $\Br_n$-invariant subspace. In contrast, $\alpha$-fusion trees work over a collection of $\Hom$-spaces: $\mathcal{H}_{n,m,\alpha}^r$ for $m=0, \dots, (n-1)(r-1)$ and makes full use of every basis element. However, if $m\neq m'$ with $f \in \mathcal{H}_{n,m,\alpha}^r$ and $g\in \mathcal{H}_{n,m',\alpha}^r$, then $f^\dagger g = 0$ so it is not appropriate to perform measurements between $\alpha$-fusion trees of $\mathcal{H}_{n,m,\alpha}^r$ and $\mathcal{H}_{n,m',\alpha}^r$. An analogy is that the collection of $\alpha$-fusion tree $\Hom$-spaces $\mathcal{H}_{n,m,\alpha}^r$ can be amalgamated into a neighborhood residing in the larger territory $\mathcal{H}^s$ where we can focus on just a single non-degenerate Hermitian form. Consequently, $\mathcal{H}^s$ provides a much more natural setting to derive a formula of $N_r^\alpha(K)$ in the language of Hermitian pairings of fusion trees.

\begin{theorem} \label{thm:special-pairing}
Let $K$ be a knot represented as a closure of an $n$-braid $\beta_n$. Define
\begin{equation*}
\mathcal{Y}:=\sum\limits_{\substack{\gamma_1 \in 2\alpha + H_r \\ \gamma_2 \in \gamma_1 + \alpha + H_r \\ \vdots \\ \gamma_{n-1} \in \gamma_{n-2} + \alpha + H_r \\ \mathbf{p}:=(\alpha,\gamma_1, \dots, \gamma_{n-1})}} \sqrt{\frac{\md(\gamma_{n-1})}{\langle Y^s_{\mathbf{p}}, Y^s_{\mathbf{p}} \rangle}}
\hackcenter{\begin{tikzpicture}[baseline=0, thick, scale=0.5, shift={(0,0)}]
\begin{scope}[shift={(1,-5.5)}]
\begin{scope}[decoration={markings, mark=at position 0.6 with {\arrow{>}}}]
    \draw[ultra thick, black, postaction=decorate] (0,4.5) to (0.8125,3.75);
    \draw[ultra thick, black, postaction=decorate] (1,4.5) to (0.8125,3.75);
    \draw[ultra thick, black, postaction=decorate] (0.8125,3.75) to (1.625,3);
    \draw[ultra thick, black, postaction=decorate] (2,4.5) to (1.625,3);
    \draw[ultra thick, black, dotted] (1.625, 3) to (2.4375, 2.25);
    \draw[ultra thick, black, postaction=decorate] (3, 4.5) to (2.4375, 2.25);
    \draw[ultra thick, black, postaction=decorate] (2.4375,2.25) to (3.25,1.5);
    \draw[ultra thick, black, postaction=decorate] (4, 4.5) to (3.25,1.5);
    \draw[ultra thick, black, postaction=decorate] (5, 4.5) to (4.0625,0.75);
    \draw[ultra thick, black, postaction=decorate] (3.25,1.5) to (4.0625,0.75);
    \draw[ultra thick, black, postaction=decorate] (6, 4.5) to (4.875,0);
    \draw[ultra thick, black, postaction=decorate] (4.0625,0.75) to (4.875,0);
    \draw[ultra thick, black, postaction=decorate] (7, 4.5) to (5.6875,-0.75);
    \draw[ultra thick, black, dotted] (4.875,0) to (5.6875,-0.75);
    \draw[ultra thick, black, postaction=decorate] (8, 4.5) to (6.5,-1.5);
    \draw[ultra thick, black, postaction=decorate] (5.6875,-0.75) to (6.5,-1.5);
    \draw[ultra thick, black, postaction=decorate] (6.5,-1.5) to (6.5,-2.5);
\end{scope}
\end{scope}
\draw (1,-0.8) node {$\scs \alpha$};
\draw (2,-0.8) node {$\scs \alpha$};
\draw (3,-0.8) node {$\scs \alpha$};
\draw (3.5,-1.25) node {$\scs \cdots$};
\draw (4,-0.8) node {$\scs \alpha$};
\draw (5,-0.8) node {$\scs \alpha$};
\draw (6,-0.8) node {$\scs -\alpha$};
\draw (7,-0.8) node {$\scs -\alpha$};
\draw (7.5,-1.25) node {$\scs \cdots$};
\draw (8,-0.8) node {$\scs -\alpha$};
\draw (9,-0.8) node {$\scs -\alpha$};
\draw (2.2,-2.5) node {$\scs \gamma_1$};
\draw (3.3,-4) node {$\scs \gamma_{n-2}$};
\draw (4.25,-4.75) node {$\scs \gamma_{n-1}$};
\draw (5.2,-5.6) node {$\scs \gamma_{n-2}$};
\draw (6.5,-7) node {$\scs \gamma_{1}$};
\draw (7.5,-8.2) node {$\scs \alpha$};
\end{tikzpicture} } . 
\end{equation*}
Then, we have the following formula:
\begin{equation} \label{eq:pairing}
N_r^\alpha(K) = \langle \mathcal{Y}, (\beta_n \cup \mathbb{I}_{n-1})\mathcal{Y} \rangle
\end{equation}
where $\mathbb{I}_{n-1}$ is the identity braid word on $(n-1)$-strands.
\end{theorem}

\begin{proof}
First, observe that the fusion channels
\begin{equation*}
(\alpha, \gamma_1, \dots, \gamma_{n-1}, \dots, \gamma_{1}, \alpha)
\end{equation*}
in the special $\alpha$-fusion tree reflects across the $\gamma_{n-1}$ term. Therefore, despite the $(\beta_n \cup \mathbb{I}_{n-1})$ action, the reflection of the fusion channels informs us that each summand is pairwise orthogonal with respect to the Hermitian pairing by Lemma~\ref{lemma:pairing}. That is,
\begin{equation} \label{eq:orthogonal-pairing} 
\langle \mathcal{Y}, (\beta_n \cup \mathbb{I}_{n-1})\mathcal{Y} \rangle = \sum\limits_{\substack{\gamma_1 \in 2\alpha + H_r \\ \gamma_2 \in \gamma_1 + \alpha + H_r \\ \vdots \\ \gamma_{n-1} \in \gamma_{n-2} + \alpha + H_r \\ \mathbf{p}:=(\alpha,\gamma_1, \dots, \gamma_{n-1})}} \frac{\md(\gamma_{n-1})}{\langle Y^s_{\mathbf{p}}, Y^s_{\mathbf{p}} \rangle} \langle Y^s_{\mathbf{p}}, (\beta_n \cup \mathbb{I}_{n-1})Y^s_{\mathbf{p}} \rangle.
\end{equation}

Given this setup, a quick calculation shows that Equation~\eqref{eq:orthogonal-pairing} agrees with the definition of $N_r^\alpha(K)$ given in Equation~\eqref{eq:definition} where the morphism inside $\langle-\rangle$ is equal to Equation~\eqref{eq:special-alpha}.
\end{proof}

\begin{corollary} \label{cor:special-pairing}
Assume the same conditions as Theorem~\ref{thm:special-pairing}. Then,
\begin{equation}
\operatorname{ADO}_r(K)|_{x=q^{2\alpha - 2}}=\frac{q^{-\frac{\alpha^2 - (r-1)^2}{2} writhe(\beta_n)}}{\md(\alpha)}\langle \mathcal{Y}, (\beta_n \cup \mathbb{I}_{n-1})\mathcal{Y} \rangle.
\end{equation}
In particular, when $r=2$, we have 
\begin{equation}
\Delta_{K}(q^{2\alpha - 2}) = \frac{q^{-\frac{\alpha^2 - 1}{2} writhe(\beta_n)}}{\md(\alpha)}\langle \mathcal{Y}, (\beta_n \cup \mathbb{I}_{n-1})\mathcal{Y} \rangle
\end{equation}
where $\Delta_K(x)=\operatorname{ADO}_2(K)$ is the Alexander polynomial of $K$.
\end{corollary}

\begin{theorem}
The formulas provided in Theorem~\ref{thm:special-pairing} and Corollary~\ref{cor:special-pairing} are independent of the choice of trivalent morphism conventions of $Y^{\alpha,\beta}_{\gamma}$ (see Equation~\eqref{eq:conventions}).
\end{theorem}

\begin{proof}
Consider an arbitrary choice of trivalent morphism convention in the generic part of $\mathcal{D}^\dagger$. That is, for every $\alpha,\beta,\gamma \in \mathbb{R}\setminus \mathbb{Z}$ where $\gamma - \alpha - \beta \in H_r$, fix an arbitrary complex number $z^{\alpha,\beta}_\gamma$ and set $\tilde{Y}^{\alpha,\beta}_{\gamma} := z^{\alpha,\beta}_\gamma Y^{\alpha,\beta}_{\gamma}$. We depict our arbitrarily chosen trivalent morphism convention $\{ \tilde{Y}^{\alpha,\beta}_{\gamma} \}$ in graphical calculus with blue\footnote{For readers with monochrome print of this article: the color-coded version is available online. All diagrams drawn within this proof is colored in blue.} edges to distinguish with our original choice (with black edges) specified in Equation~\eqref{eq:conventions}:
\begin{equation} \label{eq:alt-conventions}
\tilde{Y}_{\gamma}^{\alpha,\beta}:=\hackcenter{\begin{tikzpicture}[ scale=1.1]
\begin{scope}[decoration={markings, mark=at position 0.5 with {\arrow{>}}}]
  \draw[ultra thick, blue, postaction=decorate] (0,0) to (.6,-.6);
  \draw[ultra thick, blue, postaction=decorate] (.6,-.6) to (.6,-1.2);
  \draw[ultra thick, blue, postaction=decorate] (1.2,0) to (.6,-.6);
\end{scope}
   \node at (0,0.2) {$\alpha$};
   \node at (1.2,0.2) {$\beta$};
   \node at (.6,-1.4) {$\gamma$};
\end{tikzpicture} }=
z^{\alpha,\beta}_{\gamma}Y^{\alpha,\beta}_{\gamma}
\in \Hom(V_\gamma, V_\alpha \otimes V_\beta) \cong \mathbb{C}.
\end{equation}

One can build special $\alpha$-fusion trees $\tilde{Y}_{\mathbf{p}}^s$ with respect to this arbitrary choice of convention:

\begin{equation}
\tilde{Y}_{\mathbf{p}}^s:=
\hackcenter{\begin{tikzpicture}[baseline=0, thick, scale=0.5, shift={(0,0)}]
\begin{scope}[shift={(1,-5.5)}]
\begin{scope}[decoration={markings, mark=at position 0.6 with {\arrow{>}}}]
    \draw[ultra thick, blue, postaction=decorate] (0,4.5) to (0.8125,3.75);
    \draw[ultra thick, blue, postaction=decorate] (1,4.5) to (0.8125,3.75);
    \draw[ultra thick, blue, postaction=decorate] (0.8125,3.75) to (1.625,3);
    \draw[ultra thick, blue, postaction=decorate] (2,4.5) to (1.625,3);
    \draw[ultra thick, blue, dotted] (1.625, 3) to (2.4375, 2.25);
    \draw[ultra thick, blue, postaction=decorate] (3, 4.5) to (2.4375, 2.25);
    \draw[ultra thick, blue, postaction=decorate] (2.4375,2.25) to (3.25,1.5);
    \draw[ultra thick, blue, postaction=decorate] (4, 4.5) to (3.25,1.5);
    \draw[ultra thick, blue, postaction=decorate] (5, 4.5) to (4.0625,0.75);
    \draw[ultra thick, blue, postaction=decorate] (3.25,1.5) to (4.0625,0.75);
    \draw[ultra thick, blue, postaction=decorate] (6, 4.5) to (4.875,0);
    \draw[ultra thick, blue, postaction=decorate] (4.0625,0.75) to (4.875,0);
    \draw[ultra thick, blue, postaction=decorate] (7, 4.5) to (5.6875,-0.75);
    \draw[ultra thick, blue, dotted] (4.875,0) to (5.6875,-0.75);
    \draw[ultra thick, blue, postaction=decorate] (8, 4.5) to (6.5,-1.5);
    \draw[ultra thick, blue, postaction=decorate] (5.6875,-0.75) to (6.5,-1.5);
    \draw[ultra thick, blue, postaction=decorate] (6.5,-1.5) to (6.5,-2.5);
\end{scope}
\end{scope}
\draw (1,-0.8) node {$\scs \alpha$};
\draw (2,-0.8) node {$\scs \alpha$};
\draw (3,-0.8) node {$\scs \alpha$};
\draw (3.5,-1.25) node {$\scs \cdots$};
\draw (4,-0.8) node {$\scs \alpha$};
\draw (5,-0.8) node {$\scs \alpha$};
\draw (6,-0.8) node {$\scs -\alpha$};
\draw (7,-0.8) node {$\scs -\alpha$};
\draw (7.5,-1.25) node {$\scs \cdots$};
\draw (8,-0.8) node {$\scs -\alpha$};
\draw (9,-0.8) node {$\scs -\alpha$};
\draw (2.2,-2.5) node {$\scs \gamma_1$};
\draw (3.3,-4) node {$\scs \gamma_{n-2}$};
\draw (4.25,-4.75) node {$\scs \gamma_{n-1}$};
\draw (5.2,-5.6) node {$\scs \gamma_{n-2}$};
\draw (6.5,-7) node {$\scs \gamma_{1}$};
\draw (7.5,-8.2) node {$\scs \alpha$};
\end{tikzpicture} } = z_{\mathbf{p}} Y_{\mathbf{p}}^s \in \Hom(V_{\alpha}, V_{\alpha}^{\otimes n} \otimes V_{-\alpha}^{\otimes (n-1)}) 
\end{equation}
where $z_{\mathbf{p}}:=\prod\limits_{j=0}^{n-2} z^{\gamma_{j},\alpha}_{\gamma_{j+1}}z^{\gamma_{j+1},-\alpha}_{\gamma_j} \in \mathbb{C}$ and $\gamma_0 := \alpha$. Now, observe that 
\begin{equation*}
\langle \tilde{Y}^s_{\mathbf{p}}, \tilde{Y}^s_{\mathbf{p}} \rangle = \langle z_{\mathbf{p}} Y^s_{\mathbf{p}}, z_{\mathbf{p}}  Y^s_{\mathbf{p}} \rangle =|z_{\mathbf{p}}|^2 \langle Y^s_{\mathbf{p}}, Y^s_{\mathbf{p}} \rangle
\end{equation*}
which implies
\begin{equation*}
|z_{\mathbf{p}}|=\frac{\sqrt{\langle \tilde{Y}^s_{\mathbf{p}}, \tilde{Y}^s_{\mathbf{p}} \rangle}}{\sqrt{\langle Y^s_{\mathbf{p}}, Y^s_{\mathbf{p}} \rangle}}
\end{equation*}
so
\begin{equation*}
\frac{\tilde{Y}^s_{\mathbf{p}}}{\sqrt{\langle \tilde{Y}^s_{\mathbf{p}}, \tilde{Y}^s_{\mathbf{p}} \rangle}} = e^{i\theta}\frac{Y^s_{\mathbf{p}}}{\sqrt{\langle Y^s_{\mathbf{p}},  Y^s_{\mathbf{p}} \rangle}}
\end{equation*}
where $\theta=\arg(z_{\mathbf{p}})$. Therefore,
\begin{equation} \label{eq:convention-pairing-identity}
\frac{\langle \tilde{Y}^s_{\mathbf{p}}, (\beta_n \cup \mathbb{I}_{n-1})\tilde{Y}^s_{\mathbf{p}} \rangle}{\langle \tilde{Y}^s_{\mathbf{p}}, \tilde{Y}^s_{\mathbf{p}} \rangle} = \frac{\langle Y^s_{\mathbf{p}}, (\beta_n \cup \mathbb{I}_{n-1})Y^s_{\mathbf{p}} \rangle}{\langle Y^s_{\mathbf{p}}, Y^s_{\mathbf{p}} \rangle} .
\end{equation}

Now, define 
\begin{equation*}
\mathcal{\tilde{Y}}:=\sum\limits_{\substack{\gamma_1 \in 2\alpha + H_r \\ \gamma_2 \in \gamma_1 + \alpha + H_r \\ \vdots \\ \gamma_{n-1} \in \gamma_{n-2} + \alpha + H_r \\ \mathbf{p}:=(\alpha,\gamma_1, \dots, \gamma_{n-1})}} \sqrt{\frac{\md(\gamma_{n-1})}{\langle \tilde{Y}^s_{\mathbf{p}}, \tilde{Y}^s_{\mathbf{p}} \rangle}}  \tilde{Y}^s_{\mathbf{p}}=\sum\limits_{\substack{\gamma_1 \in 2\alpha + H_r \\ \gamma_2 \in \gamma_1 + \alpha + H_r \\ \vdots \\ \gamma_{n-1} \in \gamma_{n-2} + \alpha + H_r \\ \mathbf{p}:=(\alpha,\gamma_1, \dots, \gamma_{n-1})}} \sqrt{\frac{\md(\gamma_{n-1})}{\langle \tilde{Y}^s_{\mathbf{p}}, \tilde{Y}^s_{\mathbf{p}} \rangle}}
\hackcenter
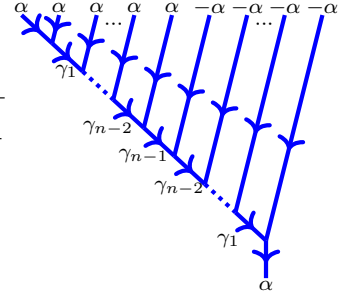
{\begin{tikzpicture}[baseline=0, thick, scale=0.5, shift={(0,0)}]
\begin{scope}[shift={(1,-5.5)}]
\begin{scope}[decoration={markings, mark=at position 0.6 with {\arrow{>}}}]
    \draw[ultra thick, blue, postaction=decorate] (0,4.5) to (0.8125,3.75);
    \draw[ultra thick, blue, postaction=decorate] (1,4.5) to (0.8125,3.75);
    \draw[ultra thick, blue, postaction=decorate] (0.8125,3.75) to (1.625,3);
    \draw[ultra thick, blue, postaction=decorate] (2,4.5) to (1.625,3);
    \draw[ultra thick, blue, dotted] (1.625, 3) to (2.4375, 2.25);
    \draw[ultra thick, blue, postaction=decorate] (3, 4.5) to (2.4375, 2.25);
    \draw[ultra thick, blue, postaction=decorate] (2.4375,2.25) to (3.25,1.5);
    \draw[ultra thick, blue, postaction=decorate] (4, 4.5) to (3.25,1.5);
    \draw[ultra thick, blue, postaction=decorate] (5, 4.5) to (4.0625,0.75);
    \draw[ultra thick, blue, postaction=decorate] (3.25,1.5) to (4.0625,0.75);
    \draw[ultra thick, blue, postaction=decorate] (6, 4.5) to (4.875,0);
    \draw[ultra thick, blue, postaction=decorate] (4.0625,0.75) to (4.875,0);
    \draw[ultra thick, blue, postaction=decorate] (7, 4.5) to (5.6875,-0.75);
    \draw[ultra thick, blue, dotted] (4.875,0) to (5.6875,-0.75);
    \draw[ultra thick, blue, postaction=decorate] (8, 4.5) to (6.5,-1.5);
    \draw[ultra thick, blue, postaction=decorate] (5.6875,-0.75) to (6.5,-1.5);
    \draw[ultra thick, blue, postaction=decorate] (6.5,-1.5) to (6.5,-2.5);
\end{scope}
\end{scope}
\draw (1,-0.8) node {$\scs \alpha$};
\draw (2,-0.8) node {$\scs \alpha$};
\draw (3,-0.8) node {$\scs \alpha$};
\draw (3.5,-1.25) node {$\scs \cdots$};
\draw (4,-0.8) node {$\scs \alpha$};
\draw (5,-0.8) node {$\scs \alpha$};
\draw (6,-0.8) node {$\scs -\alpha$};
\draw (7,-0.8) node {$\scs -\alpha$};
\draw (7.5,-1.25) node {$\scs \cdots$};
\draw (8,-0.8) node {$\scs -\alpha$};
\draw (9,-0.8) node {$\scs -\alpha$};
\draw (2.2,-2.5) node {$\scs \gamma_1$};
\draw (3.3,-4) node {$\scs \gamma_{n-2}$};
\draw (4.25,-4.75) node {$\scs \gamma_{n-1}$};
\draw (5.2,-5.6) node {$\scs \gamma_{n-2}$};
\draw (6.5,-7) node {$\scs \gamma_{1}$};
\draw (7.5,-8.2) node {$\scs \alpha$};
\end{tikzpicture} } . 
\end{equation*}
Combining Equation~\eqref{eq:orthogonal-pairing} from the proof of Theorem~\ref{thm:special-pairing} and Equation~\eqref{eq:convention-pairing-identity} informs us that 
\begin{equation}
N_r^{\alpha}(K)=\langle \mathcal{Y}, (\beta_n \cup \mathbb{I}_{n-1})\mathcal{Y} \rangle = \langle \mathcal{\tilde{Y}}, (\beta_n \cup \mathbb{I}_{n-1})\mathcal{\tilde{Y}} \rangle
\end{equation}
as desired.
\end{proof}

\subsection{Discussion}
The work presented in this article bridges homological representation theory and the mathematics of topological quantum computation. We conclude this article with a short discussion on potential applications of this work.

\subsubsection{The density and unitarity of the $q$-specialized Lawrence representations}
The question of universality requires an analysis on the density of the image of the corresponding projective unitary braid group representation. Unlike the semisimple setting, non-semisimple Hermitian TQFTs equip our $\Hom$-spaces with Hermitian forms \emph{of potentially mixed signature}. Thus, we obtain braid group representations where the image lives in the indefinite unitary group\footnote{These are also commonly referred to as the \emph{pseudo-unitary group}.} $U(p,q)$ where $p+q\geq 1$ are integers depending on the choice of $\alpha$ parameters (see~\cite{neglecton}). If $p=0$ or $q=0$ then $U(p,q)$ is the usual (definite) unitary group $U(p+q)$. Note that the pseudo-unitary group $U(p,q)$ when $p,q\geq 1$ have an entirely different topological structure than the definite unitary group: the former is a noncompact Lie group with no currently known classification of finitely-generated dense subgroups whereas the latter is well-studied. We suspect that the topological interpretations that homological representations offer will help understand how to attack and/or defend against the indefinite unitarity issue. 

Using the notations of this paper, it was shown in~\cite{NSS-Burau} that the (reduced) Burau representation specialized at the fourth root of unity $\mathcal{L}_{n,1,\alpha}^2$ is projectively isomorphic to the braid group representation on $\mathcal{H}_{n,1,\alpha}^2$. Theorem~\ref{thm:fusiontree-homological} generalizes this result to the full family of Lawrence representations specialized at any $2r^{th}$-root of unity. Moreover,~\cite{NSS-Burau} provides a concrete analysis of the density and unitarity of $\alpha$-fusion trees at the fourth root of unity. Furthermore, Stoimenow has done prior work on the density and unitarity of the Burau and Lawrence-Krammer-Bigelow representations, but for parameters different from our choice of specializations~\cite{Stoimenow}. A similar analysis for the full family of Lawrence representations specialized at roots of unity would help us understand how to effectively handle questions pertaining to density and unitarity when we incorporate $V_\alpha$ representations.

\subsubsection{A quantum algorithm for non-semisimple quantum invariants}
One of the many intrinsic mathematical motivations for the pursuit of engineering a topological quantum computer comes from its ability to \emph{experimentally} calculate topological quantum invariants. Recall that there are quantum algorithms, such as the Aharonov-Jones-Landau algorithm~\cite{AJL} (also see~\cite{FKW}), that are capable of efficiently approximating quantum invariants over their classical algorithms.

In this new non-semisimple regime, it is natural to investigate analogous questions in the case of non-semisimple quantum invariants, such as the colored Alexander invariants. Note that the Alexander polynomial of a link, which is a special case of the colored Alexander invariant, can be computed in polynomial time by a Turing machine~\cite{DRZ} (also see~\cite{complexity}). Although it is clear that the category $\mathcal{D}^\dagger$ has the ability to compute non-semisimple invariants from the representation-theoretic point of view, fusion trees provide the natural language to encode quantum information as quantum states in the NSS TQC mathematical architecture developed in~\cite{neglecton, NSSTQC}. Theorem~\ref{thm:special-pairing} and Corollary~\ref{cor:special-pairing} verifies that non-semisimple quantum knot invariants can be formulated in terms of fusion trees and Hermitian forms. If additional techniques are introduced to overcome the non-unitarity issue, it would suggest that developing a quantum circuit model for non-semisimple quantum invariants is viable.

\subsubsection{NSS TQC via surface braid group representations}
TQC at its essence is about quantum information processing using TQFTs. We remind the reader that TQFTs and their corresponding modular categories naturally provide quantum representations of mapping class groups; in principle, one may search for new design protocols for TQC with surface braid group representations. In~\cite{MCG-TQC}, Bloomquist-Wang studied the possible gate sets that can emerge from mapping class group representations in the context of unitary modular fusion categories, i.e. finite semisimple modular categories. 

In this article, we focused on a family of homological representations of the (classical) braid group $\Br_n^{\Sigma}=\mathrm{MCG}(\Sigma_n)$ known as the Lawrence representations. An and Ko studied certain extensions of the Lawrence representations to introduce a new family of homological representations of surface braid groups~\cite{AnKo}. Bellingeri-Godelle-Guaschi later formalized An and Ko's construction in \cite{BGG17}, providing a natural explanation on how these homological representations arise. We would like to know what family of homological representations of surface braid groups correspond to quantum representations coming from non-semisimple TQFTs. Furthermore, if such relationship is explicitly found, then how can we translate this identification in the language of fusion trees and Hermitian pairings for TQC. The ongoing research program of Blanchet-Palmer-Shaukat and De~Renzi-Martel hints that such an answer is plausible. In~\cite{Heisenberg}, they developed a new family of homological representations coming from the compact, connected, oriented surface $\Sigma_{g,1}$ of genus $g\geq 1$ with $1$ boundary component. De~Renzi-Martel~\cite{DRM} established an explicit isomorphism between specific families of homological representations, comparable to the family developed by Blanchet-Palmer-Shaukat, and quantum representations of mapping class groups coming from the non-semisimple TQFT. Their series of works provide insight into how to identify homological representations of surface braid groups for TQC.


\end{document}